\documentclass[11pt]{amsart}
\usepackage[square,numbers]{natbib}
\usepackage[utf8]{inputenc}
\usepackage{amsfonts, amsmath, amssymb, amsthm, bbm, color, enumerate, graphicx, mathtools, tikz, hyperref, relsize}
\usepackage[margin=30mm]{geometry}

\usepackage{paralist}

\newenvironment{inparaenuma}{\begin{inparaenum}[\upshape \bfseries (a) ]}{\end{inparaenum}}

\newcommand{\set}[1]{\left\{#1\right\}}

\newcommand{\sss}{\scriptscriptstyle}

\newcommand{\erdos}{Erd\H{o}s-R\'enyi }

\usepackage{enumerate}

\newenvironment{enumeratea}{\begin{enumerate}[\upshape (a)]}{\end{enumerate}}

\newcommand{\cT}{\mathcal{T}}
\newcommand{\cD}{\mathcal{D}}

\newcommand{\bbR}{\mathbb{R}}

\newcommand{\Ey}{\mathbf{E}}

\def\sss{}

\allowdisplaybreaks

\newtheorem{theorem}{Theorem}[section]
\newtheorem{corollary}[theorem]{Corollary}
\newtheorem{lemma}[theorem]{Lemma}

\theoremstyle{definition}

\newtheorem{remark}[theorem]{Remark}

\theoremstyle{definition}
\newtheorem{defn}{Definition}

\let\plainqed\qedsymbol
\newcommand{\claimqed}{$\lrcorner$}

\hypersetup{
 colorlinks=true,
 linkcolor=blue,          
 citecolor=red,       
 filecolor=red,   
 urlcolor=red, 
 pdftitle={},
 pdfauthor={},
 pdfsubject={},
 pdfkeywords={},
 linktocpage=true
}

\usepackage{fancyhdr}
 
\usepackage[square,numbers]{natbib}

\usepackage{amssymb}
\usepackage{amsmath}
\usepackage{amsthm}

\newcommand{{\LPC}}{\textbf{LPC}}

\renewcommand{\d}{\operatorname{d}}

\newcommand{\K}{\mathcal{K}}
\newcommand{\G}{\mathcal{G}}
\newcommand{\BP}{\operatorname{BP}}

\numberwithin{equation}{section}

\title[Persistent hubs]{Persistence of hubs in growing random networks}
\author{Sayan Banerjee}
\author{Shankar Bhamidi}
\address{Department of Statistics and Operations Research, 304 Hanes Hall, University of North Carolina, Chapel Hill, NC 27599}
\email{sayan@email.unc.edu, bhamidi@email.unc.edu}
\date{\today}							

\begin{document}
\maketitle
\begin{abstract}
We consider models of evolving networks $\set{\G_n:n\geq 0}$ modulated by two parameters: an attachment function $f:\mathbb{N}_0 \to \bbR_+$ and a (possibly random) attachment sequence $\set{m_i:i\geq 1}$. Starting with a single vertex, at each discrete step $i\geq 1$ a new vertex  $v_i$ enters the system with $m_i\geq 1$ edges which it sequentially connects to a pre-existing vertex $v\in \G_{i-1}$ with probability proportional to $f(\mbox{degree}(v))$. We consider the problem of emergence of persistent hubs: existence of a finite (a.s.) time $n^*$ such that for all $n\geq n^*$ the identity of the maximal degree vertex (or in general the $K$ largest degree vertices for $K\geq 1$) does not change. We obtain general conditions on $f$ and $\set{m_i:i\geq 1}$ under which a persistent hub emerges, and also those under which a persistent hub fails to emerge. In the case of lack of persistence, for the specific case of trees ($m_i\equiv 1$ for all $i$), we derive asymptotics for the maximal degree and the index of the maximal degree vertex (time at which the vertex with current maximal degree entered the system) to understand the movement of the maximal degree vertex as the network evolves.  A key role in the analysis is played by an inverse rate weighted martingale constructed from a continuous time embedding of the discrete time model. Asymptotics for this martingale, including concentration inequalities and moderate deviations form the technical foundations for the main results. 
\end{abstract}

\section{Introduction and motivation}
\label{sec:intro}
The availability of network data across a wide variety of fields as well as the impact of networks on our daily lives has stimulated an explosion in models and techniques across many different disciplines in the analysis and exploration of these systems. Motivations for formulating models for such systems range from extracting anomalous substructures within real world network data (e.g. community detection), understanding dynamic processes such as epidemic models on real world networks to understanding microscopic rules that lead to emergent behavior such as connectivity. We refer the interested reader to \citet{albert2002statistical,bollobas2001random, durrett2007random,newman2010networks,van2009random} and the references therein for an initial foray into this vast field.

 One important sub-field is the study of dynamic or temporal networks (see \citet{holme2012temporal,masuda2016guidance} and the references therein): networks that evolve and change over time. We give a precise definition in the next section but to fix ideas, imagine an evolving network sequence $\set{\G_n:n\geq 0}$ initialized at some \emph{seed} graph $\G_0$ at time zero; at each discrete step $i\geq 1$ a new vertex enters the system with $m_i$ edges and tries to connect to pre-existing vertices in $\G_{i-1}$. This connection mechanism typically uses a probabilistic scheme which measures the propensity of existing vertices in the network to attract connections from new vertices. In most models, this propensity is some function $f$ of the current degree of the existing vertex.  In the complex networks community such models were initially formulated to understand the dynamic evolution of networks such as social networks (e.g. Internet at the webpage level, Twitter etc.) and in particular to understand simple rules for the emergence of purported characteristics of modern networks such as heavy tailed degree distributions (see \citet{barabasi1999emergence}). 
 
 In the context of the probability community, one question that has attracted increasing attention falls under the branch of \emph{network archeology}: trying to understand the evolution of a network based off the current structure of the network. For example, based on only observing the current topology of the network, can one reconstruct the original seed that was the genesis of the network, the so-called \emph{Adam problem} (\citet{bubeck-devroye-lugosi,bubeck-mossel})? These questions have implications in a number of fields ranging from systems biology (\citet{navlakha2011network,young2018network}) to the detection of sources of malicious information on social networks (\citet{shah2012rumor}). One approach to answering these questions based from a single slice of the network $\G_n$ at some time $n$ is as follows: \begin{inparaenuma}
 	\item For each vertex $v\in \G_n$ compute a specific measure of centrality $\Psi(v)$. This measure depends purely on the topology of the network as the standing assumption is that we do not have access to vertex labels or other information that could indicate age of vertices in the network. 
	\item Output the ``top'' (could be largest or smallest depending on the centrality measure) $K$ vertices measured according to the above centrality measure.  Write $(v_{1,\Psi}^{\sss(n)}, v_{2,\Psi}^{\sss(n)}, \ldots, v_{K, \Psi}^{\sss(n)})$ for this set of vertices in $\G_n$ (for the time being breaking ties arbitrarily). 
 \end{inparaenuma} 
A wide array of centrality measures have been proposed (see \citet{borgatti2005centrality,freeman1977set,newman2005measure}) ranging from degree centrality, eigen-value centrality, rumor centrality (\citet{shah2011rumors,shah2012rumor}) and centroid centrality specific to trees (\citet{bubeck-devroye-lugosi,jordan1869assemblages}). We describe more details of some of these measures in Section \ref{sec:disc}. 

A natural question then in the context of probabilistic models for evolving networks involves the notion of {\bf persistence}.  

\begin{defn}\label{def:pers}
	Fix $K\geq 1$ and a network centrality measure $\Psi$. For an evolving network sequence $\set{\G_n:n\geq 0}$ say that the sequence is {\bf $(\Psi,K)$ persistent} if $\exists$ $n^* < \infty$ a.s. such that for all $n\geq n^*$ the optimal $K$ vertices $(v_{1,\Psi}^{\sss(n)}, v_{2,\Psi}^{\sss(n)}, \ldots, v_{K, \Psi}^{\sss(n)})$ are uniquely specified and remain the same and further the relative ordering amongst these $K$ optimal vertices remains the same. 
\end{defn}
Thus, for example, if $\Psi$ is the degree of a vertex (referred to as degree centrality) and $K=1$ then an evolving network sequence would be $(\Psi, 1)$ persistent if, almost surely, the identity of the maximal degree vertex fixates after finite time. Persistence of centrality measures both allows one to estimate the initial seed using these measures and validates robustness properties of these measures. 

\subsection{Aim of this paper}
In this paper we consider the case where the network model evolves via the arrival of new nodes that connect to pre-existing vertices with probability proportional to some function $f$ of the existing degree with each new node $i$ arriving with $m_i\geq 1$ many edges. We consider the specific case of $\Psi$ representing the degree of a vertex.  We have two major goals: 
\begin{enumeratea}
	\item Derive necessary and sufficient conditions depending on both the attachment function $f$ and nascent edge sequence $\set{m_i:i\geq 1}$  for the emergence of persistence for degree centrality (referred to as persistent {\bf hubs} in \citet{DM,galashin2013existence}). Note that, if persistent hubs emerge, one can in principle use maximal degree vertices to get estimates of the initial seed of the network. Further our results provide insight into regimes where an attachment sequence $\set{m_i:i\geq 1}$ can tip the system from the persistent regime to the lack of persistence regime. 

	\item In the case of trees ($m_i = 1$ for all $i$), when there is lack of persistence, we derive asymptotics for the maximal degree and the index of the maximal degree vertex (time at which the vertex with current maximal degree entered the system) to understand divergence properties of the identity of the hub vertex as the network evolves over time. 
\end{enumeratea}

\subsection{Organization of the paper}
We start in Section \ref{modelspec} with a precise definition of the model. Section \ref{sec:not} defines the key technical constructs required to state and prove our main results as well as global assumptions made on the attachment function.  We describe our main results in Section \ref{sec:results}. Section \ref{sec:disc} contains a discussion of related work and connections to other areas in probabilistic combinatorics. We start the proofs in Section \ref{ctbp} with continuous time constructions of the network model, and connections to branching processes, which play a key role in the analysis. One important ingredient is a martingale derived from tracking rates of growth of degrees in the continuous time construction via inversely weighting the rates. Section \ref{techlem} derives technical estimates for this process including concentration inequalities, functional central limit theorems and moderate deviation results. These technical estimates are then used to complete the proofs of the main results in Section \ref{sec:main-proofs}.


\section{Model definition}\label{modelspec}
The generalized attachment model with attachment function $f: \mathbb{N}_0 \rightarrow (0, \infty)$ and (possibly random) $\mathbb{N}$-valued attachment sequence $\{m_i\}_{i \ge 1}$ is a model for an evolving random graph sequence $\{\G_n\}_{n \ge 0}$ obtained by the following recipe conditional on a realization of the attachment sequence $\{m_i\}_{i \ge 1}$. $\G_0$ comprises one vertex $v_0$ (the root) and zero edges. Given we have obtained $\G_{n-1}$, $\G_n$ is constructed from $\G_{n-1}$ by adding one vertex $v_n$ and $m_n$ directed edges, each with one end connected to $v_n$ and directed away from $v_n$, and the other ends of these edges are connected sequentially to one of the existing vertices $\{v_i\}_{0 \le i \le n-1}$ in $\mathcal{G}_{n-1}$ according to the following rule. For $1 \le j \le m_n$, the $j$-th edge has its other end connected to $v_i$ ($0 \le i \le n-1$) with probability proportional to $f(\text{degree of }v_i)$ (the degree is computed before the connection is made). 

We will find it convenient to index the graph sequence by the number of attached edges (with both ends already connected to respective vertices). Mathematically, we construct the following (directed) random graph sequence $\{\G^*_k\}_{k \ge 0}$ where $\G^*_k$ has $k$ attached edges. Let $s_n := \sum_{i=1}^{n}m_i$ for $n \ge 1$ and $s_0 = 0$. For any $n \ge 1$ and any $s_{n-1} < k \le s_n$, $\G^*_k$ has $n+1$ vertices. For $l \ge 0$, let $d_0(l)$ denote the degree of the root after the $l$-th edge is attached. For $i \ge 1$, denote by $d_i(l), \ l > s_{i-1}$, the degree of the $(i+1)$-th vertex (i.e. vertex $v_i$) after the $l$-th edge is attached and set $d_i(l) = 0$ for all $l \le s_{i-1}$. Here and for the rest of the paper degree denotes the sum of in and out degrees of a vertex.  For $s_{n-1} < k \le s_n$, an edge $e_k$ is added to the graph with one end attached to $v_n$, directed away from $v_n$, and for $0 \le i \le n-1$,
\begin{equation}\label{attachprob}
\mathbb{P}\left(\text{other end of } e_k \text{ attached to } v_i \mid \G^*_{j}, \ j \le k-1\right) = \frac{f(d_i(k-1))}{\sum_{j=0}^{n-1}f(d_j(k-1))}.
\end{equation}
Thus, $\{\G_n\}_{n \ge 0}$ has the same law as $\{\G^*_{s_n}\}_{n \ge 0}$. Note that $\mathcal{G}_n = \mathcal{G}^*_n, n \ge 0$ in the tree case, i.e., when $m_i=1$ for all $i \ge 1$.

We say that a \emph{persistent hub emerges} in the random graph sequence $\{\G_n\}_{n \ge 0}$ (equivalently $\{\G^*_n\}_{n \ge 0}$) if, almost surely, there exists a random finite $n_0 \in \mathbb{N}_0$ such that the same vertex has maximal degree in $\G_n$ (equivalently $\G^*_n$) for all $n \ge n_0$. For $n \ge 0$, define the maximal degree in $\mathcal{G}^*_n$:
\begin{equation}\label{maxdegdef}
d_{max}(n) := \max_{0 \le k \le n}d_k(n).
\end{equation}
We also define the \emph{index of the maximal degree vertex} in $\mathcal{G}^*_n$ by
\begin{equation}\label{inddef}
\mathcal{I}^*_n := \inf\{0 \le i \le n : d_i(n) \ge d_j(n) \text{ for all } j \le n\}.
\end{equation}
In words, this is the index of the oldest vertex possessing the maximal degree at time $n$. Note that $\mathcal{I}^*_n \stackrel{P}{\longrightarrow} \infty$ immediately implies lack of persistence.  

\section{Global assumptions and key constructs}
\label{sec:not}
Let $f$ denote an attachment function. Extend $f$ to all of $[0, \infty)$ by defining $f(x) := f(\lfloor x \rfloor), x \ge 0$. The resulting function is still denoted by $f$. Throughout this article, we assume that $f$ satisfies,
\begin{equation}\label{unipos}
f_* := \inf_{i \ge 0}f(i) >0,
\end{equation}
\begin{equation}\label{sumfinfty}
\sum_{i=0}^{\infty}\frac{1}{f(i)} = \infty.
\end{equation}
Assumption \eqref{sumfinfty} is a natural as, if it is not satisfied, the qualitative behavior of the model changes. In the tree case, i.e., when $m_i =1$ for all $i \ge 0$, methods of \citet{oliveira2005connectivity} can be used to show that if \eqref{sumfinfty} is not satisfied, then almost surely, after some (random) $n$, all the added vertices connect to the same existing vertex and the maximal degree equals the degree of this vertex almost surely. Consequently, the degrees of all vertices except one are bounded almost surely. We are interested in the case when the degrees of multiple vertices diverge.

Crucial to our proof techniques are appropriate embeddings of the discrete dynamics described in Section \ref{modelspec} into continuous time processes, which are constructed from a collection $\{\Ey_i\}_{i \ge 0}$ of i.i.d. rate one exponential random variables. For $l \ge 0$, define
$$
S_k(l) := \sum_{i=0}^{l-1}\frac{\Ey_i}{f^{k}(i)},  \ \ k \in \mathbb{N},
$$
where $f^k(\cdot)$ is the $k$-th power of the function $f(\cdot)$.
Here and throughout the rest of the article, we stick to the convention `$\sum_{0}^{-1} = 0$'.

Define the point process $\xi$ by 
\begin{equation}
\label{eqn:xi-f-def}
\xi(t) := \sum_{n=1}^{\infty}\mathbb{I}[S_1(n) \le t], \ \ t \ge 0.
\end{equation}
Let $\mu(t) : =\mathbb{E}(\xi(t)), t \ge 0$. We can naturally obtain a non-negative measure on $(\mathbb{R}_+, \mathcal{B}(\mathbb{R}_+))$ from $\mu$ (the intensity measure of $\xi$) which we also denote by $\mu$.
For $A \in \mathbb{N}_0$, denote by $\xi_A(\cdot)$ the point process obtained as above with attachment function $f_A(\cdot) := f(A + \cdot)$ replacing $f(\cdot)$. Note that $\xi_0(\cdot) = \xi(\cdot)$. By convention, for any $A \in \mathbb{N}_0$, we take $\xi_A(t) = 0$ for $t < 0$.

Define the functions
\begin{equation}\label{phidef}
\Phi_k(l) := \sum_{i=0}^{l-1}\frac{1}{f^{k}(i)},  \ \ k =1,2, \ \ l \ge 0.
\end{equation}
Extend $\Phi_k(\cdot), k=1,2$ to $\mathbb{R}_+$ via linear interpolation. Note that using the above extensions of $f, \Phi_1, \Phi_2$, 
$$
\Phi_k(x) = \int_0^{x}\frac{1}{f^k(z)}dz, \ x \ge 0, \ k =1,2.
$$
Define
$$
\mathcal{K}(t) = \Phi_2 \circ \Phi_1^{-1}(t), t \ge 0.
$$
We will use the notation $\Phi_k(\infty) := \lim_{l \rightarrow \infty}\Phi_k(l)$.

As is reflected in the statement of the main results in Section \ref{sec:results}, persistence (and lack of it) and the index asymptotics of the persistent hub are characterized by the functions $\Phi_1(\cdot), \Phi_2(\cdot), \K(\cdot)$. In particular, we will see that, under certain assumptions, a persistent hub emerges if and only if $\Phi_2(\infty) < \infty$.  

Write $N(t) := \Phi_1(\xi(t)), t \ge 0$, and $\bar{f} := f \circ \Phi_1^{-1}$.
For $A \in \mathbb{N}_0$, define the process
\begin{equation}\label{martdef}
M_A(t) := \sum_{i=0}^{\xi_A(t) - 1} \frac{1}{f_A(i)} - t, \ t \ge 0.
\end{equation}
By convention, set $M_A(t) = 0$ for $t \le 0$. Denote $M_0(\cdot)$ by $M(\cdot)$. The processes $\{M_A(\cdot) : A \in \mathbb{N}_0\}$ will be shown to be martingales and their concentration properties, functional central limit theorems and moderate deviations will be studied in detail in Section \ref{techlem} which will then be used to prove the main results.

In the tree case ($m_i = 1$ for all $i \in \mathbb{N}$), we will establish persistence for a more general class of attachment functions and derive fine asymptotics of the index of the hub in case of lack of persistence. The key technical tool will be embedding $\{\mathcal{G}_n : n \in \mathbb{N}_0\}$ in a continuous time branching process as described in Section \ref{ctbp}. Let $\mu(\cdot)$ be the intensity measure of the point process $\xi(\cdot)$, namely, $\mu(\cdot)$ is the measure on Borel subsets $\mathcal{B}(\mathbb{R}_+)$ of the non-negative reals which is uniquely defined by $\mu(B) := \mathbb{E}\left(\xi(B)\right), \, B \in \mathcal{B}(\mathbb{R}_+)$. The `exponential rate of growth' of this branching process can be precisely described in terms of the Laplace transform of the intensity measure $\mu(\cdot)$ given by
\begin{equation}
\label{eqn:rho-hat-def}
	\hat{\rho}(\lambda):= \int_0^{\infty} e^{-\lambda t} \mu(dt) = \sum_{k=1}^\infty \prod_{i=0}^{k-1} \frac{f(i)}{\lambda + f(i)}, \ \ \lambda \in (0, \infty). 
\end{equation}
Let 
$\underline{\lambda}:= \inf\{\lambda > 0: \hat{\rho}(\lambda) < \infty\}$. 
For our results in the tree case, we will assume
\begin{equation}
\label{eqn:prop-under-lamb}
\underline{\lambda} < \infty \ \text{ and } \ \lim_{\lambda\downarrow\underline{\lambda}} \hat{\rho}(\lambda) > 1. 
\end{equation}
See Lemma \ref{checkcond} for checkable conditions on $f$ that imply \eqref{eqn:prop-under-lamb}. Note that \eqref{eqn:prop-under-lamb} implies \eqref{sumfinfty}.
By Assumption \eqref{eqn:prop-under-lamb} and the monotonicity of $\hat{\rho}(\cdot)$, there exists a unique $\lambda_*:=\lambda_*(f)$ such that 
\begin{equation}
\label{eqn:malthus-def}
	\hat{\rho}(\lambda_*) = 1. 
\end{equation}
This number is often referred to as the \emph{Malthusian rate} of growth parameter (which gives the exponential growth rate of the branching process).

\section{Main results}
\label{sec:results}
In this section, we describe the main results of this article. Section \ref{general} lays out general conditions on the attachment function $f$ and the attachment sequence $\{m_i\}_{i \ge 1}$ for the emergence (and non-emergence) of a persistent hub. Section \ref{tree} considers the tree case, namely, when $m_i = 1$ for all $i \ge 1$. Using an embedding recipe into branching processes (described in Section \ref{ctbp}) we obtain criteria for the emergence of a persistent hub for a more general class of attachment functions. In particular, we do not require monotonicity of $f$. In the case where a persistent hub does not emerge, we characterize the asymptotic behavior of the index of the hub as the size of the tree grows.

The following assumptions (or subsets of them) will play a fundamental role in showing lack of persistence and obtaining the asymptotics of the index of the hub. They will be explicitly mentioned in the statement of results that require them.\\

\textbf{Assumption C1: } $\Phi_2(\infty) = \infty$.\\

\textbf{Assumption C2: } $\displaystyle{\lim_{\delta \downarrow 0} \limsup_{t \rightarrow \infty} \frac{\K((1+\delta)t)}{\K(t)} = 1}$.\\

\textbf{Assumption C3: } There exist positive constants $t', D$ such that $\K(3t) \le D\K(t)$ for all $t \ge t'$.


\begin{remark}\label{2classes}
Two broad classes of attachment functions for which Assumptions C1, C2 and C3 are satisfied are the following: 
\begin{itemize}
\item[\textbf{Class I :}] $f = f_r f_b$, where $f_r$ is positive and regularly varying (not necessarily monotone) with index $\alpha \in [0, 1/2)$, and there exist $b_1, b_2>0$ such that $b_1 \le f_b(k) \le b_2$ for all $k \in \mathbb{N}_0$.\\
\item[\textbf{Class II :}] $f = g+h$, where $h$ is non-negative, $g$ is positive and non-decreasing, $\sum_{k=0}^{\infty}\frac{1}{g^2(k)} = \infty$ and $\lim_{k \rightarrow \infty} h(k)/g(k) = 0$.
\end{itemize}
See Appendix \ref{appver} for verification of the Assumptions C1, C2 and C3 for these two classes of functions.
\end{remark}
\subsection{Main Results: general case}\label{general}

The following theorem lays out general conditions on the attachment function and the attachment sequence which ensure the emergence of a persistent hub. Recall that $s_n := \sum_{i=1}^{n}m_i$ for $n \ge 1$ and $s_0 = 0$.
\begin{theorem}\label{genper}
Assume $f$ is non-decreasing and there exists $C_f>0$ such that $f(i) \le C_f (i+1)$ for all $i \ge 0$. Also, suppose $\Phi_2(\infty) < \infty$ and that, almost surely,
\begin{equation}\label{edgegrow}
\limsup_{n \rightarrow \infty} \frac{\Phi_1(m_{n})}{\log s_n} \le \frac{1}{8C_f}.
\end{equation}
Then a persistent hub emerges almost surely in the random graph sequence $\{\G_n\}_{n \ge 0}$. 
\end{theorem}
The proof of Theorem \ref{genper} readily implies the following corollary.
\begin{corollary}\label{kpersist}
Under the assumptions on $f$ in Theorem \ref{genper}, for any $K \in \mathbb{N}$, there exists a random $n_K^* \in \mathbb{N}$ such that the set of vertices with the largest $K$ degrees in $\mathcal{G}_n$ remains the same for all $n \ge n_K^*$. Moreover, the degrees of all these vertices are distinct and their relative ordering of degrees remains the same for all $n \ge n_K^*$.
\end{corollary}

\begin{remark}
Theorem \ref{genper} and Corollary \ref{kpersist} establish a regime where, for any fixed $K \in \mathbb{N}$, the vertices with the $K$ largest degrees persist as the network size grows. This indicates the possibility of developing root detection algorithms in such networks based on a prescribed number of vertices with the largest degrees (the number depending on the error tolerance of the algorithm). Such algorithms will require more quantitative control on the random variables $n_K^*, \, K \in \mathbb{N},$ described in Corollary \ref{kpersist}, like bounds on their expectations, tail probabilities, etc. A closer look at the proof of Theorem \ref{genper} shows that this involves obtaining sharp estimates on probabilities of the events $\{A_m^c\}_{m \ge 1}$, where $A_m$ is defined in \eqref{amdef}. We are investigating these questions in follow-up work.
\end{remark}

The following theorem gives general conditions on the attachment function $f$ under which a persistent hub fails to emerge.

\begin{theorem}\label{perfail}
Suppose Assumptions C1 and C2 hold. For any $A_1, A_2 \in \mathbb{N}_0$, let $\xi_{A_1}(\cdot)$  and $\xi_{A_2}(\cdot)$ be independent point processes obtained using respective attachment functions $f_{A_1}(\cdot)$ and $f_{A_2}(\cdot)$ as defined in Section \ref{sec:not} (on the same probability space), and let $M_{A_1}(\cdot)$ and $M_{A_2}(\cdot)$ be the corresponding martingales obtained by \eqref{martdef}. 
Then the process
$$
B^{(n)}(t) := \frac{1}{\sqrt{2n}}\left(M_{A_1}(\K^{-1}(nt) - M_{A_2}(\K^{-1}(nt)\right), \ t \ge 0, \ n \in \mathbb{N},
$$
converges weakly to standard Brownian motion in $D([0,\infty)  :  \mathbb{R})$ as $n \rightarrow \infty$.
In particular, almost surely, a persistent hub does not emerge in the random graph sequence $\{\G_n\}_{n \ge 0}$.
\end{theorem}
The `in particular' part of the above theorem follows from the fact that the weak convergence to Brownian motion for the process $B^{(n)}$ implies that the values of $\xi_{A_1}(t)$ and $\xi_{A_2}(t)$ agree at infinitely many time points $t$. This, in turn, implies that the degrees of any two vertices are equal to each other in $\G_n$ for infinitely many $n$, which contradicts the emergence of a persistent hub.

The next theorem analyses the case when the attachment sequence is i.i.d. and relates the emergence of a persistent hub to the tail behavior of the distribution of an element in this sequence.
\begin{theorem}\label{iid}
Assume that $f$ is non-decreasing, $\Phi_2(\infty)< \infty$ and $\{m_i\}_{i \ge 1}$ form an i.i.d. sequence supported on $\mathbb{N}$.

(i) Suppose there exists $C_f>0$ such that $f(i) \le C_f (i+1)$ for all $i \ge 0$. If there exist positive constants $D, z_0$ and $\theta >1$ such that 
\begin{equation}\label{light}
\mathbb{P}\left(\Phi_1(m_1) > z\right) \le e^{-Dz^{\theta}}, \ z \ge z_0,
\end{equation}
then almost surely a persistent hub emerges.

(ii) Suppose $\mathbb{E}(m_1) < \infty$. If there exist positive constants $D', z_0'$ and $\theta' \in (0,1)$ such that 
\begin{equation}\label{heavy}
\mathbb{P}\left(\Phi_1(m_1) > z\right) \ge e^{-D'z^{\theta'}}, \ z \ge z_0',
\end{equation}
then, almost surely, there is no persistent hub.
\end{theorem}

Theorem \ref{iid} shows that, when $\{m_n\}_{n \ge 1}$ are i.i.d., the persistence property depends on the tail of the distribution of $\Phi_1(m_1)$ which is intimately connected to the scaling of the maximum of a collection of these random variables. This might seem to imply that, if any sequence of $m_n$'s grows faster than $\Phi_1^{-1}((\log n)^{\theta''})$ for some $\theta''> 1$, then there is no persistent hub. However, perhaps surprisingly, this is not true, as is demonstrated by the following theorem.

\begin{theorem}\label{slowvarm}
Consider the attachment function $f(k) = (k+1)^{\alpha}, \ k \in \mathbb{N}_0$, for some $\alpha \in (1/2,1)$. Let $m_n = \lfloor 1 + (\log n)^{\nu}\rfloor, \ n \in \mathbb{N}$, for some $\nu>0$. Then, for any positive $\nu$, a persistent hub emerges almost surely.
\end{theorem}

\begin{remark}
Under the assumptions of Theorem \ref{slowvarm} (see \eqref{sl1}), there exist $C>0$ and $k_0 \in \mathbb{N}_0$ such that for $k \ge k_0$,
$$
\Phi_1^{-1}(k) \le C k^{\frac{1}{1-\alpha}}.
$$ 
Thus, for any $\theta'>0$, if $\nu > \theta'/(1-\alpha)$, then
$$
\liminf_{n \rightarrow \infty}\frac{m_n}{\Phi_1^{-1}((\log n)^{\theta'})} \ge \liminf_{n \rightarrow \infty}\frac{\lfloor 1 + (\log n)^{\nu}\rfloor}{C(\log n)^{\theta'/(1-\alpha)}} = \infty,
$$
although, by Theorem \ref{slowvarm}, a persistent hub emerges.

Note that the sequence $\{m_n\}_{n \ge 1}$ in Theorem \ref{slowvarm} is slowly varying. As indicated by the proof of Theorem \ref{iid} (ii), persistence is broken by \emph{rare events} when an incoming vertex has an atypically large degree. The slowly varying $\{m_n\}_{n \ge 1}$ ensures that such abrupt fluctuations in incoming vertex degrees do not occur and the maximum degree vertex picks up edges sufficiently fast to beat the degrees of the incoming vertices.
\end{remark}

\subsection{Main Results: tree case}\label{tree}
In the case when $m_i=1$ for all $i \ge 1$, using the continuous time embedding techniques discussed in Section \ref{ctbp}, persistence can be shown for a more general class of attachment functions without monotonicity assumptions on $f$. Moreover, asymptotic results on the maximal degree can be proved. Recall the function $\hat{\rho}(\cdot)$ defined in \eqref{eqn:rho-hat-def}. Also recall the Malthusian rate $\lambda_*$ defined as the unique solution to \eqref{eqn:malthus-def}.

\begin{theorem}\label{pertree}
Let $m_i=1$ for all $i \ge 1$. Assume $f$ satisfies Assumption \eqref{eqn:prop-under-lamb} and $\Phi_2(\infty) < \infty$. Then a persistent hub emerges almost surely. In particular, if $\Phi_2(\infty) < \infty$ and either (i) $\lim_{i \rightarrow \infty} \frac{f(i)}{i} = 0$, or (ii) $\bar{D} := \limsup_{i \rightarrow \infty} \frac{f(i)}{i} \in (0,\infty)$ and $\hat{\rho}(\bar{D}) >1$, then \eqref{eqn:prop-under-lamb} is satisfied and a persistent hub emerges almost surely.

Moreover, under the same assumptions on $f$, the maximal degree exhibits the following asymptotics:
\begin{equation}\label{maxasper}
d_{max}(n) = \Phi_1^{-1}\left(\frac{1}{\lambda_*}\log n + X^*_n\right),
\end{equation}
where $X^*_n$ converges almost surely to some random variable $X^*$ as $n \rightarrow \infty$.
\end{theorem}

\begin{remark}
In the context of \emph{linear preferential attachment} with $f(k) = k+\alpha$ for all $k$ and fixed $\alpha\geq 0$, one can check that $\lambda_* = (2+\alpha)$. Using the form of $\Phi_1^{-1}(\cdot)$ in this specific example, \eqref{maxasper} implies that there exists a finite random variable $W_\infty > 0$ a.s. such that  $d_{max}(n)/n^{1/(2+\alpha)} \stackrel{a.s.}{\longrightarrow} W_\infty$. This result was first proven using completely different techniques (and resulting in an explicit form of the limit random variable as functionals of Beta random variables) in \citet{mori2005maximum} with further characterizations in \citet{pekoz2013degree}. For maximal degree asymptotics with linear preferential attachment but where in addition each vertex has a random fitness see \citet{dereich2017}. 
\end{remark}

\begin{remark}
Under the assumptions on $f$ in Theorem \ref{pertree}, for any $K \in \mathbb{N}$, there is persistence of the top $K$ degree vertices in the sense of Corollary \ref{kpersist}.
\end{remark}

The following theorem gives sharp asymptotics for the maximal degree and the index of the maximal degree vertex in the case $\Phi_2(\infty) = \infty$. 
\begin{theorem}\label{indextree}
Assume $\Phi_2(\infty) = \infty$. Suppose Assumption \eqref{eqn:prop-under-lamb} holds and $f(k) \rightarrow \infty$ as $k \rightarrow \infty$. Also suppose that Assumptions C1, C2 and C3 hold.
Then the index of the maximal degree vertex exhibits the following asymptotics:
\begin{equation}\label{indexas}
\frac{\log \mathcal{I}^*_n}{\K\left(\frac{1}{\lambda_*} \log n\right)} \  \xrightarrow{P} \ \frac{\lambda_*^2}{2}, \ \text{ as } n \rightarrow \infty.
\end{equation}
Moreover, under the same assumptions, the maximal degree satisfies
\begin{equation}\label{maxas}
\frac{\Phi_1(d_{max}(n)) - \frac{1}{\lambda_*}\log n}{\K\left(\frac{1}{\lambda_*} \log n\right)} \xrightarrow{P} \ \frac{\lambda_*}{2}, \ \text{ as } n \rightarrow \infty.
\end{equation}
\end{theorem}

\begin{remark}\label{2classestree}
By Remark \ref{2classes}, the hypotheses of Theorem \ref{indextree} are satisfied if $f$ is in Class I with $\alpha \in (0,1/2)$ ($\lim_{k \rightarrow \infty}f(k) = \infty$ can be verified by applying \citet[Theorem 1.5.6]{bingham}) or $f$ is in Class II with $\lim_{k \rightarrow \infty} g(k) = \infty$.

If $f$ is in Class I with $\alpha \in (0,1/2)$ and $f_b(\cdot) \equiv 1$ (i.e. $f$ is regularly varying of index $\alpha$), we obtain an expression for the index and maximal degree similar to \citet[Remark 1.16]{DM} with an additional multiplicative factor that involves the Malthusian rate. 
Recall that $\K(t) = \Phi_2 \circ \Phi_1^{-1}(t), t \ge 0$. By Karamata's Theorem (Proposition 1.5.8 of \citet{bingham}), $\Phi_1(\cdot)$ is regularly varying with index $1-\alpha$ and $\Phi_2(\cdot)$ is regularly varying with index $1-2\alpha$. Hence, using Theorems 1.5.7 and 1.5.12 of \citet{bingham} we conclude that $\K(\cdot)$ is regularly varying with index $\theta_{\alpha} := \frac{1-2\alpha}{1-\alpha}$. Write $\K(t) = t^{\theta_{\alpha}}l(t), \Phi_1^{-1}(t) = t^{1/(1-\alpha)}\tilde l(t),  t \ge 0$, where $l(\cdot), \tilde l(\cdot)$ are slowly varying functions. Then another application of Karamata's theorem shows that 
$$
l(t) \tilde l(t) \rightarrow \frac{(1-\alpha)^2}{1-2\alpha}, \ \text {as } t \rightarrow \infty.
$$ 
Thus, we obtain the following asymptotics if $f$ is in Class I with $\alpha \in (0,1/2)$ and $f_b(\cdot) \equiv 1$:
\begin{equation}\label{indslow}
\frac{\log\mathcal{I}^*_n}{(\log n)^{\frac{1-2\alpha}{1-\alpha}}l(\log n)} \  \xrightarrow{P} \ \frac{(\lambda^{*})^{\frac{1}{1-\alpha}}}{2}, \ \text{ as } n \rightarrow \infty.
\end{equation}
\begin{equation}\label{maxslow}
\frac{d_{max}(n) - \left(\frac{1}{\lambda_*}\log n\right)^{\frac{1}{1-\alpha}} \tilde l (\log n)}{\log n} \  \xrightarrow{P} \ \frac{1-\alpha}{2(1-2\alpha)}, \ \text{ as } n \rightarrow \infty.
\end{equation}
\end{remark}

Note that in Theorem \ref{indextree}, we impose the assumption $f(k) \rightarrow \infty$ as $k \rightarrow \infty$. In fact, as demonstrated by the next theorem, \eqref{indexas} and \eqref{maxas} are no longer true if this assumption is violated.

\begin{theorem}[Uniform Random Tree]\label{unifindex}
Suppose $f(k) = 1$ for all $k \in \mathbb{N}_0$. Then the following hold:
\begin{equation}\label{uni}
\frac{\log \mathcal{I}^*_n}{\log n} \  \xrightarrow{P} \ 1 - \frac{1}{2\ln 2} \ \text{ as } n \rightarrow \infty,
\end{equation}
\begin{equation}\label{unimax}
\frac{d_{max}(n)}{\log n} \  \xrightarrow{P} \  \frac{1}{\ln 2} \ \text{ as } n \rightarrow \infty.
\end{equation}
\end{theorem}

\begin{remark}\label{unifnottrue}
In the case of $f(\cdot) \equiv 1$, $\lambda_* = 1$ and $\K(t) = t$ for $t \ge 0$. Thus, for $n \ge 1$,
$$
\frac{\lambda_*^2}{2}\K\left(\frac{1}{\lambda_*} \log n\right) = \frac{1}{2}\log n \neq \left(1 - \frac{1}{2\ln 2}\right) \log n.
$$
Hence, \eqref{indexas} does not give the right asymptotics of the hub in this case. Similarly, \eqref{maxas} does not give the right asymptotics for the maximal degree. We note here that \eqref{unimax} (in fact, almost sure convergence) has been previously proved using other techniques in \citet{devroye1995strong} with further refinements in \citet{goh2002limit,addario2018high,eslava2016depth}.
\end{remark}

\section{Discussion}
\label{sec:disc}

\subsection{Related work}
The two papers closest to this work, and that inspired this paper, are \citet{galashin2013existence} and \citet{DM}, which we now describe. In \citet{galashin2013existence}, persistence of degree centrality was proven in the cases of the attachment function $f$ being either $f(k) = k$ for all $k$, or $f$ being convex with $f(k)\to\infty$, and the attachment sequence $m_i \equiv m$ for some constant $m\geq 1$. The heart of the proof is for the $f(k) =k $ case which then allows one to extend the result to the convex unbounded case via a coupling argument. The random walk technique developed in this paper was then used in \citet{jog2016analysis} (described below) to deal with a different centrality measure for the linear preferential attachment and uniform attachment cases. Establishing general conditions for persistence was raised as an open question.

\citet{DM} also analyzed the emergence of persistent hubs and obtained the asymptotics of the hub for a somewhat related model. In fact, the phrase {\bf persistent hub} was introduced in \citet{DM}. In their model, a monotonically increasing attachment function $f : \mathbb{N}_0 \rightarrow (0, \infty)$, satisfying $f(n) \le n+1$ for each $n \ge 0$, is used to construct a growing sequence of directed random graphs as follows. Start with one vertex (the root) and zero edges at time $n=0$. At each time $n \ge 1$, a new vertex $v$ arrives and forms directed connections independently with each existing vertex: $v$ connects to an existing vertex $u$ (connection directed from $v$ to $u$) with probability $f(d^{in}_u(n-1))/(n-1)$, where $d^{in}_u(n-1)$ is the in-degree of $u$ at time $n-1$.

 However, the two models have fundamental differences, most notably, the attachment probability in our model give in \eqref{attachprob} has \textit{a random denominator} unlike that in \citet{DM} where the denominator is the (deterministic) number of vertices of the graph at a given time.
Moreover, the attachment sequence in \cite{DM} is intrinsically prescribed by the model dynamics and their proofs depend crucially on the monotonicity of $f$. In the model considered here, we address much more general attachment sequences and many of our results do not require the monotonicity of $f$. In particular, we give general conditions on the attachment sequence and attachment function for the emergence of a persistent hub (Theorem \ref{genper}). We also investigate i.i.d. attachment sequences and show that persistence, or lack of it, is characterized by the tail behavior of their distribution (Theorem \ref{iid}).
Further, we show that the lack of persistence (Theorem \ref{perfail}) and the index of the hub (in the tree case, see Theorem \ref{indextree}) are characterized by the requirement $\Phi_2(\infty) = \infty$ and certain `continuity' properties of the function $\K(\cdot) := \Phi_2 \circ \Phi_1^{-1}(\cdot)$ (Assumptions C2 and C3 described above), which are satisfied by a wide variety of attachment functions (in particular, the two classes of functions given in Remark \ref{2classes}). In \cite{DM}, in addition to requiring monotonicity of $f$, the index asymptotics for the hub is obtained only when $f$ is regularly varying. Moreover, for results implying the emergence of persistent hubs, we do not require concavity of $f$, an assumption that is used in \cite{DM}. However, despite these differences, the two models have certain technical ingredients in common and they will be pointed out in the course of this article when such technical connections arise.

In the context of uniform attachment (Theorem \ref{unifindex}), the notion of the index of the maximal hub converging to $\infty$ has been a folk theorem in the probabilistic combinatorics community. The closest result to \eqref{uni} in Theorem \ref{unifindex} is in \citet{eslava2016depth} where it is shown that the height $h_n$ of the maximal degree vertex in $\mathcal{G}_n$ for uniform attachment satisfies a central limit theorem  
$(h_n - \hat{u}\log{n})/{\sqrt{\sigma^2 \log{n}}}\stackrel{d}{\longrightarrow} N(0,1),$
with $\sigma^2 =1 - 1/(4\ln 2)$, and $\hat{u} = 1-1/(2\ln 2) $ as in \eqref{uni}.

\subsection{Persistence of other centrality measures}
The only other centrality measure for which persistence had been explored before this work was the so called tree centroid measure in \citet{jog2016analysis,jog2018persistence}, with implications for network archeology in \citet{bubeck-devroye-lugosi,curien2015scaling,bubeck2017trees}. Suppose one has a sequence of growing random trees $\set{\cT_n:n\geq 0}$ started with a single root at time zero (variants of this with more general seed trees have also been explored e.g. in \citet{lugosi2019finding,devroye2018discovery}) and where at each stage a new vertex enters the system and attaches to a single pre-existing vertex using some scheme. For fixed (large) $n$, we only observe the topology of the (unrooted, unlabelled) tree at time $n$ and the aim is to estimate the original root. For any candidate vertex $u\in \cT_n$ let $\mathcal{D}_n(u)$ denote the neighbors of $u$ in $\cT_n$. For any vertex $u$ in $\cT_n$, deleting $u$ and its incident edges would lead to a forest of $|\cD_n(u)|$ many trees. Let $\Psi_n(u)$ denote the size of the largest of these trees. The intuition is that vertices more central within the tree would have \emph{small} $\Psi_n(u)$ values (for example for a newly arrived vertex in $\cT_n$, $\Psi_n(u) = |\cT_n|-1$). Thus the natural algorithm is to output the $K$ vertices with the smallest $\Psi_n(\cdot)$ values. Persistence for this centrality measure for uniform attachment and preferential attachment ($f(k) = k$) were explored in \citet{jog2016analysis}. This was extended to a weak form of persistence for the sublinear preferential attachment ($f(k) = k^\alpha$ with $0<\alpha<1$) in \citet{jog2018persistence}. After this work was completed, we established persistence in the strong sense of Definition \ref{def:pers} for the tree centroid measure for a broad class of attachment functions in \citet{banerjee2020root}, and employed it to obtain root finding algorithms with quantifiable confidence sets for the root.

\subsection{Related questions in probabilistic combinatorics}
The one other major area where notions similar to persistence have been explored has been in the context of the evolution of the connectivity structure in random graphs. Consider the \erdos random graph process: start with $n$ vertices with no edges. At each stage an edge is selected uniformly at random amongst all possible ${n \choose 2}$ edges and added to the system. The objects of interest now are the {\bf sizes} (number of vertices) of connected components. Erd\H{o}s suggested viewing this process as a ``race'' between components in growing their respective sizes. Call the largest connected component at any given time as the {\bf leader}.  One of the questions raised by Erd\H{o}s was to understand asymptotics for the leader time: the time beyond which the identity of the leader does not change, namely the leader persists for all periods of connection after this time. See \citet{luczak1990component,addario2017probabilistic} for more on this problem. 

\section{Continuous embedding in the tree case}\label{ctbp}
We now embark on the proofs of the main results. In the special case when $m_i =1$ for all $i \ge 1$, $\{\G_n\}_{n \ge 0}$ is a growing sequence of random trees in discrete time which can be embedded into a continuous time branching process. We now discuss this embedding. We mainly follow \citet{jagers-ctbp-book,jagers1984growth,nerman1981convergence,rudas2007random}.

Fix any attachment function $f$ satisfying Assumption \eqref{eqn:prop-under-lamb}. 
A \emph{continuous time branching process (CTBP)} driven by $f$, written as $\{\BP(t) : t \ge 0\}$, is defined to be a branching process started with one individual at time $t=0$, called the root, and such that every individual born into the system (including the root) has an offspring distribution that is an independent copy of the point process $\xi$ defined in \eqref{eqn:xi-f-def}. For $t \ge 0$, let $|\BP(t)|$ denote the size (number of individuals) at time $t$. We refer the interested reader to \citet{jagers-ctbp-book,athreya1972} for general theory regarding continuous time branching processes.

An important connection between CTBP and the discrete time random tree sequence $\{\G_n\}_{n \ge 0}$ with $m_i =1$ for all $i \ge 1$ is given by the following result which is easy to check using properties of the exponential distribution (and is the starting point of the Athreya-Karlin embedding \cite{athreya1968}). 

\begin{lemma}\label{lem:ctb-embedding-no-cp}
	Fix an attachment function $f$ satisfying Assumption \eqref{eqn:prop-under-lamb} and consider the sequence of random trees $\{\mathcal{G}_n: n \ge 0\}$ constructed using attachment function $f$. Consider the CTBP driven by $f$ and define for $n \geq 0$ the stopping times $T_n:=\inf\{t\geq 0: |\BP(t)| =n+1\}$. Then, viewed as a sequence of growing random labelled rooted trees, we have $\{\BP(T_n): n \ge 0\} \stackrel{d}{=} \{\mathcal{G}_n : n \ge 0\}.$
\end{lemma}

\section{Technical Estimates}\label{techlem}
This section collects technical estimates which will then be used in the next section to complete the proofs of the main results. Intuitively the main object of interest is the evolution of degrees of vertices in the network which in the continuous time embedding is represented by the point process $\xi(\cdot)$. The martingales $\{M_A(\cdot) : A \in \mathbb{N}_0\}$ defined in \eqref{martdef} turn out to be technically easier to analyze. The aim of this section is to derive tail bounds,  moderate deviation principles and limit theory for these processes. 

  Let $f$ denote an attachment function satisfying the Assumptions \eqref{unipos} and \eqref{sumfinfty}. Let $\{\Ey_i\}_{i \ge 0}$ denote i.i.d. rate one exponential random variables. Recall $\{S_k(l) : k \ge 1, l \ge 1\}$ and the point processes $\xi(\cdot)$ and $\xi_A(\cdot), A \in \mathbb{N}_0,$ defined in Section \ref{sec:not}. Let $\mathcal{L}$ denote the generator of the continuous time Markov process $\xi(\cdot)$. Define the filtration $\mathcal{F}(t) := \sigma\{ \xi(s) : s \le t\}, \ t \ge 0$.
  
The following lemma shows that the processes $\{M_A(\cdot) : A \in \mathbb{N}_0\}$ are indeed martingales and quadratic variations are computed. We only prove the result for $M(\cdot) = M_0(\cdot)$. The same proof applies to any $A \in \mathbb{N}_0$ by replacing the role of $f$ by $f_A$.
\begin{lemma}\label{mgle}
The process $M$ as defined in \eqref{martdef} is a continuous time martingale with respect to the filtration $\{\mathcal{F}(t)\}_{t \ge 0}$ with predictable quadratic variation given by 
$$
\langle M \rangle(t)  = \int_0^t\frac{1}{f\left(\xi(s-)\right)}ds.
$$
\end{lemma}
\begin{proof}
Note that
$$
\mathcal{L} M(t) = \mathcal{L}\Phi_1(\xi(t)) - 1 = f(\xi(t))\left[\Phi_1(\xi(t)+1) - \Phi_1(\xi(t))\right] - 1 = 0
$$
proving that $M$ is a martingale. To obtain the quadratic variation, write $\Delta M(s) := M(s) - M(s-), s \ge 0$ and define the random measure
$$
\Theta(dt, dx) := \sum_{s>0}\mathbb{I}_{[\Delta M(s)>0]} \ \delta_{(s, \Delta M(s))}(dt,dx), \ t \in [0, \infty), x \in [0, \infty),
$$
where $\delta_{(a,b)}(\cdot,\cdot)$ denotes the Dirac measure at point $(a,b)$.
Observe that $\Theta$ is an integer-valued random measure in the sense of \cite[Chapter 3, Section 3]{LS}. Using Proposition 1 in \cite[Chapter 3, Section 4]{LS}, its compensator measure is computed as
$$
\nu(dt, dx) := \sum_{n \ge 0} f(n) \mathbb{I}\left[S_1(n) < t \le S_1(n+1)\right] \delta_{\frac{1}{f(n)}}(dx) dt
$$
where $\delta_a(\cdot)$ denotes the Dirac measure at point $a$. From Problem 2.4 in  \cite[Chapter 3, Section 5]{LS},
\begin{align*}
\langle M \rangle(t) &= \int_0^t\int_{(0,\infty)}x^2 \nu(dt, dx)\\
& = \sum_{n \ge 0} \frac{1}{f(n)} \int_0^t \mathbb{I}\left[S_1(n) < s \le S_1(n+1)\right] ds = \int_0^t\frac{1}{f\left(\xi(s-)\right)}ds.
\end{align*}
\end{proof}
\subsection{Martingale concentration estimates }The next lemma gives tail estimates for $M(t)$ for $t \ge 0$. Recall
$$
\mathcal{K}(t) = \Phi_2 \circ \Phi_1^{-1}(t), t \ge 0.
$$

\begin{lemma}\label{tbd}
For any $x \ge 0, t \ge 0$ and $s \in [0,t]$,
\begin{equation*}
\mathbb{P}\left(M(s) > x \K(t)\right) \le \exp\left\lbrace -\frac{x^2}{2} \frac{\K^2(t)}{f_*^{-2} + \K(t + x\K(t))} \right\rbrace.
\end{equation*}
In particular, if there are positive constants $t', D$ such that $\K(3t) \le D \K(t)$ for all $t \ge t'$, then for any $t \ge \max\left\lbrace t', \K^{-1}\left(\frac{1}{Df_*^2}\right)\right\rbrace$, any $x \in [0, 2t/\K(t)]$ and any $s \in [0,t]$,
\begin{equation*}
\mathbb{P}\left(M(s) > x \K(t)\right) \le \exp\left\lbrace -\frac{x^2}{4D} \K(t)\right\rbrace.
\end{equation*}
\end{lemma}

\begin{proof}
For any $x \ge 0, t \ge 0$ and $s \in [0,t]$,
\begin{multline}\label{mt1}
\mathbb{P}\left(M(s) > x \K(t)\right) = \mathbb{P}\left(N(s) > s + x \K(t)\right) = \mathbb{P}\left(\xi(s) > \Phi_1^{-1}\left(s + x \K(t)\right)\right)\\
\le \mathbb{P}\left(S_1\left(1 + \lfloor \Phi_1^{-1}\left(s + x \K(t)\right) \rfloor \right) \le s \right)
\le e^{\theta s}\mathbb{E}\left(\exp\left\lbrace -\theta S_1(1 + \lfloor \Phi_1^{-1}\left(s + x \K(t)\right\rfloor)\right\rbrace\right).
\end{multline}
Write $\beta(t,s,x) := 1 + \lfloor \Phi_1^{-1}\left(s + x \K(t)\right) \rfloor$. Using the explicit form of the moment generating function for exponential random variables, and the inequality $\log(1+z) \ge z - z^2/2$ for all $z \ge 0$, we obtain for any $\theta > 0$,
\begin{align*}
\log \mathbb{E}\left(\exp\left\lbrace -\theta S_1(\beta(t,s,x))\right\rbrace\right) &= - \sum_{i=0}^{\beta(t,s,x) - 1}\log\left(1 + \frac{\theta}{f(i)}\right)
\le  \sum_{i=0}^{\beta(t,s,x) - 1}\left( - \frac{\theta}{f(i)} + \frac{\theta^2}{2f^2(i)}\right)\\
&= - \theta \Phi_1(\beta(t,s,x)) + \frac{\theta^2}{2} \Phi_2(\beta(t,s,x))\\
& \le -\theta s - x \theta \K(t) + \frac{\theta^2}{2} \Phi_2(\beta(t,s,x))
\end{align*}
where, to obtain the last inequality, we used $\Phi_1(\beta(t,s,x)) \ge \Phi_1(\Phi_1^{-1}(s + x \K(t)) = s + x \K(t)$. Using this estimate in \eqref{mt1},
\begin{equation*}
\mathbb{P}\left(M(s) > x \K(t)\right) \le \exp\left\lbrace - x \theta \K(t) + \frac{\theta^2}{2} \Phi_2(\beta(t,s,x)) \right\rbrace.
\end{equation*}
Optimizing over $\theta>0$, we obtain
$$
\mathbb{P}\left(M(s) > x \K(t)\right) \le \exp\left\lbrace -\frac{x^2}{2} \frac{\K^2(t)}{\Phi_2(\beta(t,s,x))} \right\rbrace.
$$
The first assertion of the lemma now follows upon noting that
$$
\Phi_2(\beta(t,s,x)) \le \Phi_2(1 + \Phi_1^{-1}\left(s + x \K(t)\right)) \le \Phi_2(\Phi_1^{-1}\left(s + x \K(t)\right)) + f_*^{-2} \le \K(t + x\K(t)) + f_*^{-2}.
$$
The second assertion is an immediate consequence of the first.
\end{proof}

For $A \in \mathbb{N}_0$, recall the process $M_A(\cdot)$ defined in \eqref{martdef}. By Lemma \ref{mgle}, $M_A(\cdot)$ is a martingale. Let $f_*(A) := \inf_{i \ge A}f(i)$. Denote $S_k(\cdot), \Phi_k(\cdot)$, $k \ge 1$, (as defined in Section \ref{sec:not}) corresponding to $f_A(\cdot)$ (in place of $f$) by $S^A_k(\cdot), \Phi^A_k(\cdot)$. The following lemma establishes tail estimates for the random variables $\inf_{s<\infty}M_A(s)$ and $\sup_{s<\infty}M_A(s)$ ($A \in \mathbb{N}_0$) in the case $\Phi_2(\infty) < \infty$.  
\begin{lemma}\label{conc}
If $\Phi_2(\infty) < \infty$, then there exist positive constants $x_0, x_0', x_0'', C_1, C_2$ (independent of $A$) such that for any $A \in \mathbb{N}_0$,
\begin{equation}\label{concfin}
\mathbb{P}\left(\inf_{s<\infty} M_A(s) \le -x\right) \le
\left\{
	\begin{array}{ll}
		 C_1\exp\left\lbrace-C_2 x^2\right\rbrace  & \mbox{if }  x_0 \le x  < x_0'\sqrt{f_*(A)},\\
		C_1\exp\left\lbrace -C_2\sqrt{f_*(A)} x\right\rbrace & \mbox{if } x \ge x_0'\sqrt{f_*(A)}.
	\end{array}
\right.
\end{equation}
and
\begin{equation}\label{concfin2}
\mathbb{P}\left(\sup_{s<\infty} M_A(s) \ge x\right) \le C_1\exp\left\lbrace-C_2 x^2\right\rbrace  \ \text{ for all } x \ge x_0''.
\end{equation}
\end{lemma}
\begin{proof}
Fix $A \in \mathbb{N}_0$. In this proof, $C, C'$ will denote generic  positive constants not depending on $x,t, \alpha, A$ whose values might change from line to line. As the martingale $M_A$ has jumps bounded by $f_*(A)^{-1}$, by Theorem 5 in \cite[Chapter 4, Section 13]{LS}, for any $t \ge 0$, $x > 0$ and $L>0$,
\begin{equation}\label{concin}
\mathbb{P}\left(\inf_{s \le t} M_A(s) \le -x \right) \le \exp\left\lbrace - \sup_{z >0}[z x - Lu^A(z)]\right\rbrace + \mathbb{P}\left(\langle M_A \rangle(t) \ge L\right),
\end{equation}
where $u^A(z) := f_*(A)^2\left(e^{z/f_*(A)} - 1 - \frac{z}{f_*(A)}\right), \ z \ge 0$. Observe that 
\begin{equation}\label{conc1}
\langle M_A \rangle(t) = \sum_{n \ge 0}\frac{1}{f_A(n)} \int_0^t \mathbb{I}\left[S^A_1(n) < s \le S^A_1(n+1)\right] ds \le \sum_{n = 0}^{\xi_A(t)}\frac{\Ey_n}{f_A^2(n)} = S^A_2(\xi_A(t) + 1).
\end{equation}
Take any $\theta \in (0, 3f_*(A)^2/4)$. Using the explicit form of the moment generating function for exponential random variables, for any $l \in \mathbb{N}$,
\begin{multline}\label{conc2}
\log \mathbb{E}\left[\exp\left\lbrace \theta (S^A_2(l) - \Phi^A_2(l))\right\rbrace\right] = -\theta \Phi^A_2(l) - \sum_{i=0}^{l-1}\log\left(1 - \frac{\theta}{f_A^2(i)}\right)\\
 = \sum_{k=2}^{\infty} \frac{\theta^k}{k} \Phi^A_{2k}(l) \le \Phi^A_4(l)\sum_{k=2}^{\infty} \frac{\theta^k}{kf_*(A)^{2k-4}} \le C\theta^2\Phi^A_4(l).
\end{multline}
Thus, for any $\alpha>0$, any $\theta \in (0, 3f_*(A)^2/4)$ and any $l \in \mathbb{N}$,
$$
\mathbb{P}\left(S^A_2(l) - \Phi^A_2(l) \ge \alpha\right) \le e^{-\theta \alpha} \mathbb{E}\left[\exp\left\lbrace \theta (S^A_2(l) - \Phi^A_2(l))\right\rbrace\right] \le \exp\left\lbrace -\theta \alpha + C \theta^2\Phi^A_4(l)\right\rbrace.
$$
Note that as $\Phi_2(\infty) < \infty$, therefore $\Phi^A_4(\infty) \le f_*(A)^{-2}\Phi_2(\infty) < \infty$. Hence we obtain, for any $t \ge 0$, $l \in \mathbb{N}$,
\begin{align*}
\mathbb{P}\left(S^A_2(l) \ge \Phi^A_2(\infty) + \alpha\right) &\le \mathbb{P}\left(S^A_2(l) \ge \Phi^A_2(l) + \alpha\right)\\
& \le \exp\left\lbrace -\theta \alpha + C \theta^2\Phi^A_4(l)\right\rbrace \le \exp\left\lbrace C \theta^2\Phi^A_4(\infty)\right\rbrace \exp\left\lbrace -\theta \alpha\right\rbrace.
\end{align*}
Taking a limit as $l \rightarrow \infty$ in the above, we obtain for any $\theta \in (0, 3f_*(A)^2/4)$,
\begin{multline*}
\mathbb{P}\left(S^A_2(\infty) \ge \Phi^A_2(\infty) + \alpha\right) \le \exp\left\lbrace C \theta^2\Phi^A_4(\infty)\right\rbrace \exp\left\lbrace -\theta \alpha\right\rbrace\\
\le \exp\left\lbrace C \theta^2f_*(A)^{-2}\Phi_2(\infty)\right\rbrace \exp\left\lbrace -\theta \alpha\right\rbrace.
\end{multline*}
Taking $\theta = \frac{f_*}{2} f_*(A) \le \frac{f_*(A)^2}{2}$ in the above, we obtain for any $\alpha>0$,
$$
\mathbb{P}\left(S^A_2(\infty) \ge \Phi^A_2(\infty) + \alpha\right) \le C \exp\left\lbrace-C'f_*(A) \alpha \right\rbrace.
$$
From this observation and \eqref{conc1}, for any $t \ge 0$,
\begin{align}\label{conc4}
\mathbb{P}\left(\langle M_A \rangle(t) \ge\Phi^A_2(\infty) + \alpha\right) &\le \mathbb{P}\left(S^A_2(\xi_A(t) + 1)  \ge \Phi^A_2(\infty) + \alpha\right)\\
& \le \mathbb{P}\left(S^A_2(\infty) \ge \Phi^A_2(\infty) + \alpha\right) \le C \exp\left\lbrace-C'f_*(A) \alpha \right\rbrace.\nonumber
\end{align}
Let $L(\alpha) :=  \Phi_2(\infty) + \alpha$ for $\alpha > 0$. Using \eqref{conc4} in \eqref{concin} with $L = L(\alpha)$, and noting $\Phi_2^A(\infty) \le \Phi_2(\infty)$, we get for $x > 0$, $\alpha >0$,
\begin{align*}
\mathbb{P}\left(\inf_{s \le t} M_A(s) \le -x\right) &\le \exp\left\lbrace -\sup_{z >0}[z x - L(\alpha) u^A(z)]\right\rbrace + \mathbb{P}\left(\langle M_A \rangle(t) \ge L(\alpha)\right)\nonumber\\
&\le \exp\left\lbrace -u^*(x,\alpha,A)\right\rbrace + C \exp\left\lbrace-C'f_*(A) \alpha \right\rbrace
\end{align*}
where $u^*(x,\alpha,A) = xf_*(A)\left[\log \left(1 + \frac{x}{L(\alpha) f_*(A)}\right) - 1\right] + L(\alpha) f_*(A)^2\log \left(1 + \frac{x}{L(\alpha) f_*(A)}\right)$.
Taking $t \rightarrow \infty$ in the above expression, we obtain
\begin{align}\label{conc5}
\mathbb{P}\left(\inf_{s<\infty} M_A(s) \le -x\right) \le \exp\left\lbrace -u^*(x,\alpha,A)\right\rbrace + C \exp\left\lbrace-C'f_*(A) \alpha \right\rbrace.
\end{align}
Using the inequality $\log(1+z) \ge z - z^2/2$ for $z \ge 0$,
$$
u^*(x,\alpha,A) \ge \frac{x^2}{2L(\alpha)} - \frac{x^3}{2L(\alpha)^2(f_*(A))}.
$$
If $x \ge \sqrt{(f_*)(f_*(A))} \Phi_2(\infty)$, taking $\alpha = \frac{2x}{\sqrt{(f_*)(f_*(A))}}$ in \eqref{conc5},  we obtain,
$$
\mathbb{P}\left(\inf_{s<\infty} M_A(s) \le -x\right) \le \exp\left\lbrace -\frac{x\sqrt{(f_*)(f_*(A))}}{24}\right\rbrace + C \exp\left\lbrace-2C'x\frac{\sqrt{f_*(A)}}{\sqrt{f_*}}  \right\rbrace.
$$
On the other hand, if $2f_*\left(\Phi_2(\infty) + \Phi_2^2(\infty)\right) \le x < \sqrt{(f_*)(f_*(A))} \Phi_2(\infty)$, taking $\alpha = \frac{x^2}{(f_*)(f_*(A))}$ in \eqref{conc5},
$$
\mathbb{P}\left(\inf_{s<\infty} M_A(s) \le -x\right) \le \exp\left\lbrace -\frac{x^2}{4\left(\Phi_2(\infty) + \Phi_2^2(\infty)\right)}\right\rbrace + C \exp\left\lbrace-\frac{C'x^2}{f_*} \right\rbrace.
$$
This proves \eqref{concfin}.

To prove \eqref{concfin2}, fix any $A \in \mathbb{N}_0$. Denote the jump times of $M_A$ by $\{\tau_i\}_{i \ge 1}$. By the strong Markov property and the same argument used to deduce \eqref{concfin}, there exist positive constants $C,C', x_0$, not depending on $A$, such that for any $i \ge 1$, for all $x \ge x_0$,
\begin{equation}\label{conc21}
\mathbb{P}\left(\inf_{\tau_i \le s< \infty} (M_A(s) - M_A(\tau_i)) \le -x \mid \mathcal{F}(\tau_i)\right) \le C\exp\left\lbrace -C' x \right\rbrace
\end{equation}
where $\mathcal{F}(\tau_i)$ is the stopped filtration corresponding to stopping time $\tau_i$. 
By \eqref{conc21}, we can choose $x_1 \ge x_0$, not depending on $A$, such that for any $i \ge 1$ and any $x \ge x_1$,
\begin{equation}\label{conc22}
\mathbb{P}\left(\inf_{\tau_i \le s< \infty} (M_A(s) - M_A(\tau_i)) > -x \mid \mathcal{F}(\tau_i)\right) \ge \frac{1}{2}.
\end{equation}
For $x \ge 0$, let $N^x := \inf\{i \ge 1: M_A(\tau_i) \ge 2x\}$ with the convention that $N^x = \infty$ if $\sup_{s < \infty} M_A(s) < 2 x$. Note that for any $t \ge 0$,
$$
\mathbb{P}\left(\sup_{s \le t} M_A(s) \ge 2x\right) = \mathbb{P}\left(N^x < \infty, \tau_{N^x} \le t\right).
$$
For any $t \ge 0$ and any $x \ge x_1$,
\begin{align}\label{conc23}
\mathbb{P}\left(M_A(t) >  x\right) &\ge \sum_{j=1}^{\infty}\mathbb{P}\left(N^x = j, \tau_j \le t\right)\mathbb{P}\left(\inf_{\tau_j \le s< \infty} (M_A(s) - M_A(\tau_j)) > -x \mid \mathcal{F}(\tau_j)\right)\\
&\ge \frac{1}{2}\sum_{j=1}^{\infty}\mathbb{P}\left(N^x = j, \tau_j \le t\right) = \frac{1}{2}\mathbb{P}\left(N^x < \infty, \tau_{N^x} \le t\right)\nonumber\\
& = \frac{1}{2}\mathbb{P}\left(\sup_{s \le t} M_A(s) \ge 2x\right),\nonumber
\end{align}
where the second inequality follows from \eqref{conc22}. Define $\K_A(t) := \Phi^A_2 \circ (\Phi^A_1)^{-1}(t), t \ge 0$. Note that, as $\K_A(\infty) = \Phi_2^A(\infty) < \infty$, for any $t \ge \K_A^{-1}\left(\frac{\K_A(\infty)}{2}\right)$,
$$
\K_A(3t) \le \K_A(\infty) \le 2 \K_A(t).
$$
Thus, by the second assertion of Lemma \ref{tbd} with $t' = \K_A^{-1}\left(\frac{\K_A(\infty)}{2}\right)$ and $D=2$,  for any $t \ge \max\left\lbrace\K_A^{-1}\left(\frac{\K_A(\infty)}{2}\right), \K_A^{-1}\left(\frac{1}{2f_*(A)^2}\right)\right \rbrace$ and any $x \in [0, 2t]$,
\begin{align}\label{conc24}
\mathbb{P}\left(M_A(t) >  x\right) &= \mathbb{P}\left(M_A(t) > \frac{x}{\K_A(t)} \K_A(t)\right) \le \exp\left\lbrace -\frac{x^2}{8\K_A^2(t)} \K_A(t)\right\rbrace\\
& = \exp\left\lbrace -\frac{x^2}{8\K_A(t)}\right\rbrace \le \exp\left\lbrace -\frac{x^2}{8\K(\infty)}\right\rbrace,\notag
\end{align}
where we used 
$$
\K_A(t) \le \K_A(\infty) = \int_0^{\infty}\frac{1}{f_A^2(z)}dz = \int_A^{\infty}\frac{1}{f^2(z)}dz \le \Phi_2(\infty) = \K(\infty)
$$
for the last inequality.
Thus, using \eqref{conc23} and \eqref{conc24}, we obtain for any $x \ge x_1$ and any $t \ge \max\left\lbrace\K_A^{-1}\left(\frac{\K_A(\infty)}{2}\right), \K_A^{-1}\left(\frac{1}{2f_*(A)^2}\right), \frac{x}{2}\right \rbrace$,
$$
\mathbb{P}\left(\sup_{s \le t} M_A(s) \ge 2x\right) \le 2 \exp\left\lbrace -\frac{x^2}{8\K(\infty)}\right\rbrace.
$$
The second assertion of the lemma now follows by taking $t \rightarrow \infty$ in the above bound.
\end{proof}
\subsection{Absolute continuity of martingale limit } Recall the process $M=M_0$ defined in \eqref{martdef}. By Lemma \ref{conc}, if $\Phi_2(\infty) < \infty$, $M$ is a uniformly integrable martingale and hence converges almost surely to a limiting random variable $M(\infty)$. The following lemma adapts the strategy of Proposition 2.2 in \cite{DM} to our setup to prove absolute continuity of the law of $M(\infty)$. In contrast to \cite{DM}, which requires concavity of $f$ in addition to $\Phi_2(\infty)< \infty$, we do not need this concavity assumption. A main reason for this is that the model in \cite{DM} is different and the associated martingales there are defined on a discrete time set. This `lattice effect' is overcome by requiring a concavity assumption on $f$ to use appropriate linear interpolations of $f$. In contrast, as our process $M$ is in continuous time, we can obtain differential equations \eqref{ac2} and \eqref{ac0} which makes the analysis more robust in $f$.

\begin{lemma}\label{ac}
Assume $\Phi_2(\infty)< \infty$. Then $M(t) \rightarrow M(\infty)$ almost surely where the law of $M(\infty)$ is absolutely continuous with respect to the Lebesgue measure.
\end{lemma}

\begin{proof}
We prove the absolute continuity of the law of $M(\infty)$. Define $\mathbb{S} := \{\Phi_1(n) : n \in \mathbb{N}_0\}$. Recall $N(t) = \Phi_1(\xi(t)), \, t \ge 0$, where $\xi$ is the point process defined in \eqref{eqn:xi-f-def}. Also recall $M(t) = N(t)-t, \, t \ge 0$. For any $v \in \mathbb{S} \setminus \{0\}$, denote by $v^-$ the left neighbor of $v$ in $\mathbb{S}$. For $v \in \mathbb{S}$, let $p_v(t) := \mathbb{P}(N(t) = v)$. For $t \ge 0$ and small $\Delta t$, note that
\begin{multline}\label{ac1}
\left|\mathbb{P}(N(t+ \Delta t) = v) - \mathbb{P}(N(t) = v, \xi(t + \Delta t) - \xi(t) = 0)\right.\\
\left. - \mathbb{P}(N(t) = v^-, \xi(t + \Delta t) - \xi(t) = 1)\right| \le \mathbb{P}(N(t) \le v, \xi(t + \Delta t) - \xi(t) \ge 2).
\end{multline}
Recall $\bar{f} := f \circ \Phi_1^{-1}$. Let $\mu_v := \max\{\bar{f}(u): u \in \mathbb{S}, \ u \le v\}$. Since $N(t) \le v$ implies $\xi(t) \le \Phi_1^{-1}(v)$, we get
$$
\mathbb{P}(N(t) \le v, \xi(t + \Delta t) - \xi(t) \ge 2) \le \mathbb{P}\left(\operatorname{Poi}(\mu_v \Delta t) \ge 2\right) \le C(\mu_v\Delta t)^2,
$$
where a positive $C$ is a constant that does not depend on $v, \mu, t, \Delta t$. Moreover, 
$$
\mathbb{P}(N(t) = v, \xi(t + \Delta t) - \xi(t) = 0) = (1 - \bar{f}(v) \Delta t)\mathbb{P}(N(t) = v) + O((\Delta t)^2)
$$
and 
$$
\mathbb{P}(N(t) = v^-, \xi(t + \Delta t) - \xi(t) = 1) = (\bar{f}(v^-) \Delta t)\mathbb{P}(N(t) = v^-) + O((\Delta t)^2).
$$
Using these observations in \eqref{ac1}, we obtain
$$
\mathbb{P}(N(t+ \Delta t) = v) = (1 - \bar{f}(v) \Delta t)\mathbb{P}(N(t) = v) + (\bar{f}(v^-) \Delta t)\mathbb{P}(N(t) = v^-) + O((\Delta t)^2)
$$
and hence,
$$
p_v(t + \Delta t) - p_v(t) = (\bar{f}(v^-) \Delta t)\mathbb{P}(N(t) = v^-) - (\bar{f}(v) \Delta t)\mathbb{P}(N(t) = v) + O((\Delta t)^2).
$$
Dividing by $\Delta t$ and taking a limit as $\Delta t \rightarrow 0$, we obtain
\begin{equation}\label{ac2}
p_v'(t) = \bar{f}(v^-) \mathbb{P}(N(t) = v^-) - \bar{f}(v) \mathbb{P}(N(t) = v), \ \ v \in \mathbb{S}\setminus \{0\}, \ t \ge 0.
\end{equation}
Similarly,
\begin{equation}\label{ac0}
p_0'(t) = -\bar{f}(0) p_0(t), \ \ t \ge 0.
\end{equation}
We claim the following:
\begin{equation}\label{ac3}
p_v(t) \le \frac{\bar{f}(0)}{\bar{f}(v)}, \ \text{ for all } t \ge 0, \ v \in \mathbb{S}.
\end{equation}
We prove this by induction on $v \in \mathbb{S}$. For $v = 0$, from \eqref{ac0}, for $t \ge 0$,
$$
p_0(t) = p_0(0) e^{-\bar{f}(0)t} = e^{-\bar{f}(0)t} \le 1 = \frac{\bar{f}(0)}{\bar{f}(0)}
$$
proving \eqref{ac3} for $v = 0$. Suppose the \eqref{ac3} is true for $v^- \in \mathbb{S}$. Then, from \eqref{ac2}, for $t \ge 0$,
\begin{align*}
p_v'(t) &= \bar{f}(v^-) \mathbb{P}(N(t) = v^-) - \bar{f}(v) \mathbb{P}(N(t) = v)\\
& \le \bar{f}(0) - \bar{f}(v) \mathbb{P}(N(t) = v) = \bar{f}(0) - \bar{f}(v) p_v(t).
\end{align*}
From this differential inequality, and using $p_v(0) = 0$ for $v \neq 0$, we obtain for $t \ge 0$,
$$
p_v(t)  \le \bar{f}(0)e^{-\bar{f}(v)t}\int_0^te^{\bar{f}(v)s}ds = \frac{\bar{f}(0)(1 - e^{-\bar{f}(v)t})}{\bar{f}(v)} \le \frac{\bar{f}(0)}{\bar{f}(v)}.
$$
Thus, by induction, \eqref{ac3} is true. For any $-\infty < a < b < \infty$, consider the interval $I = (a,b)$. For any $t \ge 0$, let $m_t := \inf\{m \in \mathbb{N}_0 : \Phi_1(m) > a + t\}$ and $n_t := \sup\{m \in \mathbb{N}_0 : \Phi_1(m) < b + t\}$. Then
\begin{align*}
\mathbb{P}(M(t) \in I) &= \mathbb{P}(N(t) \in (a+t, b+t)) = \sum_{v \in (a+t, b+t) \cap \mathbb{S}} p_v(t)  \le \sum_{v \in (a+t, b+t) \cap \mathbb{S}} \frac{\bar{f}(0)}{\bar{f}(v)}\\
&= \bar{f}(0)\sum_{j=m_t}^{n_t}\frac{1}{\bar{f}(\Phi_1(j))} = \bar{f}(0)\sum_{j=m_t}^{n_t}\frac{1}{f(j)} = \bar{f}(0)\left(\Phi_1(n_t + 1) - \Phi_1(m_t) \right)\\
& \le \bar{f}(0)(b-a) + \bar{f}(0)\epsilon_t,
\end{align*}
where $\epsilon_t := \Phi_1(n_t+1) - \Phi_1(n_t) = \frac{1}{f(n_t)}$. As $\Phi_2(\infty) < \infty$, $f(n) \rightarrow \infty$ as $n \rightarrow \infty$, and hence, $\epsilon_t \rightarrow 0$ as $t \rightarrow \infty$. Hence, as $M(t)$ converges almost surely to $M(\infty)$, by Portmanteau theorem,
$$
\mathbb{P}(M(\infty) \in I) \le \liminf_{t \rightarrow \infty}\mathbb{P}(M(t) \in I)  \le \bar{f}(0) (b-a).
$$
Denoting by $\mathcal{G}$ the field of right semi-closed intervals, the above bound implies $\mathbb{P}(M(\infty) \in A) \le \bar{f}(0) \operatorname{Leb}(A)$ for all $A \in \mathcal{G}$, where $\operatorname{Leb}$ is the Lebesgue measure. By the monotone class theorem, $\mathbb{P}(M(\infty) \in A) \le \bar{f}(0) \operatorname{Leb}(A)$ for all $A \in \mathcal{B}(\mathbb{R})$, proving absolute continuity of the law of $M(\infty)$.
\end{proof}

\subsection{Couplings }The following lemma will be used to bound the degrees of vertices by more tractable processes in the random graph sequence $\{\G^*_n\}_{n \ge 0}$ as $n \rightarrow \infty$ for general attachment sequences $\{m_i\}_{i \ge 1}$. Recall that $s_n := \sum_{i=1}^{n}m_i$ for $n \ge 1$ and $s_0 = 0$. For $k \in \mathbb{N}$, let $s^{-1}(k)$ denote the unique $n$ such that $s_{n-1} \le k < s_{n}$. Recall that, for $k \in \mathbb{N}_0$, $d_0(k)$ denotes the degree of the root in $\mathcal{G}^*_k$ and $d_l(k)$ denotes the degree of vertex $v_l$ in $\mathcal{G}^*_k$ for $k > s_{l-1}$.

\begin{lemma}\label{coupling}
(i) Suppose $f$ is non-decreasing and suppose there exists $C_f>0$ such that $f(i) \le C_f (i+1)$ for all $i \ge 0$. Let $\underline{d}(\cdot)$ be the non-decreasing random process such that $\underline{d}(0) = d_0(0) = 0$, and for any $n \ge 1$ and any $s_{n-1} \le k < s_{n}$,
$$
\mathbb{P}\left(\underline{d}(k+1) = \underline{d}(k) + 1 \mid \underline{d}(i), i \le k\right) = 1 - \mathbb{P}\left(\underline{d}(k+1) = \underline{d}(k) \mid \underline{d}(i), i \le k\right) = \frac{f(\underline{d}(k))}{3C_f(k+1)}.
$$
Then there exists a coupling of the processes $\underline{d}(\cdot)$ and $d_0(\cdot)$ such that $\underline{d}(i) \le d_0(i)$ for all $i \ge 0$.

(ii) Suppose $f$ is non-decreasing and suppose there exists $\epsilon > 0$ such that $s^{-1}(k) \ge \epsilon (k+1)$ for all $k \ge 0$. Fix any $l \in \mathbb{N}_0$. Let $\overline{d}_l^{\epsilon}(\cdot)$ be a non-decreasing random process such that $d_l(s_{l-1}) = \overline{d}_l^{\epsilon}(s_{l-1}) = 0$, and for any $k \ge s_{l-1}$,
\begin{align*}
\mathbb{P}\left(\overline{d}_l^{\epsilon}(k+1) = \overline{d}_l^{\epsilon}(k) + 1 \mid \overline{d}_l^{\epsilon}(i), i \le k\right) &= 1 - \mathbb{P}\left(\overline{d}_l^{\epsilon}(k+1) = \overline{d}_l^{\epsilon}(k) \mid \overline{d}_l^{\epsilon}(i), i \le k\right)\\
& = \frac{f(\overline{d}_l^{\epsilon}(k))}{\epsilon (k+1) f(0)}.
\end{align*}
Then there exists a coupling of the processes $d_l(\cdot)$ and $\overline{d}_l^{\epsilon}(\cdot)$ such that $d_l(i) \le \overline{d}_l^{\epsilon}(i)$ for all $i \ge s_{l-1}$.
\end{lemma}

\begin{proof}
(i) We proceed by induction. Suppose for some $n \ge 1$ and some $s_{n-1} \le k < s_{n}$, we have coupled the random variables $\{(\underline{d}(i), d_0(i)) : i \le k\}$ such that $\underline{d}(i) \le d_0(i)$ for all $i \le k$. Let $\mathcal{H}_k := \sigma\{(\underline{d}(i), d_0(i)) : i \le k\}$. Note that as $f$ is non-decreasing, $f(i) \le C_f (i+1)$ for all $i \ge 0$, $\underline{d}(k) \le d_0(k)$ under the coupling, and $k \ge s_{n-1} \ge n-1$,
$$
\frac{f(d_0(k))}{\sum_{j=0}^{n-1}f(d_j(k))} \ge \frac{f((\underline{d}(k))}{\sum_{j=0}^{n-1}C_f(d_j(k) + 1)} \ge \frac{f((\underline{d}(k))}{C_f(2k +n)} \ge \frac{f((\underline{d}(k))}{3C_f(k + 1)},
$$
where we used $\sum_{j=0}^{n-1}d_j(k) \le \sum_{j=0}^{n}d_j(k) = 2k$ to obtain the third inequality above.

Sample a $U(0,1)$ random variable $U$ independent of $\mathcal{H}_k$ and let $(\underline{d}(k+1), d_0(k+1))= (\underline{d}(k)+1, d_0(k) + 1)$ if $0<U \le \frac{f(\underline{d}(k))}{3C_f(k+1)}$, $(\underline{d}(k+1), d_0(k+1))= (\underline{d}(k), d_0(k)+1)$ if $\frac{f(\underline{d}(k))}{3C_f(k+1)} < U \le \frac{f(d_0(k))}{\sum_{j=0}^{n-1}f(d_j(k))}$, and $(\underline{d}(k+1), d_0(k+1))= (\underline{d}(k), d_0(k))$ if $\frac{f(d_0(k))}{\sum_{j=0}^{n-1}f(d_j(k))} < U <1$. Under this coupling, we obtain $\underline{d}(i) \le d_0(i)$ for all $i \le k+1$. Thus, by induction, we obtain the required coupling.

(ii) The proof follows similarly upon noting that, if we have coupled $\{d_l(i), \overline{d}_l^{\epsilon}(i) : i \le k\}$ such that $d_l(i) \le \overline{d}_l^{\epsilon}(i)$ for all $i \le k$, then
$$
\frac{f(\overline{d}_l^{\epsilon}(k))}{\epsilon (k+1) f(0)} \ge \frac{f(\overline{d}_l^{\epsilon}(k))}{s^{-1}(k) f(0)} \ge  \frac{f(d_l(k))}{\sum_{j=0}^{s^{-1}(k)-1} f(d_j(k))}.
$$
\end{proof}
The following lemma, which is routine to check upon noting that $d_n(s_n) = m_{n}$ and using the properties of exponential distributions, will play a crucial role in our analysis. It describes the joint evolution of the degrees of vertices $v_i$ and $v_n$ for $i \ge 0$, $n >i$, in terms of a bivariate point process observed at jump times. Moreover, this process is a Markov chain whose transition probabilities are explicitly recorded.
\begin{lemma}\label{embmult}
Fix any $i\ge 0, n > i$. Let $\tilde{\tau}_0 = s_n$ and for $j \ge 1$, let $\tilde{\tau}_j := \inf\{k > \tilde{\tau}_{j-1}: (d_i(k), d_n(k)) \neq (d_i(\tilde{\tau}_{j-1}), d_n(\tilde{\tau}_{j-1}))\}$. For $j \ge 0$, let $\mathbf{X}_j := (X^{(1)}_j, X^{(2)}_j) := (d_i(\tilde{\tau}_j ), d_n(\tilde{\tau}_j ))$.
Let $\{\xi_A(\cdot) : A \in \mathbb{N}_0\}$ be independent point processes having laws prescribed in Section \ref{techlem} and independent of $\sigma\{\mathcal{G}^*_{k} : k \ge 0\}$. Let $\tilde{\xi}_1(\cdot) := d_i(s_n) + \xi_{d_i(s_n)}(\cdot)$ and $\tilde{\xi}_2(\cdot) :=m_n + \xi_{m_n}(\cdot)$. Define $\sigma_0 = 0$ and for $j \ge 1$, let $\sigma_j := \inf\{t > \sigma_{j-1}: (\tilde{\xi}_1(t), \tilde{\xi}_2(t)) \neq (\tilde{\xi}_1(\sigma_{j-1}), \tilde{\xi}_2(\sigma_{j-1}))\}$. For $j \ge 0$, let $\mathbf{Y}_j := (Y^{(1)}_j, Y^{(2)}_j) := (\tilde{\xi}_1(\sigma_j), \tilde{\xi}_2(\sigma_j))$. Then $\{\mathbf{X}_j\}_{j \ge 0}$ and $\{\mathbf{Y}_j\}_{j \ge 0}$ are Markov chains having the same distribution. This Markov chain has starting point $(d_i(s_n), m_n)$ and its transition probabilities are given for $j \ge 1$ by
\begin{multline*}
\mathbb{P}\left((X^{(1)}_j, X^{(2)}_j) = (X^{(1)}_{j-1} + 1, X^{(2)}_{j-1}) \mid (X^{(1)}_{j-1}, X^{(2)}_{j-1})\right)\\
= 1 - \mathbb{P}\left((X^{(1)}_j, X^{(2)}_j) = (X^{(1)}_{j-1}, X^{(2)}_{j-1} + 1) \mid (X^{(1)}_{j-1}, X^{(2)}_{j-1})\right) = \frac{f(X^{(1)}_{j-1})}{f(X^{(1)}_{j-1}) + f(X^{(2)}_{j-1})}.
\end{multline*}
\end{lemma}
\begin{corollary}\label{distlim}
Suppose $\Phi_2(\infty)< \infty$. Then, almost surely, for any $i \ge 0$ and any $n>i$, $d_i(k) = d_n(k)$ for only finitely many $k \ge s_n$.
\end{corollary}
\begin{proof}
Define the processes
\begin{align*}
\tilde M_1(t) &:= \sum_{j=d_i(s_n)}^{\tilde{\xi}_1(t)-1}\frac{1}{f(j)} - t = \sum_{j=0}^{\tilde{\xi}_1(t) - d_i(s_n)-1}\frac{1}{f_{d_i(s_n)}(j)} - t ,\\
\tilde M_2(t) &:= \sum_{j=m_n}^{\tilde{\xi}_2(t) -1}\frac{1}{f(j)} - t = \sum_{j=0}^{\tilde{\xi}_2(t) - m_n -1}\frac{1}{f_{m_n}(j)} - t,
\end{align*}
where $\tilde{\xi}_1(\cdot)$ and $\tilde{\xi}_2(\cdot)$ are as in Lemma \ref{embmult}. By Lemma \ref{mgle}, the above processes are martingales and by Lemma \ref{ac}, both of them converge almost surely as $t \rightarrow \infty$ to limiting random variables, which are conditionally independent given $\sigma\{\mathcal{G}^*_{j} : j \le s_n\}$, and whose laws are absolutely continuous with respect to Lebesgue measure. In particular, almost surely,
$$
\lim_{t \rightarrow \infty} \sum_{j=0}^{\tilde{\xi}_1(t)-1}\frac{1}{f(j)} - t \neq \lim_{t \rightarrow \infty} \sum_{j=0}^{\tilde{\xi}_2(t)-1}\frac{1}{f(j)} - t. 
$$ 
By Lemma \ref{embmult}, the corollary follows.
\end{proof}

\subsection{Checking Assumption \eqref{eqn:prop-under-lamb} }In the case when $m_i =1$ for all $i \ge 1$, recall the continuous time embedding of the sequence of random trees discussed in Section \ref{ctbp}. The following lemma gives checkable conditions under which Assumption \eqref{eqn:prop-under-lamb} is satisfied when $\Phi_2(\infty)< \infty$ and $\limsup_{i \rightarrow \infty} \frac{f(i)}{i} < \infty$. Recall the function $\hat{\rho}(\cdot)$ defined in \eqref{eqn:rho-hat-def}.
\begin{lemma}\label{checkcond}
Assume $\Phi_2(\infty)< \infty$ and $\bar{D} := \limsup_{i \rightarrow \infty} \frac{f(i)}{i} < \infty$. Then Assumption \eqref{eqn:prop-under-lamb} is satisfied if $\hat{\rho}(\bar{D}) >1$. In particular, Assumption \eqref{eqn:prop-under-lamb} is satisfied if $\Phi_2(\infty)< \infty$ and $\lim_{i \rightarrow \infty} \frac{f(i)}{i} = 0$.
\end{lemma}

\begin{proof}
Using the inequality $z \ge \log(1+z) \ge z - z^2/2$ for $z \ge 0$, note that for any $\lambda >0$,
$$
-\lambda \Phi_1(k) \le \log \prod_{i=0}^{k-1} \frac{f(i)}{\lambda + f(i)} = - \sum_{i=0}^{k-1}\log \left(1 + \frac{\lambda}{f(i)}\right) \le -\lambda \Phi_1(k) + \frac{\lambda^2}{2} \Phi_2(\infty)
$$
Hence, for any $\lambda >0$,
\begin{equation}\label{rhohatest}
\sum_{k=1}^{\infty} \exp \left\lbrace -\lambda \Phi_1(k)  \right \rbrace \le \hat{\rho}(\lambda) \le \exp \left\lbrace \frac{\lambda^2}{2} \Phi_2(\infty)  \right \rbrace \sum_{k=1}^{\infty} \exp \left\lbrace -\lambda \Phi_1(k)  \right \rbrace.
\end{equation}
Fix any $\epsilon>0$. There exists $n_{\epsilon}>0$ such that $f(i) \le (\bar{D} + \epsilon) i$ for all $i \ge n_{\epsilon}$. Thus, for any $n > n_{\epsilon}$,
$$
\Phi_1(n) \ge \sum_{i=n_{\epsilon}}^{n-1}\frac{1}{f(i)} \ge \frac{1}{\bar{D} + \epsilon}\sum_{i=n_{\epsilon}}^{n-1}\frac{1}{i} \ge \frac{1}{\bar{D} + \epsilon}\int_{n_{\epsilon}}^{n}\frac{1}{x}dx = \frac{\log n - \log n_{\epsilon}}{\bar{D} + \epsilon}.
$$
From this bound and \eqref{rhohatest}, for any $\lambda > \bar{D} + \epsilon$,
\begin{multline*}
\hat{\rho}(\lambda) - \exp \left\lbrace \frac{\lambda^2}{2} \Phi_2(\infty)  \right \rbrace \sum_{k=1}^{n_{\epsilon}} \exp \left\lbrace -\lambda \Phi_1(k)  \right \rbrace \le  \exp \left\lbrace \frac{\lambda^2}{2} \Phi_2(\infty)  \right \rbrace \sum_{k=n_{\epsilon}+1}^{\infty} \exp \left\lbrace -\lambda \Phi_1(k)  \right \rbrace\\
\le \exp \left\lbrace \frac{\lambda^2}{2} \Phi_2(\infty)  \right \rbrace \sum_{k=n_{\epsilon}+1}^{\infty} \exp \left\lbrace -\frac{\lambda}{\bar{D} + \epsilon} (\log n - \log n_{\epsilon})\right \rbrace \\
= \exp \left\lbrace \frac{\lambda^2}{2} \Phi_2(\infty) + \frac{\lambda \log n_{\epsilon}}{\bar{D} + \epsilon} \right \rbrace \sum_{k=n_{\epsilon}+1}^{\infty} \frac{1}{n^{\lambda/(\bar{D} + \epsilon)}} < \infty.
\end{multline*}
Therefore, as $\epsilon>0$ is arbitrary, $\hat{\rho}(\lambda) < \infty$ for all $\lambda  > \bar{D}$. Thus, $\underline{\lambda} \le \bar{D}$. Moreover, as $\hat{\rho}(\bar{D}) >1$, there exists $k_0 \in \mathbb{N}$ such that
$
\sum_{k=1}^{k_0} \prod_{i=0}^{k-1} \frac{f(i)}{\bar{D} + f(i)} > 1.
$
By continuity of $\lambda \mapsto \sum_{k=1}^{k_0} \prod_{i=0}^{k-1} \frac{f(i)}{\lambda + f(i)}$, we can choose $\delta>0$ such that $\sum_{k=1}^{k_0} \prod_{i=0}^{k-1} \frac{f(i)}{\bar{D} + \delta + f(i)} > 1$. Hence,
$$
\hat{\rho}(\bar{D} + \delta) \ge \sum_{k=1}^{k_0} \prod_{i=0}^{k-1} \frac{f(i)}{\bar{D} + \delta + f(i)} > 1.
$$
As $\hat{\rho}(\cdot)$ is non-increasing in $\lambda$ and $\underline{\lambda} \le \bar{D}$,
$$
\lim_{\lambda\downarrow\underline{\lambda}} \hat{\rho}(\lambda) \ge \hat{\rho}(\bar{D} + \delta)>1.
$$
Thus, Assumption \eqref{eqn:prop-under-lamb} holds under the hypothesis of the lemma. If $\lim_{i \rightarrow \infty} \frac{f(i)}{i} = 0$, $\bar{D} = 0$. As $\hat{\rho}(0) = \infty$, Assumption \eqref{eqn:prop-under-lamb} holds if $\Phi_2(\infty)< \infty$ and $\lim_{i \rightarrow \infty} \frac{f(i)}{i} = 0$.
\end{proof}

\subsection{Functional Central Limit Theorem }Now we explore the case when $\Phi_2(\infty) = \infty$. The following lemma gives a functional central limit theorem for the martingale $M_A(\cdot)$ for any $A \in \mathbb{N}_0$. Recall $\mathcal{K}(t) = \Phi_2 \circ \Phi_1^{-1}(t), t \ge 0$. This lemma generalizes the last assertion of Lemma 2.1 in \cite{DM} by replacing monotonicity of $f$ (required in the proof in \cite{DM}) by a much more general  
`continuity assumption' \eqref{Kcont0} on $\K(\cdot)$. As in \cite{DM}, the proof proceeds by showing that the quadratic variation of an appropriately scaled, time-changed martingale converges to that of Brownian motion. However, in the absense of monotonicity of $f$, we cannot use a `law of large numbers' type result for the martingale to directly estimate its quadratic variation (see the displays after equation (5) and before equation (6) in \cite{DM}). One needs to do a little more work to estimate the quadratic variation, as shown in the proof of Lemma \ref{fcltper}. 

In the following, $D([0,\infty)  :  \mathbb{R})$ denotes the space of c\`adl\`ag functions equipped with the Skorohod topology.
\begin{lemma}\label{fcltper}
Assume $\Phi_2(\infty) = \infty$ and further
\begin{equation}\label{Kcont0}
\lim_{\delta \downarrow 0} \limsup_{t \rightarrow \infty} \frac{\K((1+\delta)t)}{\K(t)} = 1.
\end{equation}
Then, for any $A \in \mathbb{N}_0$, the processes $\{M_A^{(n)}(\cdot)\}_{n \ge 1}$ defined by 
\begin{equation}\label{scaledproc}
M_A^{(n)}(t) := n^{-1/2} M_A(\K^{-1}(nt)), \ t \ge 0, \ n \ge 1,
\end{equation}
converge weakly to standard Brownian motion in $D([0,\infty)  :  \mathbb{R})$ as $n \rightarrow \infty$.
\end{lemma}

\begin{proof}
The key idea in the proof is to use a martingale functional central limit theorem \cite[Theorem 2.1 (i)]{whitt2007proofs}, which requires showing that the quadratic variation of the martingale $M_A^{(n)}$ converges to that of Brownian motion in a certain sense, namely \eqref{qvBM}. To do so, we aim at proving \eqref{quadconv}.

First, note that for $i=1,2$, $t \ge 0$, $A \in \mathbb{N}_0$,
$$
\Phi^A_i(t) = \int_0^t \frac{1}{f^i_A(s)}ds = \int_A^{t+A}\frac{1}{f^i(s)}ds = \Phi_i(t) + \int_t^{t+A}\frac{1}{f^i(s)}ds - \int_0^{A}\frac{1}{f^i(s)}ds.
$$
Recall $f_* := \inf_{i \ge 0}f(i)$. We deduce from the above equalities that for $i=1,2$, $t \ge 0$, $A \in \mathbb{N}_0$:
\begin{align*}
\Phi_i(t) - \frac{A}{f_*^i} & \le \Phi^A_i(t) \le \Phi_i(t) + \frac{A}{f_*^i},\\
\Phi^{-1}_i\left(t - \frac{A}{f_*^i}\right) & \le (\Phi^A_i)^{-1}(t) \le \Phi^{-1}_i\left(t + \frac{A}{f_*^i}\right).
\end{align*}
Recall $\K_A(t) := \Phi^A_2 \circ (\Phi^A_1)^{-1}(t), t \ge 0$. From the above relations, it follows that for $t \ge 0$, $A \ge 0$,
\begin{equation*}
\K\left(t - \frac{A}{f_*}\right) - \frac{A}{f_*^2} \le \K_A(t) \le \K\left(t + \frac{A}{f_*}\right) + \frac{A}{f_*^2}.
\end{equation*}
From this and \eqref{Kcont0}, we conclude that for any $A \ge 0$,
\begin{equation}\label{Krat}
\frac{\K_A(t)}{\K(t)} \rightarrow 1  \ \text{ as } \ t \rightarrow \infty.
\end{equation}
In particular, for any $A \in \mathbb{N}_0$,
\begin{equation}\label{KcontA}
\lim_{\delta \downarrow 0} \limsup_{t \rightarrow \infty} \frac{\K_A((1+\delta)t)}{\K_A(t)} = 1.
\end{equation}
Let $N_A(t) := \Phi^A_1(\xi_A(t)), t \ge 0$, and $\bar{f}_A(\cdot) := f_A \circ (\Phi^A_1)^{-1}(\cdot)$. Recall by Lemma \ref{mgle} (with $M_A(\cdot)$ replacing $M(\cdot)$) that
\begin{equation}\label{qvre}
\langle M_A \rangle(t)  = \int_0^t\frac{1}{f_A\left(\xi_A(s-)\right)}ds.
\end{equation}
As $M_A(\cdot)$ is a martingale, by Doob's $L^2$-inequality, for any $i \ge 0$,
\begin{align*}
\frac{\mathbb{E}\left(\sup_{0 \le t \le 2^{i+1}}|M_A(t)|^2\right)}{(2^{i/2}\log 2^i)^2} & \le 4\frac{\mathbb{E}\left(|M_A(2^{i+1})|^2\right)}{(2^{i/2}\log 2^i)^2}\\
&= 4\frac{\mathbb{E}\left(\langle M_A \rangle(2^{i+1})\right)}{(2^{i/2}\log 2^i)^2} \le \frac{4f_*^{-1}2^{i+1}}{2^i i^2(\log 2)^2} = \frac{8}{(\log 2)^2f_*} \frac{1}{i^2},
\end{align*}
where we used \eqref{qvre} for the second inequality.
From this, following the calculation used to derive equation (5) in \cite{DM}, we obtain $i_0 \in \mathbb{N}$ such that for any $i \ge i_0$,
\begin{align*}
\mathbb{P}\left(\sup_{t \ge 2^i}\frac{|M_A(t)|}{\sqrt{t} \log t} \ge 1\right) &\le \sum_{k=i}^{\infty}\mathbb{P}\left(\sup_{2^k \le t \le 2^{k+1}}\frac{|M_A(t)|}{\sqrt{t} \log t} \ge 1\right)\le \sum_{k=i}^{\infty}\mathbb{P}\left(\sup_{0\le t \le 2^{k+1}}|M_A(t)| \ge 2^{\frac{k}{2}} \log 2^k\right)\\
& \le \sum_{k=i}^{\infty} \frac{\mathbb{E}\left(\sup_{0 \le t \le 2^{k+1}}|M_A(t)|^2\right)}{(2^{k/2}\log 2^k)^2} \le \sum_{k=i}^{\infty} \frac{8}{(\log 2)^2f_*} \frac{1}{k^2} < \infty.
\end{align*}
Hence, using the Borel-Cantelli Lemma, we conclude
$$
\limsup_{t \rightarrow \infty} \frac{|M_A(t)|}{\sqrt{t} \log t} \le 1 \  \ \text{almost surely}.
$$
In particular, recalling $N_A(t) = M_A(t) + t, \, t \ge 0$,
\begin{equation}\label{f0}
\frac{N_A(t)}{t} \xrightarrow{a.s} 1  \ \text{ as } \ t \rightarrow \infty.
\end{equation}
Recall $S^A_i(n) := \sum_{k=0}^{n-1}\frac{\mathbf{E}_k}{f_A^i(k)}, n \ge 0, \ i=1,2$, where $\{\mathbf{E}_k\}_{k \ge 0}$ are i.i.d. exponential random variables with mean $1$. Extend $S^A_i(\cdot), i =1,2,$ to $[0,\infty)$ by linear interpolation. By the first equality in \eqref{conc1},
\begin{equation}\label{f1}
\left| \langle M_A \rangle(t) - S^A_2(\xi_A(t))  \right| \le \frac{\mathbf{E}_{\xi_A(t)}}{f_A^2(\xi_A(t))}.
\end{equation}
Now, for $n \in \mathbb{N}_0$,
$$
S^A_2(n) =  \sum_{k=0}^{n-1}\frac{(\mathbf{E}_k - 1)}{f_A^2(k)} + \Phi^A_2(n)
$$
and hence, by Tchebyshev's inequality, for any $\delta>0$,
$$
\mathbb{P}\left(\left|\frac{S^A_2(n)}{\Phi^A_2(n)} - 1\right| > \delta\right) \le \frac{1}{\left(\Phi^A_2(n)\right)^2} \operatorname{Var}\left(\sum_{k=0}^{n-1}\frac{(\mathbf{E}_k - 1)}{f_A^2(k)}\right) = \frac{\Phi^A_4(n)}{\left(\Phi^A_2(n)\right)^2} \le \frac{1}{f_*^2\Phi^A_2(n)} \rightarrow 0
$$
as $n \rightarrow \infty$ as $\Phi^A_2(\infty) = \Phi_2(\infty) = \infty$. Hence,
\begin{equation}\label{f2}
\frac{S^A_2(n)}{\Phi^A_2(n)} \xrightarrow{P} 1 \ \text{ as } \ n \rightarrow \infty.
\end{equation}
For any $\epsilon \in (0,1), t >0$, if $(1-\epsilon)t \le N_A(t) \le (1+\epsilon)t$, then noting $\Phi^A_2(\xi_A(t)) = \K_A(N_A(t))$, we obtain
\begin{align}\label{f3}
\frac{S^A_2(\xi_A(t))}{\K_A(t)} &= \frac{S^A_2(\xi_A(t))}{\Phi^A_2(\xi_A(t))}\frac{\K_A(N_A(t))}{\K_A(t)} \le \frac{S^A_2((\Phi^A_1)^{-1}((1+\epsilon)t))}{\Phi^A_2((\Phi^A_1)^{-1}((1-\epsilon)t))}\frac{\K_A((1+\epsilon)t)}{\K_A(t)}\\
&=  \frac{S^A_2((\Phi^A_1)^{-1}((1+\epsilon)t))}{\Phi^A_2((\Phi^A_1)^{-1}((1+\epsilon)t))} \ \frac{\K_A((1+\epsilon)t)}{\K_A((1-\epsilon)t)} \ \frac{\K_A((1+\epsilon)t)}{\K_A(t)},\nonumber
\end{align}
and similarly,
\begin{equation}\label{f4}
\frac{S^A_2(\xi_A(t))}{\K_A(t)} \ge \frac{S^A_2((\Phi^A_1)^{-1}((1-\epsilon)t))}{\Phi^A_2((\Phi^A_1)^{-1}((1- \epsilon)t))} \ \frac{\K_A((1-\epsilon)t)}{\K_A((1+\epsilon)t)} \ \frac{\K_A((1-\epsilon)t)}{\K_A(t)}.
\end{equation}
Take any $\eta \in (0,1/2)$. By \eqref{KcontA}, we can choose $\epsilon_{\eta}>0$ and $t_{\eta}>0$ such that for all $t \ge t_{\eta}$,
$$
\frac{\K_A^2((1-\epsilon_{\eta})t)}{\K_A((1+\epsilon_{\eta})t)\K_A(t)} > 1- \eta, \ \frac{\K_A^2((1+\epsilon_{\eta})t)}{\K_A((1-\epsilon_{\eta})t)\K_A(t)} < 1+ \eta.
$$
Therefore, by \eqref{f3} and \eqref{f4}, for $t \ge t_{\eta}$, 
\begin{multline*}
\mathbb{P}\left(\left|\frac{S^A_2(\xi_A(t))}{\K_A(t)} - 1\right| > 2\eta \right) \le \mathbb{P}\left(\left| \frac{N_A(t)}{t} - 1\right| > \epsilon_{\eta}\right)\\
 + \mathbb{P}\left( \frac{S^A_2((\Phi^A_1)^{-1}((1+\epsilon)t))}{\Phi^A_2((\Phi^A_1)^{-1}((1+\epsilon)t))} > \frac{1+2\eta}{1+\eta}\right) + \mathbb{P}\left(\frac{S^A_2((\Phi^A_1)^{-1}((1-\epsilon)t))}{\Phi^A_2((\Phi^A_1)^{-1}((1-\epsilon)t))}  < \frac{1-2\eta}{1-\eta}\right).
\end{multline*}
Hence, by \eqref{f0} and \eqref{f2}, we obtain
\begin{equation}\label{f5}
\frac{S^A_2(\xi_A(t))}{\K_A(t)} \xrightarrow{P} 1  \ \text{ as } \ t \rightarrow \infty.
\end{equation}
Moreover, for any $\eta>0$ and any $\epsilon \in (0,1)$,
\begin{align}\label{f6}
\mathbb{P}\left( \frac{\mathbf{E}_{\xi_A(t)}}{f_A^2(\xi_A(t))} > \eta \K_A(t)\right)& \le \mathbb{P}\left(\xi_A(t) \notin \left((\Phi^A_1)^{-1}((1-\epsilon)t + 1),(\Phi^A_1)^{-1}((1+ \epsilon)t) - 1\right)\right)\\
&\quad + \mathbb{P}\left(\sum_{k=\lfloor (\Phi^A_1)^{-1}((1-\epsilon)t)\rfloor + 1}^{\lfloor (\Phi^A_1)^{-1}((1+\epsilon)t) \rfloor} \frac{\mathbf{E}_{k}}{f_A^2(k)}> \eta \K_A(t)\right)\notag\\
& \le \mathbb{P}\left(\xi_A(t) \notin \left((\Phi^A_1)^{-1}((1-\epsilon)t), (\Phi^A_1)^{-1}((1+ \epsilon)t) - 1\right)\right)\notag\\
&\quad + \frac{\K_A((1+\epsilon)t) - \K_A((1-\epsilon)t)}{\eta \K_A(t)}\notag\\
&\rightarrow 0 \ \text{ as } t \rightarrow \infty \notag
\end{align}
where we used Markov's inequality in the second inequality, and the convergence in the last step is a result of \eqref{f0} and \eqref{KcontA}.
Thus, by \eqref{f1}, \eqref{f5} and \eqref{f6}, we conclude that
\begin{equation*}
\frac{\langle M_A \rangle(t)}{\K_A(t)} \xrightarrow{P} 1  \ \text{ as } \ t \rightarrow \infty.
\end{equation*}
This, along with \eqref{Krat}, implies the following asymptotic behavior of the quadratic variation process $\langle M_A \rangle(\cdot)$ for any $A \in \mathbb{N}_0$:
\begin{equation}\label{quadconv}
\frac{\langle M_A \rangle(t)}{\K(t)} \xrightarrow{P} 1  \ \text{ as } \ t \rightarrow \infty.
\end{equation}
Now we apply a martingale functional central limit theorem \cite[Theorem 2.1 (i)]{whitt2007proofs}.  Consider the martingale $M_A^{(n)}(\cdot)$ defined in \eqref{scaledproc}. By \eqref{quadconv}, for each $A \in \mathbb{N}_0$, $t>0$,
\begin{equation}\label{qvBM}
\lim_{n \rightarrow \infty} \langle M_A^{(n)} \rangle(t) \xrightarrow{P} t  \ \text{ as } \ n \rightarrow \infty.
\end{equation}
Moreover, the jumps of the martingale $M_A^{(n)}(\cdot)$ are deterministically bounded by $n^{-1/2}f_*^{-1}$ which converges to $0$ as $n \rightarrow \infty$. Hence, by \cite[Theorem 2.1 (i)]{whitt2007proofs}, for each $A \in \mathbb{N}_0$, $M_A^{(n)}(\cdot)$ converges weakly to standard Brownian motion in $D([0,\infty)  :  \mathbb{R})$ as $n \rightarrow \infty$.
\end{proof}

\subsection{Moderate deviation principles }The following moderate deviation principle (MDP) will be crucial in investigating the index of the maximal degree vertex when $\Phi_2(\infty) = \infty$.
\begin{lemma}\label{mdpexp}
Assume $\Phi_2(\infty) = \infty$. Let $\{b_n\}_{n \ge 1}$ be any positive sequence such that $b_n \rightarrow \infty$ and $\frac{b_n}{\Phi_2(n)} \rightarrow 0$ as $n \rightarrow \infty$. Then for any $x \ge 0$,
\begin{align}\label{M1}
\lim_{n \rightarrow \infty} \frac{b_n}{\Phi_2(n)} \log \mathbb{P}\left(S_1(n) - \Phi_1(n) \ge x \frac{\Phi_2(n)}{\sqrt{b_n}}\right) &= -\frac{x^2}{2},\\\label{M2}
\lim_{n \rightarrow \infty} \frac{b_n}{\Phi_2(n)} \log \mathbb{P}\left(S_1(n) - \Phi_1(n) \le -x \frac{\Phi_2(n)}{\sqrt{b_n}}\right) &= -\frac{x^2}{2}.
\end{align}
If in addition $f(k) \rightarrow \infty$ as $k \rightarrow \infty$, then for any $x \ge 0$,
\begin{equation}\label{M3}
\lim_{n \rightarrow \infty} \frac{1}{\Phi_2(n)} \log \mathbb{P}\left(S_1(n) - \Phi_1(n) \le -x \Phi_2(n)\right) = -\frac{x^2}{2}.
\end{equation}
\end{lemma}
\begin{proof}
Define
$$
Z_n := \frac{\sqrt{b_n}}{\Phi_2(n)}(S_1(n) - \Phi_1(n))
$$
and for any $\lambda \in \mathbb{R}$,
$$
\Lambda_n(\lambda) := \log \mathbb{E}\left(\exp\left\lbrace \lambda Z_n \right\rbrace \right).
$$
Then, for any $\lambda \in \mathbb{R}$, using the explicit moment generating function of exponential random variables, we obtain for any $n \ge 1$ such that $|\lambda| < \sqrt{b_n} f_*$,
\begin{multline}\label{test0}
\Lambda_n\left(\frac{\Phi_2(n)}{b_n}\lambda\right) = \log \mathbb{E}\left(\exp\left\lbrace \frac{\lambda}{\sqrt{b_n}}S_1(n) \right\rbrace \right) - \frac{\lambda \Phi_1(n)}{\sqrt{b_n}} = - \frac{\lambda \Phi_1(n)}{\sqrt{b_n}} -\sum_{i=0}^{n-1}\log \left(1 - \frac{\lambda}{\sqrt{b_n}f(i)}\right).
\end{multline}
Using the inequalities $-z - z^2/2 + \frac{|z|^3}{1-|z|} \ge \log(1- z) \ge -z - z^2/2 - \frac{|z|^3}{1-|z|}$ for $|z| < 1$ in the above expression gives for any $n \ge 1$ such that $|\lambda| < \sqrt{b_n} f_*$,
\begin{align}\label{test1}
\Lambda_n\left(\frac{\Phi_2(n)}{b_n}\lambda\right) &\le  - \frac{\lambda \Phi_1(n)}{\sqrt{b_n}} + \sum_{i=0}^{n-1}\left[\frac{\lambda}{\sqrt{b_n}f(i)} + \frac{\lambda^2}{2b_nf^2(i)} + \frac{|\lambda|^3}{b_nf^2(i)(\sqrt{b_n} f(i) - |\lambda|)}\right]\\
& \le \frac{\lambda^2\Phi_2(n)}{2b_n} + \frac{|\lambda|^3 \Phi_2(n)}{b_n(\sqrt{b_n} f_* - |\lambda|)}\notag
\end{align}
and
\begin{align}\label{test2}
\Lambda_n\left(\frac{\Phi_2(n)}{b_n}\lambda\right) &\ge - \frac{\lambda \Phi_1(n)}{\sqrt{b_n}} + \sum_{i=0}^{n-1}\left[\frac{\lambda}{\sqrt{b_n}f(i)} + \frac{\lambda^2}{2b_nf^2(i)} - \frac{|\lambda|^3}{b_nf^2(i)(\sqrt{b_n} f(i) - |\lambda|)}\right]\\
& \ge \frac{\lambda^2\Phi_2(n)}{2b_n} - \frac{|\lambda|^3 \Phi_2(n)}{b_n(\sqrt{b_n} f_* - |\lambda|)}.\notag
\end{align}
From the above bounds, we obtain for any fixed $\lambda \in \mathbb{R}$, taking $n \ge 1$ large enough such that  $|\lambda| < \sqrt{b_n} f_*$,
\begin{equation}\label{mdpexp1}
\left|\frac{b_n}{\Phi_2(n)}\Lambda_n\left(\frac{\Phi_2(n)}{b_n}\lambda\right) - \frac{\lambda^2}{2}\right| \le \frac{|\lambda|^3 }{(\sqrt{b_n} f_* - |\lambda|)}  \rightarrow 0
\end{equation}
as $n \rightarrow \infty$. \eqref{M1} and \eqref{M2} now follow by the G\"artner-Ellis Theorem \cite[Section 2.3, Theorem 2.3.6 (c)]{dembo2011large}.

To prove \eqref{M3}, we use the same notation as above with $b_n \equiv 1$. By the same calculation as in \eqref{test1} and \eqref{test2} with $b_n = 1$, we get for any $\lambda \in (-\infty, f_*)$,
\begin{align*}
\left|\frac{1}{\Phi_2(n)}\Lambda_n\left(\Phi_2(n)\lambda\right) - \frac{\lambda^2}{2}\right| &\le \frac{1}{\Phi_2(n)}\sum_{i=0}^{n-1} \frac{|\lambda|^3}{f^2(i)(f(i) - |\lambda|)}\\
&\le \frac{1}{\Phi_2(n)}\sum_{i=0}^{n-1} \frac{|\lambda|^3}{f^3(i)(1 - |\lambda|f_*^{-1})}   = \frac{|\lambda|^3 f_* \Phi_3(n)}{(f_* - |\lambda|)\Phi_2(n)}.
\end{align*}
Since $f(k) \rightarrow \infty$ as $k \rightarrow \infty$, for any $\epsilon >0$, there exists $k_{\epsilon} \in \mathbb{N}$ such that $f(k) \ge 1/\epsilon$ for all $k \ge k_{\epsilon}$. Therefore, for any $n > k_{\epsilon}$,
\begin{equation}\label{re}
\Phi_3(n) \le \sum_{k=0}^{k_{\epsilon} - 1}\frac{1}{f^3(k)} + \epsilon \sum_{k=k_{\epsilon}}^{n - 1}\frac{1}{f^2(k)} \le \sum_{k=0}^{k_{\epsilon} - 1}\frac{1}{f^3(k)} + \epsilon \Phi_2(n).
\end{equation}
As $\Phi_2(\infty) = \infty$,
$
\limsup_{n \rightarrow \infty} \frac{\Phi_3(n)}{\Phi_2(n)} \le \epsilon,
$
and as $\epsilon>0$ is arbitrary, for any $\lambda \in (-\infty, f_*)$,
\begin{equation*}
\lim_{n \rightarrow \infty} \frac{1}{\Phi_2(n)}\Lambda_n\left(\Phi_2(n)\lambda\right) = \frac{\lambda^2}{2}.
\end{equation*}
Further, since  $f(k) \rightarrow \infty$ as $k \rightarrow \infty$, there exists $n_* \in \mathbb{N}_0$ such that $f(n_*) = f_*$. Hence, for any $n \ge n_* + 1$, $\lambda \ge f_*$,
$$
\mathbb{E}\left(\exp\left\lbrace \lambda S_1(n)\right\rbrace\right) \ge \mathbb{E}\left(\exp\left\lbrace \lambda \mathbf{E}_{n_*}/f(n_*)\right\rbrace\right) = \infty.
$$
Thus, by the first equality in \eqref{test0} (with $b_n =1$), for $\lambda \ge f_*$, $\lim_{n \rightarrow \infty} \frac{1}{\Phi_2(n)}\Lambda_n\left(\Phi_2(n)\lambda\right) = \infty$. Following the notation of \cite[Section 2.3]{dembo2011large}, 
$$
\Lambda(\lambda) := \lim_{n \rightarrow \infty} \frac{1}{\Phi_2(n)}\Lambda_n\left(\Phi_2(n)\lambda\right)
$$
exists in $[-\infty, \infty]$ for all $\lambda \in \mathbb{R}$ and $\mathcal{D}_{\Lambda} := \{\lambda \in \mathbb{R} : \Lambda(\lambda) < \infty\} = (-\infty, f_*)$. The Fenchel-Legendre transform $\Lambda^*(\cdot)$ of $\Lambda(\cdot)$ is given by
$$
\Lambda^*(x) =
\left\{
	\begin{array}{ll}
		 \frac{x^2}{2}  & \mbox{if }  x  < f_*,\\
		f_*x - \frac{f_*^2}{2} & \mbox{if } x \ge f_*.
	\end{array}
\right.
$$
Therefore, $\mathcal{D}_{\Lambda^*} := \{x \in \mathbb{R} : \Lambda^*(x) < \infty\} = \mathbb{R}$. Moreover, for any $y \in \mathcal{D}_{\Lambda} \subset \mathcal{D}_{\Lambda^*}$, it can be checked that $y$ is an exposed point with exposing hyperplane $\lambda = y \in \mathcal{D}^0_{\Lambda}$ (in the sense of Definition 2.3.3 of \cite[Section 2.3]{dembo2011large}). Therefore, if $E$ denotes the set of exposed points of $\Lambda^*(\cdot)$ whose exposing hyperplane belongs to $ \mathcal{D}^0_{\Lambda}$, then $\mathcal{D}_{\Lambda} \subseteq E$. Therefore, by the  G\"artner-Ellis Theorem, for any $x \ge 0$,
$$
\limsup_{n \rightarrow \infty} \frac{1}{\Phi_2(n)} \log \mathbb{P}\left(\Phi_2(n)^{-1}(S_1(n) - \Phi_1(n)) \in (-\infty,-x] \right) \le -\inf_{z \in (-\infty,-x]} \Lambda^*(z) = -\frac{x^2}{2}
$$
and
\begin{multline*}
\liminf_{n \rightarrow \infty} \frac{1}{\Phi_2(n)} \log \mathbb{P}\left(\Phi_2(n)^{-1}(S_1(n) - \Phi_1(n)) \in (-\infty,-x] \right)\\
\ge \liminf_{n \rightarrow \infty} \frac{1}{\Phi_2(n)} \log \mathbb{P}\left(\Phi_2(n)^{-1}(S_1(n) - \Phi_1(n)) \in (-\infty,-x) \right) \ge -\inf_{z \in (-\infty,-x) \cap E} \Lambda^*(z)\\
\ge - \inf_{z \in (-\infty,-x) \cap \mathcal{D}_{\Lambda}} \Lambda^*(z) =  - \inf_{z \in (-\infty,-x)} \Lambda^*(z) = -\frac{x^2}{2}.
\end{multline*}
From the above two calculations, we conclude that for any $x \ge 0$,
$$
\lim_{n \rightarrow \infty} \frac{1}{\Phi_2(n)} \log \mathbb{P}\left(\Phi_2(n)^{-1}(S_1(n) - \Phi_1(n)) \in (-\infty,-x] \right) = -\frac{x^2}{2}
$$
proving \eqref{M3}.
\end{proof}

Recall $N(t) = \Phi_1(\xi(t)), t \ge 0$. The next lemma uses the above MDP (in particular, \eqref{M3}) to establish an MDP for the process $N(\cdot)$.

\begin{lemma}\label{MDPN}
Assume $\Phi_2(\infty) =\infty$ and $f(k) \rightarrow \infty$ as $k \rightarrow \infty$. Then, for any $x \ge 0$,
$$
\lim_{t \rightarrow \infty} \frac{1}{\K(t)} \log \mathbb{P}\left(N(t - x\K(t)) \ge t\right) = -\frac{x^2}{2}.
$$
\end{lemma}

\begin{proof}
For any $n \in \mathbb{N}$, $x \ge 0$,
\begin{align*}
\mathbb{P}\left(N(\Phi_1(n) - x\Phi_2(n)) \ge \Phi_1(n)\right) &= \mathbb{P}\left(\xi(\Phi_1(n) - x\Phi_2(n)) \ge n\right)\\
&= \mathbb{P}\left(S_1(n) \le \Phi_1(n) - x\Phi_2(n)\right)
\end{align*}
Thus, by \eqref{M3}, for any $x \ge 0$,
\begin{equation}\label{N1}
\lim_{n \rightarrow \infty} \frac{1}{\Phi_2(n)} \log \mathbb{P}\left(N(\Phi_1(n) - x\Phi_2(n)) \ge \Phi_1(n)\right) = -\frac{x^2}{2}.
\end{equation}
For any $t \ge 0$, let $n_t$ be the unique non-negative integer such that $t \in [\Phi_1(n_t), \Phi_1(n_t + 1))$. As $\Phi_2(\infty) = \infty$, for any $\epsilon>0$ there exists $t_{\epsilon}>0$ such that for all $t \ge t_{\epsilon}$, $\Phi_1(n_t) + \epsilon \Phi_2(n_t) > t$.
Thus, for any $x > 0$, $t \ge t_{\epsilon}$,
\begin{multline*}
\mathbb{P}\left(N(t - x\K(t)) \ge t\right) \le \mathbb{P}\left(N(t - x\Phi_2(n_t)) \ge t\right)\\
 = \mathbb{P}\left(N(t - \epsilon\Phi_2(n_t) - (x-\epsilon)\Phi_2(n_t)) \ge t\right) \le \mathbb{P}\left(N(\Phi_1(n_t) - (x-\epsilon)\Phi_2(n_t)) \ge t\right).
\end{multline*}
Also, note that
$$
\limsup_{t \rightarrow \infty} \frac{\Phi_2(n_t)}{\K(t)} =  \limsup_{t \rightarrow \infty} \frac{\Phi_2(n_t)}{\Phi_2(\Phi_1^{-1}(t))} \le 1.
$$
Hence, if $x>0$ and $\epsilon \in (0,x)$, then using the above bounds along with \eqref{N1},
\begin{multline*}
\limsup_{t \rightarrow \infty} \frac{1}{\K(t)} \log \mathbb{P}\left(N(t - x\K(t)) \ge t\right)\\
\le \limsup_{t \rightarrow \infty} \frac{1}{\Phi_2(n_t)} \log \mathbb{P}\left(N(\Phi_1(n_t) - (x-\epsilon)\Phi_2(n_t)) \ge t\right) = - \frac{(x-\epsilon)^2}{2}.
\end{multline*}
As $\epsilon>0$ is arbitrary, we conclude that for any $x \ge 0$ (the $x=0$ case being trivially true),
\begin{equation}\label{N2}
\limsup_{t \rightarrow \infty} \frac{1}{\K(t)} \log \mathbb{P}\left(N(t - x\K(t)) \ge t\right) \le -\frac{x^2}{2}.
\end{equation}
Similarly, for any $\epsilon>0$, there exists $t_{\epsilon}' >0$ such that for all $t \ge t_{\epsilon}' $, $t + \epsilon\Phi_2(n_t + 1)> \Phi_1(n_t + 1)$. Thus, for $x \ge 0$, $t \ge t_{\epsilon}'$,
\begin{align*}
\mathbb{P}\left(N(t - x\K(t)) \ge t\right) &\ge \mathbb{P}\left(N(t - x\Phi_2(n_t+1)) \ge t\right)\\
 &= \mathbb{P}\left(N(t + \epsilon\Phi_2(n_t + 1) - (x+\epsilon)\Phi_2(n_t+1)) \ge t\right)\\
 &\ge \mathbb{P}\left(N(\Phi_1(n_t + 1) - (x+\epsilon)\Phi_2(n_t+1)) \ge t\right).
\end{align*}
Also, note that
$$
\liminf_{t \rightarrow \infty} \frac{\Phi_2(n_t+1)}{\K(t)} = \liminf_{t \rightarrow \infty} \frac{\Phi_2(n_t+1)}{\Phi_2(\Phi_1^{-1}(t))} \ge 1.
$$
Hence, using \eqref{N1},
\begin{multline*}
\liminf_{t \rightarrow \infty} \frac{1}{\K(t)} \log \mathbb{P}\left(N(t - x\K(t)) \ge t\right)\\
\ge \liminf_{t \rightarrow \infty} \frac{1}{\Phi_2(n_t+1)} \log  \mathbb{P}\left(N(\Phi_1(n_t + 1) - (x+\epsilon)\Phi_2(n_t+1)) \ge t\right) = - \frac{(x + \epsilon)^2}{2}.
\end{multline*}
Again, as $\epsilon>0$ is arbitrary, we conclude that for any $x \ge 0$,
\begin{equation}\label{N3}
\liminf_{t \rightarrow \infty} \frac{1}{\K(t)} \log \mathbb{P}\left(N(t - x\K(t)) \ge t\right) \ge -\frac{x^2}{2}.
\end{equation}
The lemma follows from \eqref{N2} and \eqref{N3}.
\end{proof}

\subsection{Controlling maximum degrees in CTBP }Recall the embedding of $\{\G_n\}_{n \ge 0}$ in a continuous time branching process (CTBP) when $m_i =1$ for all $i \ge 1$ (see Section \ref{ctbp}) and recall the Malthusian rate $\lambda_*$ defined in Section \ref{sec:not}. For the rest of this section, we will phrase our results in terms of this continuous time embedding. They will be used to address the original discrete time model in Section \ref{sec:main-proofs}. 

For any $0 \le a < b$, denote by $n[a,b]$ the number of individuals born during the time interval $[a,b]$. 
The next lemma gives control in probability on the number of individuals born in the time interval $[at,bt]$ for $0 \le a < b$, $t \ge 0$.

\begin{lemma}\label{birthcontrol}
Assume that the attachment function $f$ satisfies Assumption \eqref{eqn:prop-under-lamb} and $f(k) \rightarrow \infty$ as $k \rightarrow \infty$. Then
\begin{equation}\label{bca}
\lim_{A \rightarrow \infty} \mathbb{P}\left(\sup_{t \ge 0}e^{-\lambda_* t}n[0, t] \le A\right) = 1,
\end{equation}
and for any $0 \le a < b$ and any $\delta \in (0, \lambda_*(b-a)/(b+a))$,
\begin{equation}\label{bcb}
\lim_{t \rightarrow \infty} \mathbb{P}\left(n[at, bt] \le e^{(\lambda_* - \delta)bt} \right) = 0.
\end{equation}
\end{lemma}

\begin{proof}
Note that, by Assumption \eqref{eqn:prop-under-lamb}, there exists $\beta < \lambda_*$ such that 
$$
\hat{\rho}(\beta) = \int_0^{\infty} e^{-\beta t} \mu(dt) < \infty,
$$
and hence, equation (5.4) of \cite{nerman1981convergence} is satisfied with the function $g$ there chosen to be $g(t) = e^{-(\lambda_* - \beta)t}$. Hence, the point process $\xi$ satisfies Condition 5.1 of \cite{nerman1981convergence}. Moreover, choosing the characteristic $\phi(t) = \mathbb{I}(t \ge 0), t \ge 0$, Condition 5.2 of \cite{nerman1981convergence} is trivially satisfied. Therefore, by Theorem 5.4 of \cite{nerman1981convergence}, we have that
$
e^{-\lambda_*t} |\BP(t)|
$
converges almost surely to a finite random variable $W$. For $\lambda \ge 0$, write $\hat{\xi}(\lambda) := \int_0^{\infty}e^{-\lambda s}\xi(ds)$. Since $f$ satisfies Assumption \eqref{eqn:prop-under-lamb} and $f(k) \rightarrow \infty$ as $k \rightarrow \infty$, by Lemma 1 of \cite{rudas2007random},
\begin{equation}\label{laplace}
\mathbb{E}\left(\hat{\xi}(\lambda_*) \log^+ \hat{\xi}(\lambda_*) \right) < \infty.
\end{equation}
Using this, and noting that $|\BP(t)| \rightarrow \infty $ almost surely as $t \rightarrow \infty$, Propositions 1.1 and 2.2 of \cite{nerman1981convergence} imply that $m := \sup_{t \ge 0} \mathbb{E}(e^{-\lambda_*t} |\BP(t)|) < \infty$ and $W>0$ almost surely. Therefore, for $A >0$,
\begin{align}\label{bc1}
\mathbb{P}\left(\sup_{t \in [0, \log A/(2\lambda_*)]}e^{-\lambda_* t}n[0, t] > A\right) &\le \mathbb{P}\left(n[0, \log A/(2\lambda_*)] > A\right)\\
& \le \frac{\mathbb{E}\left(\BP(\log A/(2\lambda_*))\right)}{A} \le m A^{-1/2}.\nonumber
\end{align}
Fix any $\epsilon>0$. By the almost sure convergence, there exists $A_0 > 1$ such that for all $A \ge A_0$,
$$
\mathbb{P}\left(\sup_{t \ge \log A/(2\lambda_*)}|e^{-\lambda_* t}n[0, t] - W| > 1\right) < \epsilon.
$$
and
$
\mathbb{P}(W > A - 1) < \epsilon.
$
Therefore, for all $A \ge A_0$,
\begin{equation}\label{bc2}
\mathbb{P}\left(\sup_{t \ge \log A/(2\lambda_*)}e^{-\lambda_* t}n[0, t] > A\right) \le \mathbb{P}\left(\sup_{t \ge \log A/(2\lambda_*)}e^{-\lambda_* t}n[0, t] - W > 1\right) + \mathbb{P}(W > A - 1) < 2\epsilon.
\end{equation}
As $\epsilon>0$ is arbitrary, \eqref{bca} follows from \eqref{bc1} and \eqref{bc2}.

 Note that for any $\delta >0$ and $0 \le a < b$,
 $$
 \mathbb{P}\left(n[at, bt] \le e^{(\lambda_* - \delta)bt} \right) \le  \mathbb{P}\left(n[0, at] \ge e^{(\lambda_* - \delta)bt} \right) +  \mathbb{P}\left(n[0, bt] \le 2e^{(\lambda_* - \delta)bt} \right).
 $$
 Taking $t$ sufficiently large that $e^{\delta bt/2} >2$, we obtain
 $$
 \mathbb{P}\left(n[0, bt] \le 2e^{(\lambda_* - \delta)bt} \right) \le \mathbb{P}\left(n[0, bt] \le e^{(\lambda_* - \delta/2)bt} \right) = \mathbb{P}\left(e^{-\lambda_* bt}n[0, bt] \le e^{- \delta bt/2} \right) \rightarrow 0
 $$
as $t \rightarrow \infty$, because $e^{-\lambda_* bt}n[0, bt]$ converges almost surely to $W$ and $W>0$ almost surely. If $a=0$, noting $n[0,0]=1$, for any $\delta \in (0, \lambda_*)$,
$
\mathbb{P}\left(n[0, at] \ge e^{(\lambda_* - \delta)bt} \right) \rightarrow 0
$
as $t \rightarrow \infty$.
If $a>0$, choosing $\delta \in (0, \lambda_*(b-a)/(b+a))$,
$$
\mathbb{P}\left(n[0, at] \ge e^{(\lambda_* - \delta)bt} \right) \le \mathbb{P}\left(n[0, at] \ge e^{(\lambda_* + \delta)at} \right) = \mathbb{P}\left(e^{-\lambda_* a t}n[0, at] \ge e^{\delta at} \right) \rightarrow 0
$$
as $t \rightarrow \infty$, because $e^{-\lambda_* at}n[0, at]$ converges almost surely to $W$ and $W < \infty$ almost surely. Thus, \eqref{bcb} follows.
\end{proof}

For $i \in \mathbb{N}_0$, denote the birth time of the $i$-th individual born in the branching process by $B^{(i)}$ (the $0$-th individual being the root with $B^{(0)}=0$). For $i \in \mathbb{N}_0$, $s \ge 0$, define the \textbf{$\mathbf{N}$-degree} of the $i$-th individual at time $s + B^{(i)}$, denoted by $D^{(i)}(s)$, as $\Phi_1(\xi^{(i)}(s))$, where $\xi^{(i)}(\cdot)$ is the point process denoting times at which the $i$-th individual gives birth, measured relative to its own birth time (i.e. $\xi^{(i)}(s')=k$ for some $s' \ge 0$, $k \in \mathbb{N}_0$, if the $i$-th individual has $k$ children at time $s' + B^{(i)}$). The name $N$-degree signifies the fact that the degree (number of children) of individual $i$ is measured via a function of the point process $\xi^{(i)}$ (recall $N(t) := \Phi_1(\xi(t)), t \ge 0$). Note that $\{D^{(i)}(\cdot) : i \in \mathbb{N}_0\}$ are i.i.d. processes. 

For $a \ge 0$, $s \ge 0$, $i \in \mathbb{N}$, let $B^{(i)}_{a}$ be the birth time of the $i$-th individual born at or after time $a$, and let $D^{(i)}_{a}(s)$ be the $N$-degree of this individual at time $s + B^{(i)}_{a}$. Again, note that for any $a \ge 0$, $\{D^{(i)}_a(\cdot) : i \in \mathbb{N}\}$ are i.i.d. processes, each having the same distribution as $D^{(0)}(\cdot)$.

For any $0 \le a < b < c$, $n[a,b]>0$, define 
\begin{equation}\label{dmaxdef}
D^{max}_{a,b}(c) := \sup\{D^{(i)}_a(c- B^{(i)}_a) : i \in \mathbb{N} \text{ such that } B^{(i)}_a \le b\}.
\end{equation}
We take the convention $D^{max}_{a,b}(c) = 0$ if $c \le b$ or $n[a,b]=0$. 

The following lemma plays a similar role for our model as that played by Lemma 5.4 in \cite{DM}. It furnishes precise asymptotics for the maximum $N$-degree of all the individuals born in appropriately scaled time intervals, under a `continuity assumption' on $\K(\cdot)$.

\begin{lemma}\label{MDPsmall}
Assume $\Phi_2(\infty) = \infty$, $f$ satisfies Assumption \eqref{eqn:prop-under-lamb} and $f(k) \rightarrow \infty$ as $k \rightarrow \infty$. Also assume that the following holds:
\begin{equation}\label{Kcont}
\lim_{\delta \downarrow 0} \limsup_{t \rightarrow \infty} \frac{\K((1+\delta)t)}{\K(t)} = 1.
\end{equation}
Then for any $0 \le a < b$ and any $\eta \in (0, \sqrt{2\lambda_*b}/2)$,
\begin{align*}
&\mathbb{P}\left(D^{max}_{a\K(t),b\K(t)}\left(t- \eta\K(t)\right) \ge t- b\K(t) + (\sqrt{2\lambda_*b} - 2\eta)\K(t)\right) \rightarrow 1,\\
&\mathbb{P}\left(D^{max}_{a\K(t),b\K(t)}\left(t + \eta\K(t)\right) \le t- a\K(t) + (\sqrt{2\lambda_*b} + 2\eta)\K(t)\right) \rightarrow 1,
\end{align*}
as $t \rightarrow \infty$.
\end{lemma}

\begin{proof}
Fix $\eta \in (0, \sqrt{2\lambda_*b}/2)$. Choose and fix $\epsilon \in (0,1)$ and $\delta>0$ small enough such that $\delta \in (0, \lambda_*(b-a)/(b+a))$ and the following hold:
\begin{equation}\label{choice}
\frac{\sqrt{2(\lambda_* - \delta) b}}{(1+\epsilon)^2} > \sqrt{2\lambda_* b} - \eta, \ \ \frac{\sqrt{2(\lambda_* + \delta)b}}{(1-\epsilon)^2}< \sqrt{2\lambda_* b} + \eta.
\end{equation}
 For each $i \in \mathbb{N}$, and $z \in \mathbb{R}$,
\begin{multline*}
\mathbb{P}\left(D^{(i)}_{a\K(t)}\left(t - b\K(t) -\eta\K(t)\right) < t - b\K(t) +  z\K(t)\right)\\
= \mathbb{P}\left(N\left(t - b\K(t) - \eta\K(t)\right) < t - b\K(t) +  z\K(t)\right).
\end{multline*}
Writing $\beta(t) := t - b\K(t) +  z\K(t)$, we obtain
\begin{multline}\label{ms1}
\mathbb{P}\left(D^{(i)}_{a\K(t)}\left(t - b\K(t) -\eta\K(t)\right) < t - b\K(t) +  z\K(t)\right)\\
 = \mathbb{P}\left(N(\beta(t) - (z + \eta)\K(t)) < \beta(t)\right).
\end{multline}
As $\Phi_2(\infty) = \infty$ and $f(k) \rightarrow \infty$ as $k \rightarrow \infty$, a calculation similar to \eqref{re} implies $\lim_{n \rightarrow \infty} \frac{\Phi_2(n)}{\Phi_1(n)} = 0$ and hence
$$
\lim_{t \rightarrow \infty} \frac{\K(t)}{t} = \lim_{t \rightarrow \infty} \frac{\Phi_2(\Phi_1^{-1}(t))}{\Phi_1(\Phi_1^{-1}(t))} = 0.
$$
Thus, $\lim_{t \rightarrow \infty} \frac{\beta(t)}{t} = 1$.
By Assumption \eqref{Kcont}, there exists $\delta_0>0, t_0>0$ such that for any $t \ge t_0$,
\begin{equation*}
 \frac{\K((1+\delta_0)t)}{\K(t)} < (1+ \epsilon).
\end{equation*}
Take $t_1 \ge t_0$ such that
$
\frac{1}{1+\delta_0} \le \frac{t}{\beta(t)} \le (1+\delta_0)
$
and $\beta(t) \ge t_0$
for all $t \ge t_1$. Therefore, for all $t \ge t_1$,
$$
\frac{\K(t)}{\K(\beta(t))} \le \frac{\K((1+\delta_0)\beta(t))}{\K(\beta(t))} < (1+ \epsilon),
$$
and
$$
\frac{\K(\beta(t))}{\K(t)} \le \frac{\K((1+\delta_0)t)}{\K(t)} < (1+ \epsilon).
$$
Using this in \eqref{ms1}, we obtain for all $t \ge t_1$,
\begin{multline*}
\mathbb{P}\left(D^{(i)}_{a\K(t)}\left(t - b\K(t) -\eta\K(t)\right) < t - b\K(t) +  z\K(t)\right)\\
\le \mathbb{P}\left(N(\beta(t) - (z+\eta)(1+\epsilon)\K(\beta(t))) < \beta(t)\right).
\end{multline*}
By Lemma \ref{MDPN}, for any $z \ge -\eta/2$, we can choose $t_2(z) \ge t_1$ (possibly depending on $z$) such that for all $t \ge t_2(z)$,
\begin{multline*}
\mathbb{P}\left(N(\beta(t) - (z+\eta)(1+\epsilon)\K(\beta(t))) < \beta(t)\right) \le 1 - \exp\left\lbrace -\frac{(1+\epsilon)^3(z+\eta)^2 \K(\beta(t))}{2} \right\rbrace\\
\le 1 - \exp\left\lbrace -\frac{(1+\epsilon)^4(z+\eta)^2 \K(t)}{2} \right\rbrace.
\end{multline*}
Thus, from the above two bounds, for any $i \in \mathbb{N}$, any $z \ge -\eta/2$, $t \ge t_2(z)$,
\begin{multline}\label{ms2}
\mathbb{P}\left(D^{(i)}_{a\K(t)}\left(t - b\K(t) -\eta\K(t)\right) < t - b\K(t) +  z\K(t)\right)\\
 \le 1 - \exp\left\lbrace -\frac{(1+\epsilon)^4(z+\eta)^2 \K(t)}{2} \right\rbrace.
\end{multline}
Finally, recalling \eqref{dmaxdef}, we obtain, for any $z \ge -\eta/2$, $t \ge t_2(z)$,
\begin{align}\label{ms2.1}
&\mathbb{P}\left(D^{max}_{a\K(t),b\K(t)}(t - \eta \K(t)) < t- b\K(t) + z\K(t)\right)\\
&\le \mathbb{P}\left(n[a \K(t),b\K(t)] \le  e^{(\lambda_*-\delta) b\K(t)}\right)\notag\\
&\quad + \mathbb{P}\left(D^{max}_{a\K(t),b\K(t)}(t - \eta \K(t)) < t- b\K(t) + z\K(t), \ n[a \K(t),b\K(t)] >  e^{(\lambda_*-\delta) b\K(t)}\right)\notag\\
&\le \mathbb{P}\left(n[a \K(t),b\K(t)] \le  e^{(\lambda_*-\delta) b\K(t)}\right)\notag\\
&\quad + \mathbb{P}\left(D^{(i)}_{a\K(t)}(t - \eta \K(t) - B^{(i)}_{a\K(t)}) < t- b\K(t) + z\K(t)\right.\notag\\ 
&\qquad \qquad \left. \text{ and } B^{(i)}_{a\K(t)} \le b\K(t)  \text{ for all }  i \le \lfloor e^{(\lambda_*-\delta) b\K(t)}\rfloor + 1\right)\notag\\
&\le \mathbb{P}\left(n[a \K(t),b\K(t)] \le  e^{(\lambda_*-\delta) b\K(t)}\right)\notag\\
&\quad + \mathbb{P}\left(D^{(i)}_{a\K(t)}(t - \eta \K(t) - b\K(t)) < t- b\K(t) + z\K(t) \text{ for all }  i \le \lfloor e^{(\lambda_*-\delta) b\K(t)}\rfloor + 1\right)\notag\\
&\le \mathbb{P}\left(n[a \K(t),b\K(t)] \le  e^{(\lambda_*-\delta) b\K(t)}\right) + \left(1 - \exp\left\lbrace -\frac{(1+\epsilon)^4(z+\eta)^2 \K(t)}{2} \right\rbrace\right)^{e^{(\lambda_*-\delta) b\K(t)}},\notag
\end{align}
where we used the independence of $\{D^{(i)}_a(\cdot) : i \in \mathbb{N}\}$ in the last step.

As $\delta \in (0, \lambda_*(b-a)/(b+a))$, by \eqref{bcb},
\begin{equation}\label{ms2.5}
\lim_{t \rightarrow \infty}\mathbb{P}\left(n[a \K(t),b\K(t)] \le  e^{(\lambda_*-\delta) b\K(t)}\right) = 0.
\end{equation}
Moreover, for any $-\eta/2 \le z < \frac{\sqrt{2(\lambda_* - \delta) b}}{(1+\epsilon)^2} - \eta$,
\begin{multline*}
e^{(\lambda_*-\delta) b\K(t)}\log  \left(1 - \exp\left\lbrace -\frac{(1+\epsilon)^4(z+\eta)^2 \K(t)}{2} \right\rbrace\right)\\
\le -\exp\left\lbrace (\lambda_*-\delta) b\K(t) -\frac{(1+\epsilon)^4(z+\eta)^2 \K(t)}{2} \right\rbrace
 \rightarrow -\infty \ \text{ as } \ t \rightarrow \infty
\end{multline*}
and hence, for any $-\eta/2 \le z < \frac{\sqrt{2(\lambda_* - \delta) b}}{(1+\epsilon)^2} - \eta$,
\begin{equation}\label{ms3}
\lim_{t \rightarrow \infty}\left(1 - \exp\left\lbrace -\frac{(1+\epsilon)^4(z+\eta)^2 \K(t)}{2} \right\rbrace\right)^{e^{(\lambda_*-\delta) b\K(t)}}= 0.
\end{equation}
As, by \eqref{choice}, $\frac{\sqrt{2(\lambda_* - \delta) b}}{(1+\epsilon)^2} - \eta > \sqrt{2\lambda_* b} - 2\eta \ge 0$ (as $\eta \in (0, \sqrt{2\lambda_*b}/2)$), choosing $z = \sqrt{2\lambda_* b} - 2\eta$ in \eqref{ms2.1} and using \eqref{ms2.5} and \eqref{ms3}, we obtain
\begin{equation}\label{ms4}
\limsup_{t \rightarrow \infty}\mathbb{P}\left(D^{max}_{a\K(t),b\K(t)}(t - \eta \K(t)) < t- b\K(t) + (\sqrt{2\lambda_* b} - 2\eta)\K(t)\right) = 0.
\end{equation}
Now, observe that for any $i \in \mathbb{N}$, $w \in \mathbb{R}$,
\begin{multline*}
\mathbb{P}\left(D^{(i)}_{a\K(t)}(t - a\K(t) + \eta \K(t)) > t - a\K(t) +  w\K(t)\right)\\
= \mathbb{P}\left(N(t - a\K(t) + \eta \K(t)) > t - a\K(t) +  w\K(t)\right).
\end{multline*}
Writing $\gamma(t) := t-a\K(t) + w\K(t)$, we obtain
\begin{multline*}
\mathbb{P}\left(D^{(i)}_{a\K(t)}(t - a\K(t) + \eta \K(t)) > t - a\K(t) +  w\K(t)\right)\\
 = \mathbb{P}\left(N(\gamma(t) - (w-\eta)\K(t)) > \gamma(t)\right).
\end{multline*}
Again note that $\frac{\gamma(t)}{t} \rightarrow 1$ as $t \rightarrow \infty$. As before, by Assumption \eqref{Kcont}, we can obtain $t_1'$ such that for all $t \ge t_1'$,
$$
\frac{\K(t)}{\K(\gamma(t))} > (1- \epsilon), \ \  \frac{\K(\gamma(t))}{\K(t)} > (1-\epsilon).
$$
Hence, for all $t \ge t_1'$,
\begin{multline*}
\mathbb{P}\left(D^{(i)}_{a\K(t)}(t - a\K(t) + \eta \K(t)) > t - a\K(t) +  w\K(t)\right)\\
\le \mathbb{P}\left(N(\gamma(t) - (w-\eta)(1-\epsilon)\K(\gamma(t)) > \gamma(t)\right).
\end{multline*}
By Lemma \ref{MDPN}, for any $w \ge 2\eta$, we can choose $t_2'(w)>0$ (possibly depending on $w$) such that for all $t \ge t_2'(w)$,
\begin{multline*}
\mathbb{P}\left(N(\gamma(t) -  (w-\eta)(1-\epsilon)\K(\gamma(t)) > \gamma(t)\right) \le \exp\left\lbrace -\frac{ (w-\eta)^2(1-\epsilon)^3\K(\gamma(t))}{2} \right\rbrace\\
 \le \exp\left\lbrace -\frac{ (w-\eta)^2(1-\epsilon)^4\K(t)}{2} \right\rbrace.
\end{multline*}
Thus, for any $w \ge 2\eta$, $t \ge t_2'(w)$,
\begin{align}\label{ms5}
&\mathbb{P}\left(D^{max}_{a\K(t),b\K(t)}(t + \eta \K(t)) > t- a\K(t) + w\K(t)\right)\\
&\le \mathbb{P}\left(n[0,b\K(t)] > e^{(\lambda_*+\delta) b\K(t)}\right)\notag\\
&\quad + \mathbb{P}\left(D^{max}_{a\K(t),b\K(t)}(t + \eta \K(t)) > t- a\K(t) + w\K(t), \ n[0,b\K(t)] \le e^{(\lambda_*+\delta) b\K(t)}\right)\notag\\
&\le \mathbb{P}\left(n[0,b\K(t)] > e^{(\lambda_*+\delta) b\K(t)}\right)\notag\\
&\quad + \mathbb{P}\left(D^{(i)}_{a\K(t)}(t + \eta \K(t) - B^{(i)}_{a\K(t)}) > t- a\K(t) + w\K(t) \text{ for some }  i \le e^{(\lambda_*+\delta) b\K(t)}\right)\notag\\
&\le \mathbb{P}\left(n[0,b\K(t)] > e^{(\lambda_*+\delta) b\K(t)}\right)\notag\\
&\quad + \mathbb{P}\left(D^{(i)}_{a\K(t)}(t + \eta \K(t) - a\K(t)) > t- a\K(t) + w\K(t) \text{ for some }  i \le e^{(\lambda_*+\delta) b\K(t)}\right)\notag\\
&\le \mathbb{P}\left(n[0,b\K(t)] > e^{(\lambda_*+\delta) b\K(t)}\right) + \exp\left\lbrace (\lambda_*+\delta) b\K(t) -\frac{ (w-\eta)^2(1-\epsilon)^4\K(t)}{2} \right\rbrace,\notag
\end{align}
where the last step follows from the union bound.
By \eqref{bca},
\begin{equation}\label{ms5.1}
\lim_{t \rightarrow \infty}\mathbb{P}\left(n[0,b\K(t)] >  e^{(\lambda_*+\delta) b\K(t)}\right) = 0.
\end{equation}
Moreover, for any $w > \frac{\sqrt{2(\lambda_* + \delta)b}}{(1-\epsilon)^2} + \eta$,
\begin{equation}\label{ms5.2}
\lim_{t \rightarrow \infty} \exp\left\lbrace (\lambda_*+\delta) b\K(t) -\frac{ (w-\eta)^2(1-\epsilon)^4K(t)}{2} \right\rbrace = 0.
\end{equation}
As, by \eqref{choice}, $\frac{\sqrt{2(\lambda_* + \delta)b}}{(1-\epsilon)^2} + \eta< \sqrt{2\lambda_* b} + 2\eta$, therefore, choosing $w = \sqrt{2\lambda_* b} + 2\eta$ in \eqref{ms5} and using \eqref{ms5.1} and \eqref{ms5.2}, we obtain
\begin{equation}\label{ms6}
\lim_{t \rightarrow \infty}\mathbb{P}\left(D^{max}_{a\K(t),b\K(t)}(t+\eta\K(t)) > t- a\K(t) + ( \sqrt{2\lambda_* b} + 2\eta)\K(t)\right) = 0.
\end{equation}
The lemma follows from \eqref{ms4} and \eqref{ms6}.
\end{proof}

The following lemma shows, under suitable assumptions on $\K(\cdot)$, that for large $t$, if an individual is born sufficiently late, it will not have the maximum $N$-degree at time $t + \K(t)$ with high probability.
\begin{lemma}\label{indmax}
Assume $\Phi_2(\infty) = \infty$. Moreover, assume that there exist positive constants $t', D$ such that $\K(3t) \le D\K(t)$ for all $t \ge t'$. Then there exists $A_0 >0$ such that
$$
\lim_{t \rightarrow \infty}\mathbb{P}\left(D^{max}_{A_0\K(t), t}(t + \K(t)) >t\right) = 0.
$$
\end{lemma}

\begin{proof}
By Lemma \ref{tbd}, there exists $t^*>0$ (depending on $t',D, f_*$) such that for any $t \ge t^*$, any $j \in [2,  1 + (2t/\K(t))]$ (assuming $t^*$ is large enough that $t \ge \K(t)$ for all $t \ge t^*$) and any $i \ge 1$,
\begin{align}\label{indmax1}
\mathbb{P}\left(D^{(i)}_{j\K(t)}(t + \K(t) - B^{(i)}_{j\K(t)}) > t \right) &\le \mathbb{P}\left(N(t + \K(t)-j\K(t)) > t\right)\\
&= \mathbb{P}\left(M(t-(j-1)\K(t)) > (j-1)\K(t)\right)\nonumber\\
& \le \exp\left\lbrace -\frac{(j-1)^2}{4D} \K(t)\right\rbrace.\nonumber
\end{align}
Take any $\epsilon>0$. Define the event
$$
\mathcal{E}_A := \{\sup_{t \ge 0} e^{-\lambda_* t} n[0,t] \le A\}.
$$
By \eqref{bca}, there exists $A_{\epsilon}>0$ such that
\begin{equation}\label{indmax2}
\mathbb{P}\left(\mathcal{E}_{A_{\epsilon}}\right) \ge 1 - \epsilon.
\end{equation}
Observe that if event $\mathcal{E}_{A_{\epsilon}}$ holds, then $n[j\K(t), (j+1)\K(t)] \le A_{\epsilon} e^{\lambda_* (j+1) \K(t)}$ for all $t \ge 0$, $j \ge 0$. Thus, using \eqref{indmax1} and the union bound, we obtain for $t \ge t^*$, any $j \in [2,  1 + (2t/\K(t))]$,
\begin{equation*}
\mathbb{P}\left(D^{max}_{j\K(t),(j+1)\K(t)}(t + \K(t)) > t, \ \mathcal{E}_{A_{\epsilon}}\right) \le  A_{\epsilon}\exp\left\lbrace \lambda_* (j+1) \K(t) -\frac{(j-1)^2}{4D} \K(t)\right\rbrace.
\end{equation*}
Thus, we obtain $j_0 \ge 2$ (depending on $\lambda_*, D$ but not $\epsilon$) such that for $j \ge j_0$, $j \le 1 + (2t/\K(t))$,
\begin{equation}\label{indmax3}
\mathbb{P}\left(D^{max}_{j\K(t),(j+1)\K(t)}(t + \K(t)) > t, \ \mathcal{E}_{A_{\epsilon}}\right) \le  A_{\epsilon}\exp\left\lbrace  -\frac{(j-1)^2}{8D} \K(t)\right\rbrace.
\end{equation}
From \eqref{indmax3} and the union bound, we obtain for any $t \ge t^*$,
\begin{multline}\label{indmax4}
\mathbb{P}\left(D^{max}_{j_0\K(t), t} (t + \K(t)) >t , \ \mathcal{E}_{A_{\epsilon}} \right) \le \sum_{j=j_0}^{1+ \lfloor t/\K(t)\rfloor}\mathbb{P}\left(D^{max}_{j\K(t),(j+1)\K(t)}(t + \K(t)) > t, \ \mathcal{E}_{A_{\epsilon}}\right)\\
\le \sum_{j=j_0}^{1+ \lfloor t/\K(t)\rfloor}A_{\epsilon}\exp\left\lbrace  -\frac{(j-1)^2}{8D} \K(t)\right\rbrace \le C_{\epsilon} \exp\left\lbrace  -\frac{(j_0-1)^2}{8D} \K(t)\right\rbrace,
\end{multline}
where $C_{\epsilon}$ depends on $\epsilon, D, t^*$. As $\Phi_2(\infty) = \infty$, $\K(t) \rightarrow \infty$ as $t \rightarrow \infty$. Hence, we obtain from \eqref{indmax2} and \eqref{indmax4},
$$
\limsup_{t \rightarrow \infty} \mathbb{P}\left(D^{max}_{j_0\K(t), t}(t + \K(t)) >t \right) \le \limsup_{t \rightarrow \infty}\mathbb{P}\left(D^{max}_{j_0\K(t), t}(t+\K(t)) >t , \ \mathcal{E}_{A_{\epsilon}} \right) + \mathbb{P}\left(\mathcal{E}_{A_{\epsilon}}^c\right) \le \epsilon.
$$
As $\epsilon>0$ is arbitrary and $j_0$ does not depend on $\epsilon$, the lemma follows with $A_0 = j_0$. 
\end{proof}

\section{Proofs of the main results}
\label{sec:main-proofs}
We now sequentially prove all the main results using the results in the previous section. 

\begin{proof}[Proof of Theorem \ref{genper}]
We work with the random graph sequence $\{\mathcal{G}^*_k\}_{k \ge 0}$. As the graph sequence is obtained conditionally on a realization of the attachment sequence $\{m_i\}_{i \ge 1}$ (see Section \ref{modelspec}), without loss of generality, we assume that $\{m_i\}_{i \ge 1}$ is non-random and satisfies hypothesis \eqref{edgegrow}.

We start by obtaining a lower bound on the degree of the root vertex for large $k$.
By Lemma \ref{coupling} (i), there exists a coupling of $d_0(\cdot), \underline{d}(\cdot)$ such that $d_0(i) \ge \underline{d}(i)$ for all $i \ge 0$. Note that the process $\{\underline{d}(k) : k \ge 0\}$ evolves exactly as the degree evolution of the root in \cite{DM} using attachment function $f(\cdot)/(3C_f)$, with $n$ there (number of vertices) playing the same role as $k+1$ (number of edges) here. Thus, by Proposition 1.4 of \cite{DM},
$$
\lim_{k \rightarrow \infty} \frac{\Phi_1(\underline{d}(k))}{\log k} = \frac{1}{3C_f}, \ \ \text{almost surely}.
$$
Hence,
\begin{equation}\label{gp1}
\liminf_{k \rightarrow \infty} \frac{\Phi_1(d_0(k))}{\log k} \ge \frac{1}{3C_f}, \ \ \text{almost surely}.
\end{equation}
Define the following events for $m,n \ge 1$:
\begin{equation}\label{amdef}
A_m := \left\lbrace \frac{\Phi_1(d_0(k))}{\log k} \ge \frac{1}{4C_f} \text{ for all } k \ge m\right\rbrace
\end{equation}
and
$$
E_n := \{d_{n}(k) \ge d_0(k) \text{ for some } k \ge s_{n}\}.
$$
As in Lemma \ref{embmult}, independently of $\{\mathcal{G}^*_k: k \ge 0\}$, sample independent point processes $\{\xi_A(\cdot) :A \in \mathbb{N}_0\}$, $\xi_{A}(\cdot)$ obtained using attachment function $f_A(\cdot)$ as in Section \ref{sec:not}, and define the martingale $M_A(\cdot)$ for each $A \in \mathbb{N}_0$ as in \eqref{martdef}.
Let $\xi_1^{(n)}(\cdot) := d_0(s_n) + \xi_{d_0(s_n)}(\cdot)$ and $\xi_2^{(n)}(\cdot) := m_n + \xi_{m_n}(\cdot)$. For $t \ge 0$, define
\begin{align*}
M^{(n)}_1(t) &:= \sum_{k=d_0(s_n)}^{\xi_1^{(n)}(t)-1}\frac{1}{f(k)} - t = \sum_{k=0}^{\xi_1^{(n)}(t) - d_0(s_n)-1}\frac{1}{f_{d_0(s_n)}(k)} - t ,\\
M^{(n)}_2(t) &:= \sum_{k=m_n}^{\xi_2^{(n)}(t) -1}\frac{1}{f(k)} - t = \sum_{k=0}^{\xi_2^{(n)}(t) - m_n -1}\frac{1}{f_{m_n}(k)} - t.
\end{align*}
For $n \ge m$, noting that $s_n \ge n$, on the event $A_m$, $\Phi_1(d_0(s_n)) \ge \frac{\log s_n}{4C_f}$. Moreover, by Lemma \ref{embmult}, the Markov chains obtained by observing the continuous time process $\{(\xi_1^{(n)}(t), \xi_2^{(n)}(t)) : t \ge 0\}$ and the discrete time process $\{(d_0(k), d_n(k)) : k \ge s_n\}$ at their respective jump times have the same distribution and hence, we obtain for $n \ge m$,
\begin{multline*}
\mathbb{P}(E_n \cap A_m)\\
\le \mathbb{P}\left(M^{(n)}_2(t) + \Phi_1(m_n) \ge M^{(n)}_1(t) + \Phi_1(d_0(s_n)) \text{ for some } t \ge 0, \ \Phi_1(d_0(s_n)) \ge \frac{\log s_n}{4C_f}\right).
\end{multline*}
Note that $M^{(n)}_1(\cdot)$ has the same distribution as $M_{d_0(s_n)}(\cdot)$ and $M^{(n)}_2(\cdot)$ has the same distribution as $M_{m_n}(\cdot)$. 
Thus,
\begin{align}\label{gp3}
\mathbb{P}(E_n \cap A_m) &\le \mathbb{P}\left(\inf_{s<\infty} M_{d_0(s_n)}(s) \le -\frac{\Phi_1(d_0(s_n)) - \Phi_1(m_n)}{2}, \ \Phi_1(d_0(s_n)) \ge \frac{\log s_n}{4C_f}\right)\\
& \quad + \mathbb{P}\left(\sup_{s<\infty} M_{m_n}(s) \ge \frac{\Phi_1(d_0(s_n)) - \Phi_1(m_n)}{2}, \ \Phi_1(d_0(s_n)) \ge \frac{\log s_n}{4C_f}\right).\nonumber
\end{align}
By the hypothesis \eqref{edgegrow}, there exists $n_0 \ge m$ such that for all $n \ge n_0$, $\Phi_1(m_n) \le \frac{\log s_n}{6C_f}$. Therefore, for all $n \ge n_0$, from \eqref{gp3},
\begin{align}\label{gp4}
\mathbb{P}(E_n \cap A_m) &\le \mathbb{P}\left(\inf_{s<\infty} M_{d_0(s_n)}(s) \le -\frac{\log s_n}{24C_f}, \ \Phi_1(d_0(s_n)) \ge \frac{\log s_n}{4C_f}\right)\\
&\quad + \mathbb{P}\left(\sup_{s<\infty} M_{m_n}(s) \ge \frac{\log s_n}{24C_f}, \ \Phi_1(d_0(s_n)) \ge \frac{\log s_n}{4C_f}\right).\notag
\end{align}
Assume $n_0$ is large enough so that $\frac{\log s_n}{24C_f} \ge \max\{x_0, x_0''\}$, where $x_0$ and $x_0''$ are constants defined in Lemma \ref{conc}. By \eqref{concfin} in Lemma \ref{conc}, there exist positive constants $C, C'$ such that for all $n \ge n_0$,
\begin{align}\label{slu1}
\mathbb{P}&\left(\inf_{s<\infty} M_{d_0(s_n)}(s) \le -\frac{\log s_n}{24C_f}, \ \Phi_1(d_0(s_n)) \ge \frac{\log s_n}{4C_f}\right)\\
&\le  \mathbb{E}\left[\left(C\exp\left\lbrace-C' (\log s_n)^2\right\rbrace \right.\right.\notag\\
&\qquad \left.\left. + C\exp\left\lbrace -C'\sqrt{f_*(d_0(s_n))} (\log s_n)\right\rbrace\right) \mathbb{I}\left(\Phi_1(d_0(s_n)) \ge \frac{\log s_n}{4C_f}\right)\right]\notag\\
&\le C\exp\left\lbrace-C' (\log s_n)^2\right\rbrace  + C\exp\left\lbrace -C'\sqrt{f_*\left(\left\lfloor\Phi_1^{-1}\left(\frac{\log s_n}{4C_f}\right)\right\rfloor\right)} (\log s_n)\right\rbrace\notag\\
&\le C\exp\left\lbrace-C' (\log n)^2\right\rbrace  + C\exp\left\lbrace -C'\sqrt{f_*\left(\left\lfloor\Phi_1^{-1}\left(\frac{\log n}{4C_f}\right)\right\rfloor\right)} (\log n)\right\rbrace := a_n,\notag
\end{align}
and by \eqref{concfin2},
\begin{multline}\label{slu2}
\mathbb{P}\left(\sup_{s<\infty} M_{m_n}(s) \ge \frac{\log s_n}{24C_f}, \ \Phi_1(d_0(s_n)) \ge \frac{\log s_n}{4C_f}\right) \le C\exp\left\lbrace-C' (\log s_n)^2\right\rbrace\\
\le C\exp\left\lbrace-C' (\log n)^2\right\rbrace := b_n.
\end{multline}
Using the above estimates in \eqref{gp4}, we obtain
$$
\mathbb{P}(E_n \cap A_m) \le a_n + b_n \ \text{ for all } n \ge \max\{n_0,m\}.
$$
As $\Phi_2(\infty)< \infty$, $f_*\left(\left\lfloor\Phi_1^{-1}\left(\frac{\log n}{4C_f}\right)\right\rfloor\right) \rightarrow \infty$ as $n \rightarrow \infty$. Hence, the sequence $\{a_n + b_n\}_{n \ge n_0}$ is summable. Thus, by Borel-Cantelli lemma, almost surely, for any $m \ge 1$, on the event $A_m$, there exists $n_1$ (possibly depending on $m$) such that for any $n \ge n_1$, the maximum degree vertex in $\mathcal{G}^*_{s_n}$ is one of the first $n_1$ vertices. Thus, by Corollary \ref{distlim}, for any $m \ge 1$, a persistent hub emerges almost surely on the event $A_m$.

The theorem now follows upon noting that $\mathbb{P}\left(A_m\right) \rightarrow 1$ as $m \rightarrow \infty$.
\end{proof}

\begin{proof}[Proof of Theorem \ref{perfail}]
The weak convergence assertion is a direct consequence of Lemma \ref{fcltper}. To prove the second assertion, fix any  $A \in \mathbb{N}_0$. Conditional on $\{\mathcal{G}^*_k : k \le s_A\}$, sample two (conditionally) independent point processes $\xi_{d_0(s_A)}(\cdot)$ and $\xi_{m_A}(\cdot)$ corresponding to attachment functions $f_{d_0(s_A)}(\cdot)$ and $f_{m_A}(\cdot)$ as in Section \ref{sec:not}.

For $t \ge 0$, define
\begin{align*}
\hat M_1(t) &:= \sum_{k=d_0(s_A)}^{\xi_{d_0(s_A)}(t) + d_0(s_A) -1}\frac{1}{f(k)} - t = \sum_{j=0}^{\xi_{d_0(s_A)}(t) -1}\frac{1}{f_{d_0(s_A)}(k)} - t ,\\
\hat M_2(t) &:= \sum_{k=m_A}^{\xi_{m_A}(t) + m_A -1}\frac{1}{f(k)} - t = \sum_{k=0}^{\xi_{m_A}(t)  -1}\frac{1}{f_{m_A}(k)} - t.
\end{align*}
Write
$$
W(t)  := (\Phi_1(d_0(s_A)) + \hat M_1(t)) - (\Phi_1(m_A) + \hat M_2(t)), \ t \ge 0.
$$
In the following, we will say that a function $G: [0, \infty) \rightarrow \mathbb{R}$ changes sign $k$-times for some $k \in \mathbb{N}$ if there exist $0 < t_1 < \dots < t_{k+1} < \infty$ such that
\begin{equation}\label{chsgn}
G(t_i) G(t_{i+1}) < 0
\end{equation}
for all $1 \le i \le k$. We say that $G$ changes sign infinitely often if we can obtain an infinite sequence $\{t_i\}_{i \ge 1}$ of positive numbers such that \eqref{chsgn} holds for all $i \ge 1$.

By Lemma \ref{embmult}, 
\begin{multline}\label{perf1}
\mathbb{P}\left(d_0(\cdot) - d_A(\cdot) \text{ changes sign infinitely often}\right)
= \mathbb{P}\left(W(\cdot) \text{ changes sign infinitely often}\right).
\end{multline}
Conditional on $\{\mathcal{G}^*_k : k \le s_A\}$, $\hat M_1(\cdot)$ and $\hat M_2(\cdot)$ respectively have the same laws as $M_{d_0(s_A)}(\cdot)$ and $M_{m_A}(\cdot)$. Thus, by the functional central limit theorem proved in the first part of the theorem, 
$
\{(2n)^{-1/2}W \left(\K^{-1}(nt)\right), \ t \ge 0\}
$
converges weakly to a standard Brownian motion in $D([0,\infty)  :  \mathbb{R})$ as $n \rightarrow \infty$. By the Skorohod representation theorem \cite[Theorem 6.7]{billingsley2013convergence}, there exist $D([0,\infty)  :  \mathbb{R})$-valued random variables $W^{(n)}(\cdot)$ and $W^*(\cdot)$ defined on a common probability space such that $W^{(n)}(\cdot)$ has the same distribution as $(2n)^{-1/2}W \left(\K^{-1}(n \ \cdot)\right)$ for each $n \ge 1$, $W^*(\cdot)$ is a standard Brownian motion and $W^{(n)}(\cdot)$ converges almost surely to $W^*(\cdot)$ with respect to the Skorohod metric (see \cite[Equation (12.13), Pg 124]{billingsley2013convergence}). As $W^*(\cdot)$ is almost surely continuous, this convergence can be strengthened to almost sure uniform convergence on compact subsets of $[0, \infty)$ (see \cite[Ch. 3, Pg 124]{billingsley2013convergence}). As, almost surely, $W^*(\cdot)$ changes sign infinitely often, for any $k \in \mathbb{N}$, by this uniform convergence
$$
\mathbb{P}\left(W^{(n)}(\cdot) \text{ changes sign at least } k \text{ times}\right) \rightarrow 1
$$
as $n \rightarrow \infty$. As $W^{(n)}(\cdot)$ has the same distribution as $(2n)^{-1/2}W \left(\K^{-1}(n \ \cdot)\right)$ and $\K(\cdot)$ is strictly increasing, for any $k,n \in \mathbb{N}$,
$$
\mathbb{P}\left(W(\cdot) \text{ changes sign at least } k \text{ times}\right) = \mathbb{P}\left(W^{(n)}(\cdot) \text{ changes sign at least } k \text{ times}\right).
$$
From the above two observations, 
$$
\mathbb{P}\left(W(\cdot) \text{ changes sign at least } k \text{ times}\right) = 1.
$$
Hence, almost surely, $W(\cdot)$, and by \eqref{perf1}, $d_0(\cdot) - d_A(\cdot)$, changes sign infinitely often implying that a persistent hub does not emerge.
\end{proof}

\begin{proof}[Proof of Theorem \ref{iid}]
(i) By the Borel-Cantelli Lemma, \eqref{light} implies that  there exists a (deterministic) constant $C>0$ such that, almost surely, $\Phi_1(m_n) \le C(\log n)^{\frac{1}{\theta}}$ for all $n \in \mathbb{N}$. Thus, almost surely,
$$
\limsup_{n \rightarrow \infty} \frac{\Phi_1(m_{n})}{\log s_n} \le \limsup_{n \rightarrow \infty} \frac{C(\log n)^{\frac{1}{\theta}}}{\log s_n} \le \limsup_{n \rightarrow \infty} \frac{C(\log n)^{\frac{1}{\theta}}}{\log n} = 0.
$$
Persistence is now a direct consequence of Theorem \ref{genper}.

(ii) Define for $\epsilon>0$,
$
\Omega_{\epsilon} : = \{\{m_i\}_{i \ge 1}  :  s^{-1}(k) \ge \epsilon (k+1) \text{ for all } k \in \mathbb{N}\}.
$
For any $l \ge 0$, by Lemma \ref{coupling} (ii), the process $\{d_l(k) : k \ge s_{l-1}\}$ is stochastically dominated by the process $\{\overline{d}_l^{\epsilon}(k) : k \ge s_{l-1}\}$ (where $\overline{d}_l^{\epsilon}(\cdot)$ is defined in the lemma) on the event $\Omega_{\epsilon}$. Combining this observation with Proposition 1.4 of \cite{DM}, we obtain on $\Omega_{\epsilon}$,
$$
\lim_{k \rightarrow \infty} \frac{\Phi_1(d_l(k))}{\log k} \le \lim_{k \rightarrow \infty} \frac{\Phi_1(\overline{d}_l^{(\epsilon)}(k))}{\log k} = \frac{1}{\epsilon f(0)}, \ \ \text{almost surely}.
$$
Thus, almost surely on $\Omega_{\epsilon}$, $\Phi_1(d_l(k)) \le \frac{2\log k}{\epsilon f(0)}$ for all sufficiently large $k$.
However, by the (converse) Borel-Cantelli Lemma, \eqref{heavy} implies that there is a (deterministic) constant $C'>0$ such that, almost surely, there exist infinitely many $n \ge 1$ such that $\Phi_1(m_n) \ge C'(\log n)^{\frac{1}{\theta'}}$. Call the set of such $n$ as $\mathcal{S}$. By the strong law of large numbers, almost surely, $s_n = \sum_{i=1}^n m_i < 2\mathbb{E}(m_1) n$ for all sufficiently large $n$. Thus, almost surely on $\Omega_{\epsilon}$, for all sufficiently large $n \ge l$ such that $n \in \mathcal{S}$,
$$
\Phi_1(d_n(s_n)) = \Phi_1(m_n) \ge C'(\log n)^{\frac{1}{\theta'}} > \frac{2\log (2\mathbb{E}(m_1) n)}{\epsilon f(0)} > \frac{2\log s_n}{\epsilon f(0)} \ge \Phi_1(d_l(s_n)),
$$
and as $\Phi_1(\cdot)$ is strictly increasing, for any such $n$, we have $d_{n}(s_n) > d_l(s_n)$. Hence, almost surely, there is no persistent hub on the event $\Omega_{\epsilon}$. 
Moreover, again by the strong law of large numbers, for any $a >0$ satisfying $a \mathbb{E}(m_1) <1$, almost surely, $s_{\lfloor ak \rfloor} = \sum_{i=1}^{\lfloor ak \rfloor} m_i < k$, and thus (recalling that $s^{-1}(k)$ denotes the unique $j$ such that $s_{j-1} \le k < s_{j}$), $s^{-1}(k) > ak$, for all sufficiently large $k$. Hence, $\mathbb{P}\left(\Omega_{\epsilon}\right) \rightarrow 1$ as $\epsilon \rightarrow 0$, which completes the proof.
\end{proof}

\begin{proof}[Proof of Theorem \ref{slowvarm}]
In this proof, $C,C', C''$ will denote universal positive constants (only depending on $\alpha, \nu$) whose values might change from line to line. By using the observation $\frac{1}{(k+1)^{\alpha}} \le \int_{k}^{k+1}\frac{1}{z^{\alpha}}dz \le \frac{1}{k^{\alpha}}$ for any $k \ge 1$, one obtains
\begin{equation}\label{sl1}
Ck^{1-\alpha} \le \Phi_1(k) \le C'k^{1-\alpha}, \ k \ge 1.
\end{equation}
Moreover, $s_n \ge n$ for all $n \ge 1$.
Therefore, if $\nu < 1/(1-\alpha)$,
$$
\frac{\Phi_1(m_n)}{\log s_n} \le \frac{C'(\lfloor 1 + (\log n)^{\nu}\rfloor)^{1-\alpha}}{\log n} \rightarrow 0
$$
as $n \rightarrow \infty$. Thus, noting that $f$ satisfies the assumptions of Theorem \ref{genper} with $C_f=1$, by Theorem \ref{genper},
\begin{equation}\label{t1}
\text{A persistent hub emerges if } \ \nu < 1/(1-\alpha).
\end{equation}
Now we consider the case $\nu \ge 1/(1-\alpha)$. For $k \ge 0$, recall that $d_{max}(k)$ denotes the maximal degree in $\mathcal{G}^*_{k}$ defined in \eqref{maxdegdef}, and $\mathcal{I}^*_k$ denotes the index of the maximal degree vertex in $\mathcal{G}^*_{k}$ defined in \eqref{inddef}. 
Let $\{X^*_k, k \ge 1\}$ be independent random variables such that for any $n \ge 1$ and any $s_{n-1} < k \le s_n$, $X^*_k \sim \operatorname{Bern}(n^{-1})$. Define the process $\{d^*(k) : k \ge 0\}$ by $d^*(0) = 0$ and $d^*(k) := \sum_{j=1}^kX^*_j, \ k \ge 1$.
As for any $n \ge 1$ and any $s_{n-1} \le k < s_n$,
$$
\mathbb{P}\left(d_{max}(k+1) = d_{max}(k) + 1 \ \vert \ \mathcal{G}^*_j, \ j \le k\right) \ge \frac{(d_{max}(k) + 1)^{\alpha}}{\sum_{j=0}^{n-1}(d_j(k) + 1)^{\alpha}} \ge \frac{1}{n},
$$
by the same argument as that in Lemma \ref{coupling}, there exists a coupling such that $d_{max}(k) \ge d^*(k)$ for all $k \ge 0$. Set $r(0) = \sigma^{*2}(0) =0$, and for any $n \ge 1$ and any $s_{n-1} <k \le s_n$, define
$$
r(k) := \sum_{l=1}^{n-1}\frac{m_l}{l} + \frac{k-s_{n-1}}{n}
$$
and
$$
\sigma^{*2}(k) := \sum_{l=1}^{n-1}\frac{m_l}{l}\left(1 - \frac{1}{l}\right) + \frac{k-s_{n-1}}{n}\left(1 - \frac{1}{n}\right)
$$
(with the sums in the above taken to be zero if $n=1$). Note that, for $k \ge 1$, $\mathbb{E}(d^*(k)) = r(k)$ and $\operatorname{Var}(d^*(k)) = \sigma^{*2}(k)$. As $0 \le X^*(k) \le 1$ for $k \ge 1$, and $\sigma^{*2}(k) \rightarrow \infty$ as $k \rightarrow \infty$, writing $L(k) :=\max\{1, \log \log k\}, k \ge 1$, it follows by Kolmogorov's law of iterated logarithm \cite{kolmogorov1929uber} that
$$
\limsup_{k \rightarrow \infty} \frac{|d^*(k) - r(k)|}{(2\sigma^{*2}(k)L(\sigma^{*2}(k)))^{1/2}} =1, \ \text{almost surely}.
$$
Further, as $r(k) \ge \sigma^{*2}(k)$ for all $k \ge 0$, the above implies
\begin{equation}\label{sl2}
\frac{d^*(k)}{r(k)} \rightarrow 1, \ \text{ almost surely as } \ k \rightarrow \infty.
\end{equation}
Note that, for $l \ge 2$,
$$
\frac{m_l}{l} \ge \int_{l}^{l+1} \frac{\lfloor 1 + (\log l)^{\nu}\rfloor}{z} dz \ge \int_{l}^{l+1} \frac{(\log l)^{\nu}}{z} dz \ge \int_{l}^{l+1} \frac{(\log (z/2))^{\nu}}{z} dz
$$
and hence, there exists a universal constant $C_r>0$ such that for $n \ge 2$,
\begin{equation}\label{sl3}
r(s_n) \ge \int_2^{n+1} \frac{(\log (z/2))^{\nu}}{z} dz \ge C_r(\log n)^{\nu+1}.
\end{equation}
Define the following events (similarly as in the proof of Theorem \ref{genper}) for $m, n \ge 1$:
$$
\tilde{A}_m := \left \lbrace d_{max}(s_k) \ge \frac{C_r}{2} (\log k)^{\nu+1} \text{ for all } k \ge m  \right \rbrace
$$
and
$$
\tilde{E}_n :=\{d_n(k) \ge d_{\mathcal{I}^*_{s_n}}(k) \text{ for some } k \ge s_{n}\}.
$$
By the coupling of $d_{max}(\cdot)$ and $d^*(\cdot)$, \eqref{sl2} and \eqref{sl3},
\begin{equation}\label{sl4}
\mathbb{P}(\tilde{A}_m) \rightarrow 1 \ \text{ as } m \rightarrow \infty.
\end{equation}
Independently of $\{\mathcal{G}^*_k: k \ge 0\}$, sample independent point processes $\{\xi_A(\cdot) :A \in \mathbb{N}_0\}$, $\xi_{A}(\cdot)$ obtained using attachment function $f_A(\cdot)$, and define the martingale $M_A(\cdot)$ for each $A \in \mathbb{N}_0$ as in \eqref{martdef}. By the same argument used to deduce \eqref{gp3}, for $n \ge m \ge 1$,
\begin{align}\label{sl5}
\mathbb{P}\left(\tilde{E}_n \cap \tilde{A}_m\right) &\le \mathbb{P}\left(\inf_{s<\infty} M_{d_{max}(s_n)}(s) \le -\frac{\Phi_1(d_{max}(s_n)) - \Phi_1(m_n)}{2}, \ d_{max}(s_n) \ge \frac{C_r}{2} (\log n)^{\nu+1}\right)\\
& \quad + \mathbb{P}\left(\sup_{s<\infty} M_{m_n}(s) \ge \frac{\Phi_1(d_{max}(s_n)) - \Phi_1(m_n)}{2}, \ \d_{max}(s_n) \ge \frac{C_r}{2} (\log n)^{\nu+1}\right).\nonumber
\end{align}
If $d_{max}(s_n) \ge \frac{C_r}{2} (\log n)^{\nu+1}$, then using \eqref{sl1},
$$
\Phi_1(d_{max}(s_n)) \ge \Phi_1\left(\frac{C_r}{2} (\log n)^{\nu+1}\right) \ge  C\left(\log n\right)^{(\nu+1)(1-\alpha)}.
$$
As, by \eqref{sl1}, $\Phi_1(m_n) \le C (\log n)^{\nu (1-\alpha)}$ for $n \ge 1$, there exists $n_0 \in \mathbb{N}_0$ such that for all $m \ge 1$ and all $n \ge \max\{n_0,m\}$,
$$
\Phi_1(d_{max}(s_n)) - \Phi_1(m_n) \ge C'\left(\log n\right)^{(\nu+1)(1-\alpha)}.
$$
Assume $n_0$ is large enough so that the lower bound above satisfies $C'\left(\log n\right)^{(\nu+1)(1-\alpha)} \ge \max\{x_0, x_0''\}$, where $x_0$ and $x_0''$ are constants defined in Lemma \ref{conc}.
Thus, for all $m \ge 1$ and all $n \ge \max\{n_0,m\}$, along the same lines as \eqref{slu1},
\begin{multline}\label{sl6}
\mathbb{P}\left(\inf_{s<\infty} M_{d_{max}(s_n)}(s) \le -\frac{\Phi_1(d_{max}(s_n)) - \Phi_1(m_n)}{2}, \ d_{max}(s_n) \ge \frac{C_r}{2} (\log n)^{\nu+1}\right)\\
\le C \exp\left\lbrace -C'(\log n)^{2(\nu+1)(1-\alpha)} \right\rbrace + C \exp\left\lbrace -C'\sqrt{f_*\left(\frac{C_r}{2} (\log n)^{\nu+1}\right)}(\log n)^{(\nu+1)(1-\alpha)} \right\rbrace\\
\le C \exp\left\lbrace -C'(\log n)^{2(\nu+1)(1-\alpha)} \right\rbrace + C \exp\left\lbrace -C'\sqrt{\left(\frac{C_r}{2} (\log n)^{\nu+1}\right)^{\alpha}}(\log n)^{(\nu+1)(1-\alpha)} \right\rbrace\\
\le C \exp\left\lbrace -C'(\log n)^{2(\nu+1)(1-\alpha)} \right\rbrace + C \exp\left\lbrace -C''(\log n)^{(\nu+1)\left(1-\frac{\alpha}{2}\right)} \right\rbrace.
\end{multline}
Moreover, for all $m \ge 1$ and all $n \ge \max\{n_0,m\}$, along the same lines as \eqref{slu2},
\begin{multline}\label{sl7}
\mathbb{P}\left(\sup_{s<\infty} M_{m_n}(s) \ge \frac{\Phi_1(d_{max}(s_n)) - \Phi_1(m_n)}{2}, \ d_{max}(s_n) \ge \frac{C_r}{2} (\log n)^{\nu+1}\right)\\
\le C \exp\left\lbrace -C'(\log n)^{2(\nu+1)(1-\alpha)} \right\rbrace.
\end{multline}
Using \eqref{sl6} and \eqref{sl7} in \eqref{sl5}, we obtain for all $m \ge 1$ and all $n \ge \max\{n_0,m\}$,
\begin{equation}\label{sl8}
\mathbb{P}\left(\tilde{E}_n \cap \tilde{A}_m\right) \le C \exp\left\lbrace -C'(\log n)^{2(\nu+1)(1-\alpha)} \right\rbrace + C \exp\left\lbrace -C''(\log n)^{(\nu+1)\left(1-\frac{\alpha}{2}\right)} \right\rbrace.
\end{equation}
Observe that for $\alpha \in (1/2, 1)$ and $\nu \ge 1/(1-\alpha)$,
$$
2(\nu+1)(1-\alpha) \ge 2\frac{2-\alpha}{1-\alpha} (1-\alpha) = 2(2-\alpha) > 2
$$
and
$$
(\nu+1)\left(1-\frac{\alpha}{2}\right) \ge \frac{(2-\alpha)^2}{2(1-\alpha)} >  \frac{(2- 1)^2}{2(1-0.5)} = 1.
$$
Hence, the right hand side of \eqref{sl8} is summable in $n$ for any $\nu \ge 1/(1-\alpha)$. Thus, by \eqref{sl4}, the Borel-Cantelli Lemma and Corollary \ref{distlim},
\begin{equation}\label{t2}
\text{A persistent hub emerges if } \ \nu \ge 1/(1-\alpha).
\end{equation}
The theorem now follows from \eqref{t1} and \eqref{t2}.
\end{proof}

\begin{proof}[Proof of Theorem \ref{pertree}]
Recall the continuous time embedding of $\{\mathcal{G}_n : n \ge 1\}$ into a (time-inhomogeneous) branching process $\BP(\cdot)$ given in Lemma \ref{lem:ctb-embedding-no-cp} and recall the stopping time $T_n:=\inf\{t\geq 0: |\BP(t)| =n+1\}, n \ge 0$. Fix $\delta \in (0, 1/3)$. Define the events
$$
A'_m := \left\lbrace \frac{\Phi_1(d_0(k))}{\log k} \ge \frac{1-2\delta}{\lambda_*} \text{ for all } k \ge m\right\rbrace, \ m \ge 1,
$$
and
$$
E_n := \{d_{n}(i) \ge d_0(i) \text{ for some } i \ge n\}, \ n \ge 1.
$$
By Lemma \ref{conc},
\begin{equation}\label{treeper1}
\mathbb{P}\left(\inf_{k \ge m} (\Phi_1(\xi(T_k)) - T_k) \le -\frac{\delta}{\lambda_*}\log (m+1)\right) \le \mathbb{P}\left(\inf_{s < \infty} M_s \le -\frac{\delta}{\lambda_*}\log (m+1)\right) \rightarrow 0
\end{equation}
as $m \rightarrow \infty$.

Note that $\Phi_2(\infty) < \infty$ implies $f(k) \rightarrow \infty$ as $k \rightarrow \infty$. Hence, by the same argument sketched at the beginning of the proof of Lemma \ref{birthcontrol},
$
e^{-\lambda_*t} |\BP(t)|
$
converges almost surely to a finite random variable $W$ as $t \rightarrow \infty$, and $W > 0$ almost surely. Hence, for $m \ge 0$,
\begin{align}\label{treeper2}
&\mathbb{P}\left(T_n \le \frac{1-\delta}{\lambda_*} \log (n+1) \text{ for some } n \ge m\right)\\
&\quad= \mathbb{P}\left(e^{-\lambda_*T_n} |\BP(T_n)| \ge (n+1)^{\delta} \text{ for some } n \ge m \right)\notag\\
&\quad \le \mathbb{P}\left(e^{-\lambda_*T_n} |\BP(T_n)| \ge (m+1)^{\delta} \text{ for some } n \ge m \right) \rightarrow 0\notag
\end{align}
as $m \rightarrow \infty$. Recalling that \{$d_0(n) : n \ge 0\}$ has the same distribution as $\{\xi(T_n) : n \ge 0\}$, we obtain using \eqref{treeper1} and \eqref{treeper2},
\begin{align}
\mathbb{P}((A'_m)^c) &\le \mathbb{P}\left(\Phi_1(d_0(k)) < \frac{1-2\delta}{\lambda_*} \log (k+1) \text{ for some } k \ge m\right)\\
&\le  \mathbb{P}\left(\Phi_1(\xi(T_k)) - T_k < -\frac{\delta}{\lambda_*} \log (k+1) \text{ for some } k \ge m\right)\notag\\
&\quad + \mathbb{P}\left(T_k \le \frac{1-\delta}{\lambda_*} \log (k+1) \text{ for some } k \ge m\right)\notag\\
&\le \mathbb{P}\left(\inf_{k \ge m} (\Phi_1(\xi(T_k)) - T_k) \le -\frac{\delta}{\lambda_*}\log (m+1)\right)\notag\\
&\quad + \mathbb{P}\left(T_k \le \frac{1-\delta}{\lambda_*} \log (k+1) \text{ for some } k \ge m\right) \notag\rightarrow 0\notag
\end{align}
as $m \rightarrow \infty$. Hence,
\begin{equation}\label{treeper3}
\mathbb{P}(A'_m) \rightarrow 1 \ \text{ as } \ m \rightarrow \infty.
\end{equation}
Now, along the lines of \eqref{gp3} and \eqref{gp4}, there exist positive constants $C, C'$ and $n_0 \ge m$ such that for all $n \ge n_0$, $\Phi_1(1) < \frac{\delta}{\lambda_*}\log n$ and
\begin{align*}
\mathbb{P}(E_n \cap A'_m) &\le \mathbb{P}\left(\inf_{s<\infty} M_{d_0(n)}(s) \le -\frac{\Phi_1(d_0(n)) - \Phi_1(1)}{2}, \ \Phi_1(d_0(n)) \ge \frac{1-2\delta}{\lambda_*} \log n\right)\nonumber\\
& \quad + \mathbb{P}\left(\sup_{s<\infty} M_{1}(s) \ge \frac{\Phi_1(d_0(n)) - \Phi_1(1)}{2}, \ \Phi_1(d_0(n)) \ge \frac{1-2\delta}{\lambda_*} \log n\right)\\
&\le \mathbb{P}\left(\inf_{s<\infty} M_{d_0(n)}(s) \le -\frac{1-3\delta}{2\lambda_*} \log n, \ \Phi_1(d_0(n)) \ge \frac{1-2\delta}{\lambda_*} \log n\right)\\
&\quad + \mathbb{P}\left(\sup_{s<\infty} M_{1}(s) \ge \frac{1-3\delta}{2\lambda_*} \log n, \ \Phi_1(d_0(n)) \ge \frac{1-2\delta}{\lambda_*} \log n\right)\\
&\le C\exp\left\lbrace-C' (\log n)^2\right\rbrace  + C\exp\left\lbrace -C'\sqrt{f_*\left(\left\lfloor\Phi_1^{-1}\left(\frac{1-2\delta}{\lambda_*}\log n\right)\right\rfloor\right)} (\log n)\right\rbrace
\end{align*}
which is summable. Using this observation, along with \eqref{treeper3}, the Borel-Cantelli Lemma and Corollary \ref{distlim}, we conclude that a persistent hub emerges almost surely. If $f$ satisfies condition (i) or (ii) stated in the theorem,  then Lemma \ref{checkcond} implies that Assumption \eqref{eqn:prop-under-lamb} holds and hence, a persistent hub emerges almost surely.

To prove \eqref{maxasper}, recall the quantities $B^{(i)}, \xi^{(i)}(\cdot), D^{(i)}(\cdot)$, $i \in \mathbb{N}_0$, defined just before Lemma \ref{MDPsmall}. Note that, for any $i \in \mathbb{N}_0$, by Lemma \ref{ac}, there exists a random variable $X^{(i)}$ (whose distribution is absolutely continuous with respect to Lebesgue measure) such that
\begin{equation}\label{maxasper1}
\Phi_1\left(\xi^{(i)}(T_n - B^{(i)})\right) - (T_n - B^{(i)}) \xrightarrow{a.s} X^{(i)} \ \text{ as } n \rightarrow \infty.
\end{equation}
Moreover, as $
e^{-\lambda_*t} |\BP(t)|
$
converges almost surely to a finite positive random variable $W$ (see proof of Lemma \ref{birthcontrol}),
\begin{equation}\label{maxasper2}
T_n - \frac{1}{\lambda_*}\log n \xrightarrow{a.s} -\frac{1}{\lambda_*}\log W \ \text{ as } n \rightarrow \infty.
\end{equation}
Combining \eqref{maxasper1} and \eqref{maxasper2}, we obtain
\begin{equation}\label{maxasper3}
\Phi_1\left(\xi^{(i)}(T_n - B^{(i)})\right) -  \frac{1}{\lambda_*}\log n \xrightarrow{a.s} Y^{(i)} \ \text{ as } n \rightarrow \infty,
\end{equation}
where $Y^{(i)} := X^{(i)} - B^{(i)} -\frac{1}{\lambda_*}\log W$.

Define
\begin{equation*}
\mathcal{N}^*_c(t) := \inf\{i : D^{(i)}(t-B^{(i)}) = D^{max}_{0,t}(t)\},
\end{equation*}
namely, the $\mathcal{N}^*_c(t)$-th individual is the first individual born into the branching process whose $N$-degree at time $t$ equals the maximal degree at that time. The continuous time embedding (see Section \ref{ctbp}) implies that $\{\mathcal{N}^*_c(T_n) : n \in \mathbb{N}_0\}$ has the same distribution as $\{\mathcal{I}^*_n : n \in \mathbb{N}_0\}$. Hence, as a persistent hub emerges, 
\begin{equation}\label{maxasper4}
\mathcal{N}^*_c(T_n) \xrightarrow{a.s} \mathcal{N}^* \ \text{ as } n \rightarrow \infty,
\end{equation}
for some integer valued random variable $\mathcal{N}^*$. Thus, from \eqref{maxasper3} and \eqref{maxasper4},
\begin{equation}\label{maxasper5}
\Phi_1\left(\xi^{(\mathcal{N}^*_c(T_n))}(T_n - B^{(\mathcal{N}^*_c(T_n))})\right) -  \frac{1}{\lambda_*}\log n \xrightarrow{a.s} Y^{(\mathcal{N}^*)} \ \text{ as } n \rightarrow \infty.
\end{equation}
Finally, noting that $\{d_{max}(n) : n \in \mathbb{N}_0\}$ has the same distribution as $\{\xi^{(\mathcal{N}^*_c(T_n))}(T_n - B^{(\mathcal{N}^*_c(T_n))})) : n \in \mathbb{N}_0\}$, \eqref{maxasper} follows from \eqref{maxasper5}.
\end{proof}

\begin{proof}[Proof of Theorem \ref{indextree}]
The proof adapts the approach of \cite[Theorem 1.15]{DM} with Lemmas \ref{MDPsmall} and \ref{indmax} replacing the use of Lemma 5.4 and Proposition 5.1 there, although additional technical challenges arise in the continuous time setting. Recall the quantities $B^{(i)}, D^{(i)}(\cdot), D^{max}_{a,b}(c)$, $i \in \mathbb{N}_0$, $0 \le a < b <c$, defined just before Lemma \ref{MDPsmall}. Define the continuous time analogue of the index of the maximal degree vertex as
\begin{equation*}
\mathcal{I}^*_c(t) := \inf\{B^{(i)} : D^{(i)}(t-B^{(i)}) = D^{max}_{0,t}(t)\},
\end{equation*}
that is, the birth time of the first individual born into the branching process whose $N$-degree at time $t$ equals the maximal degree at that time.

Consider the unimodal function
$$
H(u) = - u + \sqrt{2\lambda_* u}, \ u \ge 0.
$$
This function has a unique maximum at $u = u^* := \lambda_*/2$ with maximum value $H(u^*) = \lambda_*/2 = u^*$. Recall $A_0$ from Lemma \ref{indmax} and take $A_1 := \lambda_* + A_0$. 

Let $\epsilon_0 \in (0, \lambda_*/4)$ be small enough that $\max\{ H(u^*) - H(u^* - \epsilon), H(u^*) - H(u^* + \epsilon)\} < \sqrt{2\lambda_*(u^*-\epsilon)}$ for all $\epsilon \in (0, \epsilon_0)$. Fix any $\epsilon \in (0, \epsilon_0)$. Let $\zeta := \min\{ H(u^*) - H(u^* - \epsilon), H(u^*) - H(u^* + \epsilon)\}$. Partition the set $[0, u^* - \epsilon) \cup [u^* + \epsilon, A_1)$ into disjoint intervals $[u_i,v_i)$, $i \in \mathbb{J}$, of mesh $\min\{\zeta/3, u^*-\epsilon\}$ (with the interval containing zero having length equal to the mesh). Let $p = \min\{1, \zeta/12\}$. Then
\begin{align*}
\mathbb{P}&\left(\mathcal{I}^*_c(s) \in [(u^*-\epsilon)\K(t), (u^* + \epsilon)\K(t)] \text{ for all } s \in [t-p\K(t), t + p \K(t)]\right)\\
&\quad \ge \mathbb{P}\left(D^{max}_{(u^* - \epsilon)\K(t), u^*\K(t)}(s) \ge t + (H(u^*) - \zeta/3)\K(t),\right.\\
&\qquad \qquad \left. D^{max}_{u_i\K(t), v_i\K(t)}(s) \le t + (H(v_i) + \zeta/2)\K(t) \text{ for all } i \in \mathbb{J},\right.\\
&\qquad \qquad \left. D^{max}_{A_1\K(t), t}(s) \le t, \text{ for all } s \in [t-p\K(t), t + p \K(t)]\right)\\
&\quad \ge \mathbb{P}\left(D^{max}_{(u^* - \epsilon)\K(t), u^*\K(t)}(t - p\K(t)) \ge t + (H(u^*) - \zeta/3)\K(t),\right.\\
&\qquad \qquad \left. D^{max}_{u_i\K(t), v_i\K(t)}(t + p\K(t)) \le t + (H(v_i) + \zeta/2)\K(t) \text{ for all } i \in \mathbb{J},\right.\\
&\qquad \qquad \left. D^{max}_{A_1\K(t), t}(t + p\K(t)) \le t\right)\\
& \quad = \mathbb{P}\left(D^{max}_{(u^* - \epsilon)\K(t), u^*\K(t)}(t - p\K(t)) \ge t + (H(u^*) - \zeta/3)\K(t)\right) \\
& \qquad \times \prod_{i \in \mathbb{J}} \mathbb{P}\left(D^{max}_{u_i\K(t), v_i\K(t)}(t + p\K(t)) \le t + (H(v_i) + \zeta/2)\K(t)\right)\\
& \qquad \times \mathbb{P}\left(D^{max}_{A_1\K(t), t}(t + p\K(t)) \le t \right)\\
& \quad \ge \mathbb{P}\left(D^{max}_{(u^* - \epsilon)\K(t), u^*\K(t)}(t - (\zeta/6)\K(t)) \ge t + (H(u^*) - \zeta/3)\K(t)\right) \\
& \qquad \times \prod_{i \in \mathbb{J}} \mathbb{P}\left(D^{max}_{u_i\K(t), v_i\K(t)}(t + (\zeta/12)\K(t)) \le t + (H(v_i) + \zeta/2)\K(t)\right)\\
& \qquad \times \mathbb{P}\left(D^{max}_{A_1\K(t), t}(t + \K(t)) \le t \right).
\end{align*}
The first two terms in the product converge to one as $t \rightarrow \infty$ on respectively taking $\eta = \zeta/6, a=u^*-\epsilon, b=u^*$ and $\eta =  \zeta/12, a=u_i, b=v_i$ for each $i \in \mathbb{J}$ in Lemma \ref{MDPsmall}. Note that our choice of $\epsilon_0$ and the length of the interval containing zero in the partition ensure that $\eta \in (0, \sqrt{2\lambda_*b}/2)$ (for corresponding $b$) in each such application of Lemma \ref{MDPsmall}. The last term converges to one as $t \rightarrow \infty$ by Lemma \ref{indmax} upon noting that $A_1 > A_0$. Hence, we obtain
\begin{equation}\label{indextree1}
\mathbb{P}\left(\mathcal{I}^*_c(s) \in [(u^*-\epsilon)\K(t), (u^* + \epsilon)\K(t)] \text{ for all } s \in [t-p\K(t), t + p \K(t)]\right) \rightarrow 1,  \text{ as } t \rightarrow \infty.
\end{equation}
As $
e^{-\lambda_*t} |\BP(t)|
$
converges almost surely to a finite positive random variable $W$ (see proof of Lemma \ref{birthcontrol}), for any $\delta>0$,
\begin{equation}\label{indextree2}
\mathbb{P}\left(T_n \in \left[\frac{1}{\lambda_*} \log n - \delta\K\left(\frac{1}{\lambda_*} \log n\right), \frac{1}{\lambda_*} \log n + \delta\K\left(\frac{1}{\lambda_*} \log n\right)\right]\right) \rightarrow 1, \text{ as } n \rightarrow \infty.
\end{equation}
Therefore,
\begin{align}\label{indextree3}
&\mathbb{P}\left(\mathcal{I}^*_c(T_n) \in \left[(u^*-\epsilon)\K\left(\frac{1}{\lambda_*} \log n\right), (u^* + \epsilon)\K\left(\frac{1}{\lambda_*} \log n\right)\right]\right)\\
&\ge \mathbb{P}\left(\mathcal{I}^*_c(s) \in \left[(u^*-\epsilon)\K\left(\frac{1}{\lambda_*} \log n\right), (u^* + \epsilon)\K\left(\frac{1}{\lambda_*} \log n\right)\right]\right.\nonumber\\
&\qquad\qquad \left. \text{ for all } s \in \left[\frac{1}{\lambda_*} \log n - p\K\left(\frac{1}{\lambda_*} \log n\right), \frac{1}{\lambda_*} \log n + p\K\left(\frac{1}{\lambda_*} \log n\right)\right]\right)\nonumber\\
&\quad - \mathbb{P}\left(T_n \notin \left[\frac{1}{\lambda_*} \log n - p\K\left(\frac{1}{\lambda_*} \log n\right), \frac{1}{\lambda_*} \log n + p\K\left(\frac{1}{\lambda_*} \log n\right)\right]\right) \rightarrow 1\nonumber
\end{align}
as $n \rightarrow \infty$ upon taking $t = \frac{1}{\lambda_*} \log n$ in \eqref{indextree1} and using \eqref{indextree2} with $\delta = p$. Finally,
\begin{align*}
\mathbb{P}&\left(\log \mathcal{I}^*_n \notin \left[\lambda_* (u^*- 2\epsilon)\K\left(\frac{1}{\lambda_*} \log n\right), \lambda_*(u^* + 2\epsilon)\K\left(\frac{1}{\lambda_*} \log n\right)\right]\right)\\
&=\mathbb{P}\left(\mathcal{I}^*_n \notin \left[\exp\left\lbrace \lambda_* (u^*- 2\epsilon)\K\left(\frac{1}{\lambda_*} \log n\right)\right \rbrace, \exp\left\lbrace \lambda_*(u^* + 2\epsilon)\K\left(\frac{1}{\lambda_*} \log n\right)\right \rbrace\right]\right)\\
&\le \mathbb{P}\left(\mathcal{I}^*_c(T_n) \notin \left[(u^*-\epsilon)\K\left(\frac{1}{\lambda_*} \log n\right), (u^* + \epsilon)\K\left(\frac{1}{\lambda_*} \log n\right)\right]\right)\\
&\qquad + \mathbb{P}\left(n\left[0, (u^*-\epsilon)\K\left(\frac{1}{\lambda_*} \log n\right)\right] \le \exp\left\lbrace \lambda_* (u^*- 2\epsilon)\K\left(\frac{1}{\lambda_*} \log n\right)\right \rbrace\right)\\
&\qquad + \mathbb{P}\left(n\left[0,(u^* + \epsilon)\K\left(\frac{1}{\lambda_*} \log n\right)\right] \ge \exp\left\lbrace \lambda_* (u^* + 2\epsilon)\K\left(\frac{1}{\lambda_*} \log n\right)\right \rbrace\right).
\end{align*}
The first term in the above bound converges to zero as $n \rightarrow \infty$ by \eqref{indextree3}. The second term converges to zero upon taking $\delta = \lambda_*\epsilon/(u^* - \epsilon), a=0, b=1$ and $t= (u^*-\epsilon)\K\left(\frac{1}{\lambda_*} \log n\right)$ in \eqref{bcb}. The third term converges to zero by \eqref{bca}.

As $\epsilon \in (0, \epsilon_0)$ is arbitrary and $u^* = \lambda_*/2$, this proves \eqref{indexas}.

To prove \eqref{maxas}, fix any $\delta \in (0, \lambda_*/4)$. By the continuity of $H(\cdot)$, we obtain $\epsilon \in (0, \epsilon_0 \wedge \delta)$ small enough such that $|H(u) - u^*| < \delta$ for all $u \in [u^* - \epsilon, u^* + \epsilon]$. For $n \ge 1$,
\begin{align}\label{maxas1}
&\mathbb{P}\left(\Phi_1(d_{max}(n))- \frac{1}{\lambda_*}\log n - \frac{\lambda_*}{2}\K\left(\frac{1}{\lambda_*} \log n\right) > 8\delta \K\left(\frac{1}{\lambda_*} \log n\right)\right)\\
&\le \mathbb{P}\left(\mathcal{I}^*_c(T_n) \notin \left[(u^*-\epsilon/2)\K\left(\frac{1}{\lambda_*} \log n\right), (u^* + \epsilon/2)\K\left(\frac{1}{\lambda_*} \log n\right)\right]\right)\notag\\
&\quad + \mathbb{P}\left(D^{max}_{(u^* - \epsilon/2)\K\left(\frac{1}{\lambda_*} \log n\right), (u^* + \epsilon/2)\K\left(\frac{1}{\lambda_*} \log n\right)}(T_n) > \frac{1}{\lambda_*}\log n + \left(\frac{\lambda_*}{2} + 8\delta\right)\K\left(\frac{1}{\lambda_*} \log n\right)\right)\notag\\
&\le \mathbb{P}\left(\mathcal{I}^*_c(T_n) \notin \left[(u^*-\epsilon/2)\K\left(\frac{1}{\lambda_*} \log n\right), (u^* + \epsilon/2)\K\left(\frac{1}{\lambda_*} \log n\right)\right]\right)\notag\\
&\quad + \mathbb{P}\left(D^{max}_{(u^* - \epsilon/2)\K\left(\frac{1}{\lambda_*} \log n\right), (u^* + \epsilon/2)\K\left(\frac{1}{\lambda_*} \log n\right)}\left(\frac{1}{\lambda_*}\log n + \delta \K\left(\frac{1}{\lambda_*} \log n\right)\right)\right.\notag\\
&\qquad \qquad \qquad \qquad \qquad \qquad \qquad \qquad \left. > \frac{1}{\lambda_*}\log n + \left(\frac{\lambda_*}{2} + 8\delta\right)\K\left(\frac{1}{\lambda_*} \log n\right)\right)\notag\\
&\quad + \mathbb{P}\left(T_n > \frac{1}{\lambda_*}\log n + \delta \K\left(\frac{1}{\lambda_*} \log n\right)\right).\notag
\end{align}
The first and third terms in the above bound converge to zero as $n \rightarrow \infty$ using \eqref{indextree3} and \eqref{indextree2} respectively. To estimate the second term,
write $\alpha_n := \frac{1}{\lambda_*}\log n + \delta \K\left(\frac{1}{\lambda_*} \log n\right)$. Recalling $\frac{\K(t)}{t} \rightarrow 0$ as $t \rightarrow \infty$ and using Assumption C2, we obtain $n' \in \mathbb{N}$ (depending on $\delta, \epsilon$) such that for all $n \ge n'$,
$
(u^*-\epsilon)\K(\alpha_n) \le (u^* - \epsilon/2) \K\left(\frac{1}{\lambda_*} \log n\right),
$
$
(u^*+ \epsilon)\K(\alpha_n) \ge (u^* + \epsilon/2) \K\left(\frac{1}{\lambda_*} \log n\right),
$
$
(3\delta + 2\epsilon)\K(\alpha_n) \le (4\delta + 2\epsilon) \K\left(\frac{1}{\lambda_*} \log n\right)
$
and
$
u^*\K(\alpha_n) \le (u^* + \delta) \K\left(\frac{1}{\lambda_*} \log n\right).
$
Hence,
\begin{align*}
\frac{1}{\lambda_*}\log n + \left(\frac{\lambda_*}{2} + 8\delta\right)\K\left(\frac{1}{\lambda_*} \log n\right) &\ge \alpha_n + (u^* + \delta) \K\left(\frac{1}{\lambda_*} \log n\right) + (4\delta + 2\epsilon)\K\left(\frac{1}{\lambda_*} \log n\right)\\ 
&\ge \alpha_n + u^*\K(\alpha_n) + (3\delta + 2\epsilon)\K(\alpha_n)\\
&\ge \alpha_n + H(u^* + \epsilon) \K(\alpha_n) + 2(\delta + \epsilon) \K(\alpha_n)\\
&= \alpha_n - (u^* - \epsilon)\K(\alpha_n) + \left(\sqrt{2\lambda_*(u^* + \epsilon)} + 2\delta\right)\K(\alpha_n).
\end{align*}
Thus, taking $a = u^*-\epsilon, b = u^* + \epsilon$ and $\eta = \delta$ in Lemma \ref{MDPsmall}, we obtain for $n \ge n'$,
\begin{multline*}
\mathbb{P}\left(D^{max}_{(u^* - \epsilon/2)\K\left(\frac{1}{\lambda_*} \log n\right), (u^* + \epsilon/2)\K\left(\frac{1}{\lambda_*} \log n\right)}\left(\alpha_n\right) > \frac{1}{\lambda_*}\log n + \left(\frac{\lambda_*}{2} + 8\delta\right)\K\left(\frac{1}{\lambda_*} \log n\right)\right)\\
\le \mathbb{P}\left(D^{max}_{(u^*-\epsilon)\K(\alpha_n), (u^*+ \epsilon)\K(\alpha_n)}\left(\alpha_n\right) > \alpha_n - (u^* - \epsilon)\K(\alpha_n) + \left(\sqrt{2\lambda_*(u^* + \epsilon)} + 2\delta\right)\K(\alpha_n)\right)\\
\rightarrow 0, \ \text{ as } n \rightarrow \infty.
\end{multline*}
Thus, by \eqref{maxas1}, for any $\delta \in (0, \lambda_*/4)$,
\begin{equation}\label{maxas2}
\mathbb{P}\left(\Phi_1(d_{max}(n))- \frac{1}{\lambda_*}\log n - \frac{\lambda_*}{2}\K\left(\frac{1}{\lambda_*} \log n\right) > 8\delta \K\left(\frac{1}{\lambda_*} \log n\right)\right) \rightarrow 0, \ \text{ as } n \rightarrow \infty.
\end{equation}
By a similar argument, for any $\delta \in (0, \lambda_*/4)$,
\begin{equation}\label{maxas3}
\mathbb{P}\left(\Phi_1(d_{max}(n))- \frac{1}{\lambda_*}\log n - \frac{\lambda_*}{2}\K\left(\frac{1}{\lambda_*} \log n\right) < -8\delta \K\left(\frac{1}{\lambda_*} \log n\right)\right) \rightarrow 0, \ \text{ as } n \rightarrow \infty.
\end{equation}
\eqref{maxas} follows from \eqref{maxas2} and \eqref{maxas3}.
\end{proof}

\begin{proof}[Proof of Theorem \ref{unifindex}]
The proof is only outlined as it follows the same approach as that of Theorem \ref{indextree}. Note that when $f(k) = 1$ for all $k \in \mathbb{N}_0$, $\Phi_1(t) = \Phi_2(t) = \K(t) = t$ and $N(t) = \xi(t)$ for all $t \ge 0$, where $\xi(\cdot)$ is a rate one Poisson process. From large deviation principles for empirical means of i.i.d. Exponential random variables (see \cite[Section 2.2, Exercise 2.2.23]{dembo2011large}), we obtain the following analogue of Lemma \ref{MDPN}:
\begin{equation}\label{uni1}
\lim_{t \rightarrow \infty}\frac{1}{t} \log \mathbb{P}\left(N((1 - x)t) \ge t\right) = x + \log (1-x), \ x \in [0,1).
\end{equation}
Moreover, the Malthusian rate $\lambda_*$ equals $1$ in this case and $e^{-t}|\BP(t)|$ converges almost surely to a positive finite random variable as \eqref{laplace} is satisfied in this case
(see proof of Proposition 51 in \citet{bhamidi2007universal}). Thus, Lemma \ref{birthcontrol} holds.

Using \eqref{uni1} and Lemma \ref{birthcontrol}, we obtain the following analogue of Lemma \ref{MDPsmall}. Define the function
$$
w(\theta) := \theta - (1+\theta) \log (1+\theta).
$$
Then for any $0 \le a < b <1$, there exist positive constants $\eta_0, C_0$ (depending only on $b$) such that for any $\eta \in (0, \eta_0)$,
\begin{align}\label{uni2}
&\mathbb{P}\left(D^{max}_{at,bt}\left((1- \eta)t\right) \ge \left[(1-b)\left(1 + w^{-1}\left(-\frac{b}{1-b}\right)\right) - C_0\eta\right]t\right) \rightarrow 1,\\
&\mathbb{P}\left(D^{max}_{at,bt}\left((1 + \eta)t\right) \le \left[(1-a)\left(1 + w^{-1}\left(-\frac{b}{1-a}\right)\right) + C_0\eta\right]t\right) \rightarrow 1,\notag
\end{align}
as $t \rightarrow \infty$. Also, using \eqref{uni1} and Lemma \ref{birthcontrol}, we obtain $a_0 \in (0,1)$, given by the unique solution to the equation $1+ x + \log(1-x) = 0, \ x \in (0,1)$, such that for any $a >a_0$, there exists $\eta_1>0$ (depending on $a$) such that
\begin{equation}\label{uni3}
\mathbb{P}\left(D^{max}_{at,t}\left((1+ \eta_1)t\right) > t\right) \rightarrow 0, \ \text{ as } t \rightarrow \infty,
\end{equation}
which is the analogue of Lemma \ref{indmax}. Thus, using the same argument as in the proof of Theorem \ref{indextree}, we conclude that if the function
$$
\Psi(u) := (1-u)\left(1 + w^{-1}\left(-\frac{u}{1-u}\right)\right), \ u \in [0,1),
$$
has a unique maximum on the interval $[0,1)$ at some point $\hat{u}$, then
$$
\frac{\log \mathcal{I}^*_n}{\log n} \  \xrightarrow{P} \ \hat{u}, \ \text{ as } n \rightarrow \infty.
$$
Elementary calculus can be used to conclude that $\Psi(\cdot)$ indeed has a unique maximum on $[0,1)$ at $\hat{u} = 1 - \frac{1}{2\ln 2}$, proving \eqref{uni}. Moreover, $H(\hat{u}) = 1/\ln 2$. This observation, combined with calculations similar to \eqref{maxas1}, show \eqref{unimax}.
\end{proof}

\section*{Acknowledgements}
Bhamidi was partially supported by NSF grants DMS-1613072, DMS-1606839 and ARO grant W911NF-17-1-0010. Banerjee was partially supported by a Junior Faculty Development Award made by UNC, Chapel Hill. We acknowledge valuable feedback of an associate editor and an anonymous referee which led to major improvements in the presentation of this article.

 \bibliographystyle{abbrvnat}
\bibliography{persistence,pref_change_bib,scaling}

\appendix

\section{Verifying Assumptions C1, C2 and C3 for Classes I and II in Remark \ref{2classes}}\label{appver}

\textbf{Class I: } As $f_r$ is positive and regularly varying with index $\alpha \in [0, 1/2)$, 
taking $\epsilon \in (0, \alpha)$ such that $2\alpha + \epsilon< 1$ and using \cite[Theorem 1.5.6]{bingham}, we obtain
$$
\lim_{k \rightarrow \infty}\frac{f_r(k)}{k^{2\alpha + \epsilon}} = 0.
$$
Recalling $f = f_rf_b \le b_2f_r$, we conclude $\Phi_2(\infty) = \infty$, and hence, Assumption C1 holds.

Define $\Phi_{i,r} (t) := \int_0^t \frac{1}{f^i_r(z)}dz, \ t \ge 0, i=1,2$, and $\K_r(\cdot) := \Phi_{2,r} \circ \Phi_{1,r}^{-1}(\cdot)$. Note that, using a simple $u$-substitution,
$$
\K(t) := \Phi_2 \circ \Phi_1^{-1}(t) = \int_0^t \frac{1}{\bar{f}(z)}dz
$$
where $\bar{f}(\cdot) := f \circ \Phi_1^{-1}(\cdot)$. Thus, for any $\delta>0, t \ge 0$, recalling $f \ge b_1 f_r$,
\begin{equation}\label{ap1.1}
\K((1+\delta)t) = \K(t) + \int_t^{(1+\delta)t} \frac{1}{\bar{f}(z)}dz \le \K(t) + \frac{1}{b_1}\int_t^{(1+\delta)t} \frac{1}{f_r\circ \Phi_1^{-1}(z)}dz.
\end{equation}
Further, note that
$
\frac{1}{b_2} \Phi_{1,r}(s) \le \Phi_1(s) \le \frac{1}{b_1} \Phi_{1,r}(s)
$
for $s \ge 0$, and hence, $\Phi^{-1}_{1,r}(b_1s) \le \Phi^{-1}_1(s) \le \Phi^{-1}_{1,r}(b_2s), s \ge 0$. Thus,
\begin{equation}\label{ap1.2}
1 \le \frac{\Phi^{-1}_1(s)}{\Phi^{-1}_{1,r}(b_1s)} \le  \frac{\Phi^{-1}_{1,r}(b_2s)}{\Phi^{-1}_{1,r}(b_1s)}, \ s \ge 0.
\end{equation}
By Theorem 1.5.12 of \citet{bingham}, $\Phi^{-1}_{1,r}(\cdot)$ is regularly varying with index $(1-\alpha)^{-1}$ and hence,
$$
\lim_{s \rightarrow \infty} \frac{\Phi^{-1}_{1,r}(b_2s)}{\Phi^{-1}_{1,r}(b_1s)} = \left(\frac{b_2}{b_1}\right)^{1/(1-\alpha)}.
$$ 
Using this in \eqref{ap1.2}, we obtain $s_0>0$ such that for all $s \ge s_0$,
\begin{equation}\label{ap1.3}
1 \le \frac{\Phi^{-1}_1(s)}{\Phi^{-1}_{1,r}(b_1s)} \le 2 \left(\frac{b_2}{b_1}\right)^{1/(1-\alpha)}.
\end{equation}
By Theorem 1.5.2 of  \cite{bingham}, there exists $s_1 \ge s_0$ such that for all $s \ge s_1$,
$$
\left| \frac{f_r(x\Phi^{-1}_{1,r}(b_1s))}{f_r(\Phi^{-1}_{1,r}(b_1s))} - x^{\alpha}\right| < \frac{x^{\alpha}}{2} \ \text{ for all } \ x \in \left[1, 2 \left(\frac{b_2}{b_1}\right)^{1/(1-\alpha)}\right].
$$
In particular, by \eqref{ap1.3}, for all $s \ge s_1$,
\begin{equation}\label{ap1.4}
\frac{1}{2} \le \frac{f_r(\Phi^{-1}_{1}(s))}{f_r(\Phi^{-1}_{1,r}(b_1s))} \le \frac{3}{2}\left[2\left(\frac{b_2}{b_1}\right)^{1/(1-\alpha)}\right]^{\alpha} =: b(\alpha).
\end{equation}
From \eqref{ap1.1} and \eqref{ap1.4}, for $t \ge s_1$,
\begin{align}\label{ap1.5}
\K((1+\delta)t) & \le  \K(t) + \frac{2}{b_1}\int_t^{(1+\delta)t} \frac{1}{f_r\circ \Phi_{1,r}^{-1}(b_1z)}dz\\
& =  \K(t) + \frac{2}{b_1^2}\int_{b_1t}^{b_1(1+\delta)t} \frac{1}{f_r\circ \Phi_{1,r}^{-1}(z)}dz = \K(t) + \frac{2\left(\K_r(b_1(1+\delta)t) - \K_r(b_1 t)\right)}{b_1^2}.\nonumber
\end{align}
Moreover, recalling $f \le b_2 f_r$ and again using \eqref{ap1.4}, for $t \ge s_1$,
\begin{equation}\label{ap1.6}
\K(t) \ge \frac{1}{b_2}\int_{s_1}^{t} \frac{1}{f_r\circ \Phi_1^{-1}(z)}dz \ge \frac{1}{b_2 b(\alpha)}\int_{s_1}^{t}\frac{1}{f_r\circ \Phi_{1,r}^{-1}(b_1z)}dz = \frac{\K_r(b_1 t) - \K_r(b_1s_1)}{b_1 b_2 b(\alpha)}.
\end{equation}
Hence, by \eqref{ap1.5} and \eqref{ap1.6},
\begin{equation}\label{ap1.7}
1 \le \limsup_{t \rightarrow \infty} \frac{\K((1+\delta)t)}{\K(t)} \le 1 + \frac{2b_2b(\alpha)}{b_1}\limsup_{t \rightarrow \infty}\left(\frac{\K_r(b_1(1+\delta)t)}{\K_r(b_1t)} - 1\right).
\end{equation}
By Proposition 1.5.8, Theorem 1.5.7 and Theorem 1.5.12 of \cite{bingham}, $\K_r(\cdot)$ is regularly varying with index $\theta_{\alpha} := \frac{1-2\alpha}{1-\alpha}$. Using this observation in \eqref{ap1.7}, we conclude that for any $\delta >0$,
\begin{equation}\label{ap1}
1 \le \limsup_{t \rightarrow \infty} \frac{\K((1+\delta)t)}{\K(t)} \le 1 + \frac{2b_2b(\alpha)}{b_1}\left((1+\delta)^{\theta_{\alpha}} - 1\right).
\end{equation}
Taking a limit as $\delta \rightarrow 0$ in \eqref{ap1} shows that Assumption C2 is satisfied. Taking $\delta = 2$ in \eqref{ap1} shows that Assumption C3 is satisfied with $D = 2\left[1 + \frac{2b_2b(\alpha)}{b_1}\left(3^{\theta_{\alpha}} - 1\right)\right]$ and some $t'>0$.\\

\textbf{Class II: } 
From the assumptions $\sum_{k=0}^{\infty}\frac{1}{g^2(k)} = \infty$ and $\lim_{k \rightarrow \infty} h(k)/g(k) = 0$, it readily follows that $\Phi_2(\infty) = \infty$, and hence, Assumption C1 holds.
To verify Assumptions C2 and C3, extend $g,h$ as $f$ to all of $[0, \infty)$.
Note that for any $\delta \ge 0$,
\begin{equation}\label{ap2.1}
\K((1+\delta)t) = \int_0^{\Phi^{-1}_1((1+\delta)t)}\frac{1}{f^2(z)}dz = (1+\delta)\int_0^t\frac{1}{f \circ \Phi^{-1}_1((1+\delta)u)}du
\end{equation}
where the last step follows by a $u$-substitution with $u= \Phi_1(z)/(1+\delta)$. 

Take any $\delta>0$. As $h(\cdot)$ is non-negative, $g(t) \le f(t)$ for all $t \ge 0$. Moreover, as $\lim_{t \rightarrow \infty} h(t)/g(t) = 0$ we obtain $t_{\delta}>0$ such that $f(t) \le (1+\delta)g(t)$ for all $t \ge t_{\delta}$. Hence, by \eqref{ap2.1}, for all $t \ge \Phi_1(t_{\delta})$,
\begin{align*}
\K((1+\delta)t) &\le (1+\delta)\int_0^t\frac{1}{g \circ \Phi^{-1}_1((1+\delta)u)}du \le (1+\delta)\int_0^t\frac{1}{g \circ \Phi^{-1}_1(u)}du\\
& \le (1+\delta)\int_0^{\Phi_1(t_{\delta})}\frac{1}{g \circ \Phi^{-1}_1(u)}du +  (1+\delta)^2\int_{\Phi_1(t_{\delta})}^t\frac{1}{f \circ \Phi^{-1}_1(u)}du\\
&\le (1+\delta)\int_0^{\Phi_1(t_{\delta})}\frac{1}{g \circ \Phi^{-1}_1(u)}du +  (1+\delta)^2\int_{0}^t\frac{1}{f \circ \Phi^{-1}_1(u)}du\\
&= (1+\delta)\int_0^{\Phi_1(t_{\delta})}\frac{1}{g \circ \Phi^{-1}_1(u)}du +  (1+\delta)^2\K(t).
\end{align*}
Hence, for any $\delta>0$, 
$$
\limsup_{t \rightarrow \infty} \frac{\K((1+\delta)t}{\K(t)} \le (1+\delta)^2.
$$
Assumptions C2 and C3 follow from this by respectively taking $\delta \rightarrow 0$ and $\delta = 2$.

%

\end{document}